\documentclass[11pt]{amsart}
\usepackage{graphicx}
\usepackage[usenames]{color}
\usepackage{amssymb,amscd}
\usepackage{comment}

\usepackage{subfiles}	
\usepackage{hyperref}
\hypersetup{
    colorlinks,
    citecolor=black,
    filecolor=black,
    linkcolor=black,
    urlcolor=black
}

\newcommand{\bea}{\begin{eqnarray*}}
\newcommand{\eea}{\end{eqnarray*}}

\newcommand{\bm}{\begin{pmatrix}}
\newcommand{\fm}{\end{pmatrix}}

\newcommand\Z{\mathbb Z}

\newcommand\R{\mathbb R}
\newcommand{\ra}{\rightarrow}

\newcommand{\G}{\Gamma}

\newtheorem{theorem}{Theorem}[section]
\newtheorem{proposition}[theorem]{Proposition}
\newtheorem{prop}[theorem]{Proposition}

\newtheorem{lemma}[theorem]{Lemma}
\newtheorem{remark}[theorem]{Remark}
\newtheorem{Le}[theorem]{Lemma}
\newtheorem{conjecture}[theorem]{Conjecture}

\newtheorem{corollary}[theorem]{Corollary}
\newtheorem{definition}[theorem]{Definition}

\title{A Tukia-type theorem   for nilpotent Lie groups and quasi-isometric rigidity of
    solvable groups\\ }
\author{Tullia Dymarz}
\address{Department of Mathematics,
University of Wisconsin-Madison, 480 Lincoln Drive,  Madison, WI 53706} \email{dymarz@math.wisc.edu}

\author{David Fisher}
\address{Department of Mathematics,
 Rice  University,  6100 Main Street
Houston, TX 77005}
   \email{davidfisher@rice.edu}

\author{Xiangdong Xie}
\address{Department of Mathematics and Statistics,
Bowling Green State University,
Bowling Green, OH 43403} \email{xiex@bgsu.edu}
\thanks{}

\begin{document}
\maketitle

\begin{abstract}
In this paper  we  study uniform quasiconformal groups of  Carnot-by-Carnot groups. We show that they can be conjugated into conformal groups  provided the induced action on the space of  distinct  pairs is cocompact. Following the approach of Eskin-Fisher-Whyte \cite{EFW07} these results have applications  to quasi-isometric rigidity of  
 certain families of solvable groups.
\end{abstract}




\tableofcontents

%
%
%
%
%

\section{Introduction}

{\bf Motivation.} The motivation for this paper comes from Gromov's program of classifying finitely generated groups up to \emph{quasi-isometric equivalence}. A map $f:X \to Y$ between metric spaces is a quasi-isometry if  there are $A \geq 0$ and $L\geq 1$ such that  for all $x_1, x_2 \in X$
 $$ -A + d(x_1,x_2)/L \leq  d(f(x_1), f(x_2))  \leq L\ d(x_1,x_2)+A$$
 and for all $y \in Y$ there is $x \in X$ with $d(f(x),y) \leq A$. 
 In \cite{EFW07} Eskin-Fisher-Whyte announced quasi-isometric rigidty for lattices in $SOL=\R^2 \rtimes \R$, the simplest non-nilpotent unimodular solvable Lie group and in \cite{EF10} they outlined a conjecture and program for proving a quasi-isometric classification for all non-nilpotent polycyclic groups (or equivalently all unimodular non-nilpotent solvable Lie groups $G$). Along with the quasi-isometric classification of nilpotent groups this remains one of the most important open problems in Gromov's Program. 

The strategy outlined in \cite{EF10} has two main steps:
\begin{enumerate}
\item Understand the structure self quasi-isometries of the solvable Lie group $G$. 
\item  Use this structure to identify certain groups of self quasi-isometries of $G$  as conjugates of groups of isometries of a possibly different Lie group $\bar{G}$. 
\end{enumerate}

In the case of SOL the second step is proved in a short appendix to a paper by Farb-Mosher \cite{FM99}. In many other cases  this second step requires detailed quasiconformal analysis on a homogeneous group $N$ resulting in a so called \emph{Tukia-type} theorem (see below).  Recall that a homogeneous group is a nilpotent Lie group together with a dilation that induces a homogenous metric on $N$. Previous work has been in the case when either $N$ is abelian \cite{D10} or in the {Carnot} case 
\cite{DyX16} which is the case where the dilation is completely determined by the algebra of $N$.
In this paper we present the first case of understanding this analysis for non-abelian non-Carnot homogeneous groups. We call this case the \emph{Carnot-by-Carnot} case (Definition \ref{carnotbycarnot}).
This case has significant new challenges that involve subtle interactions between the algebra of the homogeneous group (coming from $N$ being non abelian) and the geometry of the homogeneous metric (coming from $N$ being non-Carnot). For this we introduce new analysis techniques that we expect will provide the foundation to prove the necessary results on more general homogeneous  Lie groups.  Once complete this program will establish the quasi-isometric classification of all \emph{SOL-like} solvable Lie groups (Definition \ref{SOL-like}).

Our main theorems are Theorem \ref{main-uniform}  and Theorem \ref{theorem:QIrigid} below.
Theorem \ref{theorem:QIrigid} is quasi-isometric rigidity for our new family of SOL-like groups while Theorem \ref{main-uniform}, whose proof is the bulk of this paper,
provides Step (2) for Theorem \ref{theorem:QIrigid}. Concrete examples of both Carnot-by-Carnot groups and lattices in the accompanying SOL-like groups can be found in appendices A and B respectively.



{\bf Tukia-type Theorems.}  
     A Tukia-type theorem asserts that a  group of self maps of a metric space $(X,d)$ that distort distances by a uniformly controlled amount must actually be, up to conjugation, a lot more rigid. In this paper,
 our uniformly controlled maps will be \emph{quasi-similarities}, i.e.   bijections $f:(X,d) \to (X,d)$ such that
 \begin{equation}\label{qsim-eqn} K_f^- d(x,y) \leq d(f(x), f(y)) \leq  K_f^+ d(x,y) \quad \textrm{ for all } x,y \in X\end{equation}
 where $K_f^+/K_f^-\leq K$ and $K$ is a uniform  constant  over all group elements. A more common term for a map satisfying Equation   (\ref{qsim-eqn})
       is \emph{biLipschitz} but then  $K_f^-$ is usually defined to be the inverse of $K_f^+$ so uniformity in our sense is harder to describe.
  The goal conclusion will be that up to conjugation by a  biLipschitz map, the group must be acting by \emph{similarities}, that is maps where $K_f^-=K_f^+$.  Such a theorem is always true for example when $X= \R$ \cite{FM99} but in general, additional hypotheses on the size of the group are often needed.

 Tukia's original theorem \cite{T86} was proved for groups of uniform quasiconformal maps of the $n$-sphere for $n\geq 2$, showing that up to conjugation these groups act by conformal maps.  For $n=2$ this was also proved by Sullivan \cite{S78}. When $n\geq 3$ the additional hypothesis that was needed was that every point be a \emph{radial limit point} (see Definition \ref{radial:defn}) under the action. We will also need such a hypothesis in  our case.
 Tukia's theorem was first generalized in \cite{Ch96} to quasiconformal groups of the real Heisenberg group equipped with a Carnot-Caratheodory metric. A similar generalization holds for all Carnot groups endowed with a Carnot-Caratheodory metric (see  Theorem \ref{tukia}) but in the conclusion the conformal action may be with respect to a different Carnot-Caratheodory metric.
  Other generalizations of Tukia's theorem for uniform quasisimilarity groups may be found in \cite{DyX16} where $N=\R^n$ with homogeneous distances that are not the standard Euclidean distances.
      This paper contains the first Tukia-type theorem   for non-abelian non-Carnot homogeneous groups.

 \begin{definition}\label{carnotbycarnot} Let  $N$ be a simply connected nilpotent Lie group with Lie algebra  $\mathfrak n$.  Let $D$  be  a diagonalizable
derivation of  $\mathfrak n$   with  positive eigenvalues. Let $\mathfrak w$ be the Lie sub-algebra generated by the eigenspace of $D$ associated with the smallest eigenvalue.
  {We say that $(N,D)$ is a \emph{Carnot-by-Carnot} group if the following conditions are satisfied:
  \begin{enumerate}
   \item
  $\mathfrak w$ is a proper  ideal of $\mathfrak n$;
  \item   $\mathfrak n/\mathfrak w$ is a Carnot algebra;
  \item
   $D$ induces a derivation $\bar{D}$ on $\mathfrak n/\mathfrak w$ that is a multiple of the standard Carnot   derivation on $\mathfrak n/\mathfrak w$.
   \end{enumerate}  }
 \end{definition}

 We remark                      that     $\mathfrak w$   is itself a Carnot algebra, see Lemma \ref{derivation lemma}.
Carnot-by-Carnot groups  are abundant.
  The obvious examples are   given   by a direct product  of two Carnot groups (where the derivation is a product of   distinct multiples of the standard Carnot derivation    on the Lie algebras of the two factors).  More interesting examples are   provided  by central products and semi-direct products. Using Lie algebra cohomology one can also obtain Carnot-by-Carnot groups that are not semi-direct products.
  See  Appendix  \ref{nilexample} for more details.

  Borrowing  terminology from \cite{LDNG21}, we say   a
   distance   $d$  on $N$ is  \emph{$D$-homogeneous}    if it is left invariant,   induces the manifold topology on $N$ and such that
 $d(e^{tD}x, e^{tD}y)=e^t d(x, y)$ for all $x, y\in N$ and $t\in \mathbb R$, where $\{e^{tD}|t\in \mathbb R\}$  is the one parameter group of automorphisms  of $N$ generated by the derivation $D$.
 There are many $D$-homogeneous distances  $d$ on $N$ and   it is easy to see that any two
 $D$-homogeneous distances
  are pairwise biLipschitz equivalent.
  When we talk about a biLipschitz map $F$ of $N$  we mean that $F$ is biLipschitz  with respect to a  $D$-homogeneous distance on $N$. 
Note, however, that for two  $D$-homogeneous distances $d_1$, $d_2$ on $N$,
  a map $F: (N, d_1)\to (N, d_1)$ being a similarity does not imply that  $F: (N, d_2)\to (N, d_2)$  is a similarity.
     A   $D$-homogeneous distance $d_0$  on $N$ is   \emph{maximally symmetric} if for any
      $D$-homogeneous distances $d$ on $N$ there is a biLipschitz automorphism $\phi$ such that $\phi \text{Sim}(N,d)\phi^{-1}\subset  \text{Sim}(N,d_0)$, where
     $ \text{Sim}(N,d)$ denotes the group of similarities of $(N,d)$.
             We remark that  $N$ always admits  a    maximally symmetric $D$-homogeneous distance $d_0$, see Lemma 2.3, \cite{DFX}.

  \begin{theorem}\label{main-uniform}
Let  $(N,D)$ be a Carnot-by-Carnot group and  $\Gamma$ a
    group
with a  uniform quasisimilarity action of $N$  such that almost every point in $N$ is a radial   limit point.
  When $\dim(\mathfrak w)=1$, we further   assume  $\Gamma$ is  locally compact   amenable.
   Then there exists a biLipschitz map $F_0$ of $N$
      such  that   $F_0\Gamma F_0^{-1}\subset \text{Sim}(N,d_0)$,
        where $d_0$ is a fixed maximally symmetric $D$-homogeneous distance on $N$.
\end{theorem}

 We remark that  a version of Theorem \ref{main-uniform} for uniform quasiconformal groups is not needed since by \cite{CP17}  when
  $N$ is Carnot-by-Carnot,
each quasiconformal map of
 $N\cup \{\infty\}$ fixes $\infty$  and a  uniform quasiconformal group  of $N\cup \{\infty\}$  is the same thing as a uniform quasisimilarity group of $N$.
 Theorem  \ref{main-uniform} also holds for Carnot groups but in this case
  all that is required is running the proof of Theorem \ref{tukia}  
  with the additional assumption that all maps in question  fix $\infty$ and are biLipschitz.
  In a general homogeneous group there are two sources of metric distortion, one from brackets and one
from differing rates of contraction from the derivation $D$.  The Carnot case is exactly the one
where these two distortions agree.  A major difficulty in going beyond the Carnot
case [DFX] and the abelian case [Dy] comes when it is difficult to separate the two sources
of distortion.
The  case of general $(N,D)$ involves more phenomena
of this kind not already seen in the Carnot-by-Carnot case.
In the case of products of Carnot groups this is easier than in the general
Carnot-by-Carnot case. In fact
 we can easily prove a similar theorem for uniform quasisimilarity groups of $(N,D)$ with  $N=\prod_i N_i$  where  each $N_i$ is Carnot  with Lie algebra $\mathfrak n_i$ and $D\vert_{\mathfrak n_i}$ is an increasing multiple of a Carnot derivation (see Theorem \ref{product of Carnot}).


The amenability condition when $\dim(\mathfrak w)=1$   is needed in the proof in order to apply Day's fixed point theorem which states that affine actions of amenable groups on compact convex subsets of locally convex topological  vector spaces have fixed points.
We actually also use Day's fixed point theorem in  the case $\dim(\mathfrak w)\ge 2$
   but in that  case we are able to deduce the amenability of the acting group.
This extra amenability condition does not affect the application of our theorem to quasi-isometric rigidity since amenability of discrete groups is preserved by
quasi-isometries.   Nevertheless it would be interesting to know if this condition can be removed.
The proof of  Theorem \ref{main-uniform} takes up the bulk of the paper and so we  outline it in more detail in the next section.
%


{\bf Quasi-isometric rigidity of solvable groups.} The power of a Tukia-type theorem in geometric group theory is that it can often be used to prove quasi-isometric rigidity results.  This is done
through the notion of quasi-action.
A quasi-action of a group $\Gamma$ on a metric space $X$ is an assignment $ \gamma \mapsto G_\gamma$ where $G_\gamma$ is a self quasi-isometry of $X$ such that
\begin{enumerate}
\item  $G_\gamma$ is an $(L,A)$ quasi-isometry where $L$ and $A$ are uniform over all $\gamma \in \Gamma$.
\item  $G_{\gamma\eta}$ and  $G_\gamma \circ G_\eta$ are  bounded distance apart in the sup norm, uniformly over all $\gamma, \eta \in \Gamma$.
\item $G_{Id}$ is bounded distance from the identity map on $X$.
\end{enumerate}
A quasi-action is \emph{cobounded} if there is a bounded set $S \subset X$ such that for any $x \in X$ there is $\gamma \in \Gamma$ such that $ G_\gamma(x) \in S$.

The standard example of a cobounded quasi-action arises when $\Gamma$ is a  group with   a left invariant  metric (for example, a  finitely generated group  with a word metric or a Lie group with a left invariant Riemannian metric) and $\phi: \Gamma \to X$ is a quasi-isometry with coarse inverse $\bar{\phi}$. Then $\gamma \mapsto
\phi\circ L_\gamma \circ \bar{\phi}$   defines a cobounded quasi-action   of $\Gamma$ on $X$, where $L_\gamma$ is the left translation of $\Gamma$ by $\gamma$.



 In the case that $X$ is Gromov hyperbolic,  quasi-isometries of $X$ correspond to quasiconformal maps (quasisymmetric maps, to be more precise) of the ideal boundary $\partial X$,  and quasi-actions on $X$ correspond to uniform quasiconformal   actions on  $\partial X$.
   Furthermore, with suitable choices of metrics on $X$ and $\partial X$, isometries of $X$  often correspond  to conformal maps of  $\partial X$   or similarities  of a one-point complement of  $\partial X$.
 As a result,   a group $\Gamma$ quasi-acting on a Gromov hyperbolic space $X$ induces a
 uniform quasiconformal   action on  $\partial X$;  if
 a Tukia-type theorem   is available, then this uniform quasiconformal   action is conjugate by a quasiconformal map of $\partial X$ to a conformal action of $\Gamma$ on
 $\partial X$, which implies the original quasi-action is
   conjugate by a quasi-isometry  to an isometric  action on $X$.

   We now apply the outline of the proceeding paragraph in the case when $X$ is a so called \emph{Heintze group}:   Given a simply connected nilpotent Lie group $N$ with Lie algebra $\mathfrak n$ and  a derivation $D$  of $\mathfrak n$, one  can form the semi-direct product $S=N\rtimes_D \mathbb R$, where the action of $\mathbb R$ on $N$ is by the automorphisms
    $e^{tD}$ of $N$ generated by the derivation $D$.  Such a group $S$ is called a Heintze group  if  the eigenvalues of $D$ all have positive real parts.
   For example  when $(N,D)$ is  Carnot-by-Carnot the associated extension $S$ is a Heintze group.
     By \cite{H74},
       Heintze groups    are  Gromov hyperbolic.
    By applying our Tukia-type theorem we get the following (a similar statement holds when $N$ is  a  Carnot  group, see  Corollary 1.2 in \cite{DFX}):

   \begin{theorem} \label{NCHM}
  Let $(N,D)$ be Carnot-by-Carnot and $S=N\rtimes_D \R$  the associated Heintze group.
  Then there is a left invariant Riemannian metric $g_0$ on $S$ with the following property.
 If $\Gamma$ is a
      group that quasi-acts coboundedly  on  $S$ (and when $\dim(W)=1$, we further   assume  $\Gamma$ is locally compact amenable) then
   this  quasi-action is quasi-conjugate to an isometric action of $\Gamma$  on
     $(S, g_0)$.   
\end{theorem}


   More generally, Tukia-type theorems have been used in the proofs of quasi-isometric rigidity of  certain solvable Lie groups that are not themselves negatively curved but instead are foliated by negatively curved spaces. Most notably Eskin-Fisher-Whyte's quasi-isometric rigidity theorem for SOL \cite{EFW07} uses the fact that SOL is foliated by hyperbolic planes $\mathbb{H}^2$ whose visual boundary can be identified with $S^1 \simeq \R\cup \{\infty\}$. Recall that $SOL\simeq \R^2 \rtimes \R$ where the action of $\R$ on $\R^2$ is given by $e^{tA}$ where $A$ is the diagonal matrix with $1$ and $-1$ on the diagonal and a left invariant metric can be given by $ds^2= e^{-2t}dx^2 +
   e^{2t} dy^2 + dt^2$.
 There are two foliations by hyperbolic planes, given by  fixing either the $x$ or the $y$ coordinate.    As a consequence of Eskin-Fisher-Whyte's coarse differentiation theorem on the structure of quasi-isometries of SOL, a quasi-action on SOL  induces two quasi-actions on the hyperbolic plane (and hence two quasisimilarity actions on $\R$), one for each foliation.


\begin{definition}\label{SOL-like}
We call a group \emph{SOL-like} if  it has the form $(N_1 \times N_2) \rtimes \R$,      where
     $N_i$ is  a simply connected nilpotent Lie group  with Lie algebra $\mathfrak n_i$
         and  the action of $\R$ on $N_1 \times N_2$   is given by $e^{tD}$ with  $D$  the derivation of
   $\mathfrak n_1\times \mathfrak n_2$
  given by $D:=( D_1,-D_2 )$   and  $D_i$ a  derivation  of
     $\mathfrak n_i$  whose  eigenvalues  have positive real part.
  \end{definition}   
The negatively curved spaces foliating $(N_1 \times N_2) \rtimes \R$ are simply $N_i \rtimes \R$ where the action of $\R$ is defined by the derivation $D_i$.
 The visual boundary of $N_i \rtimes \R$ is simply $N_i \cup \{ \infty\}$ but in the context   of  quasi-isometries of SOL-like groups the point $\infty$ is always preserved. The structure of quasi-isometries of SOL-like groups follows from Eskin-Fisher-Whyte's coarse differentiation techniques \cite{EFW07,EF10} (see also \cite{Fe22}  for a
  detailed treatment of the ``non-unimodular''  case  in  a more general setting).    A consequence of  Eskin-Fisher-Whyte's argument  is that any quasi-action of a group $\Gamma$ on $(N_1 \times N_2) \rtimes \R$
  induces a quasi-action of $\Gamma$  on $N_i \rtimes \R$  for $i=1, 2$.
   A Tukia-type theorem (if available)  can then be applied  to get isometric actions of $\Gamma$ on
 $N_1 \rtimes \R$   and $N_2 \rtimes \R$.  One  then  tries
  to show that the two isometric actions
 combine to give an isometric action of $\Gamma$ on $(N_1 \times N_2) \rtimes \R$.
  Using this strategy and our Tukia-type theorem  we get:





\begin{theorem}\label{theorem:QIrigid}
Let $S=(N_1 \times N_2) \rtimes \R$  be  a  SOL-like group,
  where each $(N_i, D_i)$ is either   Carnot  or
Carnot-by-Carnot.
In the case  $(N_i, D_i)$ is Carnot-by-Carnot with  $\dim(\mathfrak w_i)=1$ for  at least one $i$, where  $\mathfrak w_i$ is the Lie sub-algebra of $\mathfrak n_i$ generated by the eigenspace of  the smallest eigenvalue of $D_i$, we further assume  that $\Gamma$ is amenable or that  the isometry group $\text{Isom}(S,g)$ admits a  uniform lattice
  for some  left invariant Riemannian metric $g$  on $S$.
 There is a  left invariant Riemannian metric $g_0$ on $S$ with the following property.  Suppose a finitely generated group $\Gamma$ is   quasi-isometric  to $S$.       Then  $\Gamma$ is, up to finite index and finite kernel, a lattice in
    $\text{Isom}(S,g_0)$.   
\end{theorem}


Note that any two finitely generated groups that differ by finite index or finite kernel are automatically quasi-isometric so it is impossible  to drop these conditions.
The assumption that $\text{Isom}(S,g)$ admits a  uniform lattice implies that $\Gamma$ is amenable and so we can apply Theorem \ref{main-uniform}:   by Theorem 3 in \cite{W00},
a  uniform lattice  $\Gamma_0$ in $\text{Isom}(S,g)$   is virtually solvable; on the other hand,  $\Gamma$ is quasi-isometric to $S$ and so is also  quasi-isometric to $\Gamma_0$; since  amenability is a quasi-isometry invariant among finitely generated groups it follows that $\Gamma$ is amenable.
   In  Appendix  \ref{lattice}   we give examples of SOL-like groups  $S$ that admit lattices (in this case $\text{Isom}(S,g)$  admit uniform lattices for any  left invariant Riemannian metric $g$). The proof of Theorem \ref{theorem:QIrigid} can be found in Section \ref{cdiff-sec}.


{Theorem \ref{theorem:QIrigid}   is related to Conjecture 1.2.2(2)  in \cite{LDPX22}.    }
Theorem \ref{theorem:QIrigid} also has consequences for SOL-like groups  {whose isometry groups do not admit  lattices.}

 \begin{conjecture}
   Let $S$ be a  SOL-like group.  Suppose
$\text{Isom}(S,g)$  does not admit a  uniform lattice for any
left invariant Riemannian metric $g$  on $S$.     Then $S$ is not quasi-isometric  to  any finitely generated group.
\end{conjecture}

 By Theorem \ref{theorem:QIrigid} this conjecture is true when both $N_i$ are either Carnot or Carnot-by-Carnot with $\dim(\mathfrak w_i) \geq 2$. In case $\dim(\mathfrak w_i)=1$ we can only conclude that $S$ is not quasi-isometric to any amenable finitely generated group.





{\bf Quasi-isometric  classification of solvable Lie groups.}
An open problem  concerning the large scale geometry of Lie groups  is the quasi-isometry  classification of simply connected solvable Lie groups.    A simply connected solvable Lie group $S$ with Lie algebra $\mathfrak s$  is of   \emph{real type} (also called completely solvable, or split-solvable) if all   eigenvalues of
   $\text{ad}(X)$   are  real   for all $X\in \mathfrak s$.    By  Theorem 4.21 in \cite{CKLGO21}
     for any simply connected solvable Lie group $S$,  there is a unique simply connected solvable Lie group $S_{\mathbb R}$  of real type,
       so called   \emph{real shadow}  of $S$,      such that  $S$ and
    $S_{\mathbb R}$  can be made isometric to each other.   Recall that two Lie groups
    $S_1$ and $S_2$ can be made isometric to each other if
     there are left invariant  Riemannian metrics $g_i$  on $S_i$ ($i=1, 2$)     such that $(S_1, g_1)$ and $(S_2, g_2)$ are isometric.
    Notice that this result also implies  two simply connected solvable Lie groups $S_1$ and $S_2$  of real type  can be made isometric if and only if they are isomorphic.  It follows that to classify simply connected solvable Lie groups up to quasi-isometry, it suffices to restrict attention to the class of simply connected solvable Lie groups of real type.
     Here is a  conjecture  by   Y. Cornulier.

  \begin{conjecture} (\cite{dC18}, Conjecture 19.113)
     Two simply connected solvable Lie groups $S_1$ and $S_2$  of real type
      are  quasi-isometric if and only if they are isomorphic.
     \end{conjecture}

   The following two results
   confirm the above conjecture    for
     two  classes   of   solvable Lie groups.

  \begin{theorem}\label{QIclassification}
 Let $\mathcal S$ be the class of SOL-like groups $S=(N_1\times N_2)\rtimes \mathbb R$ where $(N_i, D_i)$  ($i=1, 2$) is either Carnot  or Carnot-by-Carnot.  Then
  two members  $S_1, S_2\in \mathcal S$ are quasi-isometric if and only if they are isomorphic.

 \end{theorem}

In the special case when $S$ is non-unimodular and $(N_i, D_i)$  ($i=1, 2$) is Carnot,
 Theorem \ref{QIclassification}  is the same as Theorem C , \cite{Fe22}.  Similarly Theorem \ref{QIC-CarnotbyCarnot}  follows from Pansu's differentiability
 theorem in the case when $(N, D)$ is  Carnot.


\begin{theorem}\label{QIC-CarnotbyCarnot}
 Let $\mathcal S$ be the class of simply connected solvable Lie groups of the form
  $S=N\rtimes_D \mathbb R$, where $(N, D)$ is  Carnot  or Carnot-by-Carnot.   Then two members $S_1, S_2\in \mathcal S$ are quasi-isometric if and only if they are isomorphic.

 \end{theorem}

As indicated above, a quasi-isometry between $S_1, S_2$ induces a quasi-action of $S_1$ on $S_2$. Our Tukia-type  theorem and Eskin-Fisher-Whyte's coarse differentiation method  then imply this quasi-action can be conjugated into an isometric action of $S_1$  on $S_2$.
 From here it is not too difficulty to show that $S_1$ and $S_2$ are isomorphic.

{\bf Structure of the paper:}    We collect
the preliminaries  in Section \ref{prelim-sec}.  The proofs of the quasi-isometric rigidity theorems are covered in Section \ref{cdiff-sec}.
The proof of Theorem \ref{main-uniform}  spans from  Section \ref{compatible}  to Section \ref{n/w}, see Section \ref{outline}  for  an  outline of this proof.   
 Appendix  \ref{nilexample}    provides examples of Carnot-by-Carnot groups.
 In   Appendix \ref{lattice}  we review some  examples of SOL-like groups that admit  lattices.   The content of    Appendix C is  the proof of   Lemma \ref{sj}.

{\bf{Acknowledgments.}}
T. Dymarz was supported by NSF    career  Grant 1552234.   D. Fisher was supported by NSF grants DMS-1906107 and DMS-2246556.
 X. Xie  was  supported by Simons Foundation grant \#315130. X. Xie
   would  like to thank  the department of mathematics,   University of Wisconsin at Madison   for financial support  during his visit there   in February 2020. We would also like to thank Dave Witte Morris for useful conversations  {and thank  Tom Ferragut and Gabriel Pallier for thoughtful comments on an earlier version.}

%

\section{Outline of main theorem} \label{outline}
 In this section we outline the proof of Theorem \ref{main-uniform}. 

Let $\G$ be a group acting on a Carnot-by-Carnot group $(N,D)$  by uniform quasisimilarities with respect to a $D$-homogeneous distance
 that satisfies the conditions of Theorem \ref{main-uniform}.    We use the notation $W$, $\mathfrak w$ and $\bar D$ from Definition \ref{carnotbycarnot} for  a  Carnot-by-Carnot group.
Without loss of generality we can assume that $D$ scales by $1$ on the first layer of $\mathfrak{n}$ (equivalently the first layer of $\mathfrak{w}$) and by $\alpha$ on the first layer of $\mathfrak{n}/\mathfrak{w}$. For example if $N\simeq \R^n$ then $D$ is a diagonal matrix with diagonal entries  $1$ and $\alpha$.
 (see Section \ref{prelim-sec} for details).


   A biLipschitz map  $f$ of $N$ is a  \emph{fiber similarity  map}  if  there exist  Carnot-Caratheodory metrics   $\bar d_{CC}$ and $d_{CC}$ on $N/W$ and  $W$ respectively
  such that       \begin{enumerate}
 \item   $f$  induces a  similarity of  $(N/W, \bar d_{CC})  $,
\item    $f$   acts by similarities along cosets of $W$, i.e. for each $g \in N$   
  the map
     $L_{f(g)^{-1}}\circ f \circ L_g : (W, d_{CC}) \to (W, d_{CC})$ is a similarity.
     \end{enumerate}
  We say  $\Gamma$ is  a \emph{fiber similarity group}  if  there exist  Carnot-Caratheodory metrics   $\bar d_{CC}$ and $d_{CC}$ on $N/W$ and  $W$ respectively
  such that   every element $\gamma\in \Gamma$   is a fiber similarity map with respect to  $\bar d_{CC}$ and $d_{CC}$.
   When $\dim(N/W)\ge 2, \dim(W)\ge 2$,    Theorem \ref{foliatedtheorem}  implies that  $\Gamma$ can be conjugated into  a fiber similarity group.


{\bf Individual biLipschitz maps (Section \ref{compatible}).}   
 In Section \ref{compatible}  we study the structure of individual  fiber similarity  maps  of $N$.  
   In general  fiber similarity  maps are not similarities of $N$ since     the similarity induced on $N/W$ may not be compatible with the similarities on the cosets of $W$  and
there may be a  shear along the cosets of $W$.
When $N$ is abelian, one can show that any  fiber similarity   map is the composition of a similarity of $N$ and   
  a shear map $(w,h)  \mapsto (w+ s(h),h)$ along cosets of $W$ (these are called ``almost similarities" in \cite{D10}). The same is true if $N$ is a direct product of  $W$  and $H$: up to left translation, any such map has the form $(w,h) \mapsto (\phi (w)s(h),  B(h))$ where
  $\phi:W \to W$ and $B:H \to H$ are automorphisms, $(w,h) \mapsto (\phi(w), B(h))$ is an automorphism of $N$ (and a similarity with respect to some $D$-homogeneous distance) and $s:H \to W$ is the {shear} amount.

The general case is more involved. To understand the interaction of $W$ and $N/W$ we decompose the  Lie algebras $\mathfrak n$ of $N$ and $\mathfrak w$ of $W$ into eigenspaces
 $$\mathfrak n=V_{\lambda_1} \oplus \cdots \oplus V_{\lambda_n} \textrm{ and } \mathfrak w= W_1 \oplus \cdots \oplus W_m.$$ Here $V_{\lambda_j}$ is the eigenspace associated to the eigenvalue $\lambda_j$ of $D$ and $W_i$ is the eigenspace of $D\vert_{\mathfrak w}$ associated to eigenvalue $i$. Recall that without loss of generality we have assumed that $\lambda_1=1$ so that, since ${\mathfrak w}$ is Carnot with respect to $D\vert_{\mathfrak w}$, its eigenvalues are all positive integers.  The eigenvalue $\lambda_j$ will either be an integer or a multiple of $\alpha$ (or both) where $\alpha$ is the smallest eigenvalue of the induced derivation $\bar{D}$ on the Lie algebra of $N/W$.

 If $k$ is not an integral multiple of $\alpha$ then $V_k=W_k$ otherwise $k=j\alpha$ and we write $$ V_k=W_{j\alpha} \oplus H_j$$
for some complement $H_j \subset V_k$  of  $W_{j\alpha}$  (see Section \ref{prelim-sec} for details).
Then we have a bijection from $H=H_1 \oplus  \cdots \oplus H_s$ to $\mathfrak n/\mathfrak w$ so that $H$ is a transversal for $\mathfrak w$ in $\mathfrak n$. Note that
 a  fiber similarity  map does not necessarily preserve $H$.
  In Lemma \ref{sj} we show that any such map must have a \emph{compatible expression}.  For a rigorous definition of compatible expression see Definition \ref{defcompatible} but morally speaking it still involves  an automorphism $\phi:W\to W$ and a shear map on $N$ determined by a map
  $s:  N/W \to W$    but now $B:H \to \mathfrak n$ is  only a linear map from the transversal $H$ into $\mathfrak n$ that satisfies the following  compatibility condition with $\phi$: for any $w \in \mathfrak w$ and $h \in H$,  $[d\phi(w), Bh]=d\phi[w,h]$. In particular $w+h \mapsto d\phi(w)+ Bh$ does not necessarily define an automorphism of $\mathfrak n$. In addition both $B$ and the shear $s$ depend on the choice of transversal $H$.

{\bf Case when $\alpha$ is not an integer.} Note that when $\alpha$ is not an integer then $V_\alpha= H_1$  and  $\mathfrak n$ is a central product (see Section \ref{structure of n}). This simplifies many of our arguments.  
 In fact a substantial portion of our later analysis takes place on $V_\alpha=W_\alpha \oplus H_1$ when $W_\alpha \neq 0$.


{\bf Shear maps (Section \ref{shear-bilipschitz}).} In Section \ref{shear-bilipschitz} we characterize  maps $s:  N/W \to W$   for which  the corresponding shear maps $g \mapsto g\cdot s(gW)$  are  biLipschitz maps of $N$.     BiLipschitz  shear maps are used to conjugate a uniform quasisimilarity group to eliminate the  $W_j$ shear components  for $j<\alpha$  in the compatible expression of group elements.   As above we work with Lie algebra coordinates, we identify $W$ with $\mathfrak w$ and $N/W$ with $\mathfrak n/\mathfrak w$ via the exponential map and write $s$ as $s: \mathfrak n/\mathfrak w\ra \mathfrak w$.
First we show that all such maps  $s$ must have image in $Z(\mathfrak w)$, the center of $\mathfrak w$.
In the abelian case any $\frac{1}{\alpha}$-H\o lder map $\mathfrak n/\mathfrak w \to Z(\mathfrak w)$ induces a biLipschitz map of $N$. When $\alpha$ is not an integer
there is also a simple description of  such admissible  maps  $s$. In general, the description is more complex. To describe  $s:\mathfrak n/\mathfrak w \to Z(\mathfrak w)$ we first decompose $s=(s_i)$  where $s_i: H \to  Z_i:= W_i\cap Z(\mathfrak w)$.   Then we show that $s$ is admissible as long as each $s_i$ lies in a function space $\mathcal{H}_i$  that is recursively defined by the vanishing of certain integrals along horizontal closed curves.  (See Section \ref{shear-sec} and especially Propositions \ref{shear-bilip} and \ref{construction} for details).



{\bf Day's theorem (Section \ref{eliminate}).} Now we are ready to look at the entire uniform group $\G$, acting by fiber similarity  maps on $N$. The goal in this step is to show that  after conjugation
   the shear  components $s_j$ for $j < \alpha$
     of each group element must be zero. This is where Day's theorem comes into play. Recall that Day's theorem shows that any amenable group that has an affine action on a compact convex subset of a locally convex topological vector space has a fixed point.
In previous cases, such as those covered in \cite{DyX16},  Day's theorem applies directly to the action of $\G$ on the Banach space of $\frac{1}{\alpha}$-Holder functions.   Using the fixed point  obtained, one can use the corresponding biLipschitz shear map to conjugate the group action to an action where all shear  components are trivial. In this paper   we need to show that the fixed point lies  in $\mathcal{H}_j$, the subspace of H\o lder continuous maps defined above.   We do this in order to construct a conjugating biLipschitz shear map, see the preceding paragraph.


 {\bf{Conformal structure in the $V_\alpha$ direction (Section \ref{cstructure}).}}
   At this point,
  the shear  components  $s_j$ for $j<\alpha$ of each  group element  are  zero.
   In Section \ref{cstructure} we conjugate one more time so that $s_\alpha$ is a homomorphism. This  part requires substantially more work. It involves first  applying a modified version of a foliated Tukia-type argument (as in \cite{D10} and \cite{DFX})  that uses a measurable conformal structure defined solely in the $V_\alpha$  direction.
   Some care is needed to define this structure since our maps do not necessarily preserve $V_\alpha$. Once we know that $s_\alpha$ is a homomorphism it is not hard to show that our maps are affine maps  of $N$ and hence similarities with respect to some $D$-homogeneous   distance on $N$.
   This section is only needed when $\alpha$ is an integer.

{\bf Dimension one (Sections \ref{dim(w)=1}    and \ref{n/w}).}  Theorem \ref{foliatedtheorem}
 only applies when $\dim(N/W)$ and $\dim(W)$ are both at least two. In the case where $\dim(W)=1$ we are not able to define a foliated conformal structure along $W$ but instead we use the action of the group on the space of derivative maps in the $W$ direction. Once again we require Day's theorem which we apply to this action. We are able to construct a conjugating map that conjugates our group action into one by  fiber similarity maps. 
   If instead $\dim(N/W)=1$  then there are two cases. If $N$ is a
 direct product we can use Farb-Mosher \cite{FM99} to conjugate the action on $N/W$. Otherwise we use the algebraic structure to show that the group action was already by  fiber similarity maps. 

\section{Preliminary}\label{prelim-sec}

\subsection{Quasi-actions  and  uniform quasisimilarity actions}

Let $L\ge 1$, $A\ge 0$.
A (not necessarily continuous) map
$f:X\ra Y$ between two metric spaces is an $(L,A)$-\emph{quasi-isometry}
if:
\newline (1) $d(x_1,x_2)/L-A\le d(f(x_1), f(x_2))\le L\, d(x_1, x_2)+A$
for all $x_1,x_2\in X$;
\newline (2)  for any $y\in Y$, there is some $x\in X$ with
$d(f(x), y)\le A$.
\newline
  An $(L,0)$-quasi-isometry is also called a $L$-biLipschitz map.

   One of the main motivations for studying  quasi-isometries comes from geometric group theory.
A finitely generated group $\Gamma$ with a symmetric finite generating set $S$ has an associated left invariant \emph{word metric} given by $d_S(g,h)= \| g^{-1} h \|_S$ where $\| a \|_S$ denotes the minimal $k \in \Z_+$ such that $a=s_1 s_2 \cdots s_k$ with $s_i \in S$ for $i=1, \ldots , k$. It is a simple exercise to show that for any two finite generating sets $S, S'$ we have that $(\Gamma, d_S)$ and $(\Gamma, d_{S'})$ are quasi-isometric.

The definition of a quasi-action  can be found in the Introduction.




Let $\Lambda \ge 1$ and $C>0$. A bijection $F:X\ra Y$ between two
metric spaces is called a $(\Lambda,C)$-\emph{quasisimilarity}    if
\[
   \frac{C}{\Lambda}\, d(x,y)\le d(F(x), F(y))\le C\,\Lambda\, d(x,y)
\]
for all $x,y \in X$.
   When $\Lambda=1$, we say $F$ is a {similarity}.
It is clear that a map is a quasisimilarity if and only if it is a
biLipschitz map. The point of using the notion of quasisimilarity is
that sometimes there is control on $\Lambda$ but not on $C$.
 An action  $\Phi: \Gamma\ra \text{Homeo}(X)$ of a group $\Gamma$ on a metric space $X$ is  a \emph{uniform quasisimilarity action} if there is some $\Lambda\ge 1$  with the following property: for every $\gamma\in \Gamma$ there is a constant $C_\gamma>0$  such that $\Phi(\gamma)$ is a $(\Lambda, C_\gamma)$-quasisimilarity.


Let $X$ be a Gromov hyperbolic space, and $\Gamma$ a group quasi-acting on $X$.
  If there is a point $\infty\in \partial X$ such that for any $\gamma\in \Gamma$,
    the homeomorphism of $\partial X$  induced by $G_\gamma$ fixes $\infty$ and is biLipschitz with respect to a fixed visual metric on $\partial X$, then the quasi-action of $\Gamma$ on $X$ induces a uniform quasisimilarity action of $\Gamma$ on
     $(\partial X\backslash \{\infty\},  d)$ for some parabolic visual metric   $d$  based at $\infty$.
       Conversely,  if $X$ is a visual Gromov hyperbolic space and $\Gamma$
   admits a uniform quasisimilarity action on   $(\partial X\backslash \{\infty\},  d)$ for some parabolic visual metric  $d$  based at $\infty$, then this action extends to a quasi-action of $\Gamma$ on $X$.   More generally, there is a notion of uniform quasi-Mobius actions,
     and  quasi-actions on  visual Gromov hyperbolic spaces correspond to  uniform quasi-Mobius actions on the visual boundary.  Since we do not need  the notion of uniform quasi-Mobius actions,  we shall not recall the definition here.


\subsection{Nilpotent Lie algebras and nilpotent Lie groups}\label{nilpprelim}

Let $\mathfrak n$ be a Lie algebra.  The lower central series is defined recursively as follows:
 ${\mathfrak n}^{(1)}=\mathfrak n$, ${\mathfrak n}^{(k)}=[\mathfrak n, \mathfrak n^{(k-1)}]$.
 The Lie algebra $\mathfrak n$ is called nilpotent if  ${\mathfrak n}^{(t+1)}=0$  for some $t\ge 1$. The smallest such $t$ is called the nilpotency of $\mathfrak n$.
  A connected Lie group is nilpotent if and only if its Lie algebra is nilpotent.

Let  $N$ be a simply connected nilpotent Lie group with Lie algebra $\mathfrak n$.
  Then the  exponential map  $\text{exp}:  {\mathfrak n}\ra N$ is a diffeomorphism.
    Under this identification the Lebesgue   measure on
  $\mathfrak n$  is a  Haar measure on $N$ (\cite{CG90},  page 19).
   One can
     pull back the group operation from $N$ to get a group structure
 on $\mathfrak n$.     This group structure on $\mathfrak n$  can be described by the   Baker-Campbell-Hausdorff formula
 (BCH formula in short),  which expresses the product $X*Y$ ($X, Y\in {\mathfrak n}$)
    in terms of the iterated Lie brackets  of $X$ and $Y$.
     The group operation in $N$ will be denoted by $\cdot$.
   The pull-back group operation   $*$  on $\mathfrak n$ is defined as follows.
      For $X,   Y\in \mathfrak n$,   define
   $$X*Y=\text{exp}^{-1}(\text{exp} X\cdot \text{exp} Y).$$
  Then the BCH formula (\cite{CG90},    page 11)  says
\begin{align*}
 X*Y  & =X+Y+\frac{1}{2}[X,Y]+\frac{1}{12}[X,[X,Y]]
-\frac{1}{12}[Y, [X, Y]]+\cdots.
\end{align*}

We also recall the following formula (\cite{Hel78}, pages 127, 128):
\begin{equation}\label{conjugationformula}
Y*X*(-Y)=\sum_{k=0}^\infty\frac{1}{k!}({\rm ad} \,Y)^k (X)=X+[Y, X]+\cdots +\frac{1}{k!}({\rm ad} \,Y)^k (X)+\cdots.
\end{equation}

Recall, if $\phi: G_1\to G_2$ is a Lie group homomorphism and $d\phi: \mathfrak g_1\to \mathfrak g_2$ denotes the associated Lie algebra homomorphism then
$\phi\circ \text{exp}= \text{exp}  \circ d\phi$.

{\bf Standing Assumption.} {In this paper (particularly in Sections \ref{shear-bilipschitz}--\ref{cstructure}   and Appendix \ref{proof of compatible}) we shall often identify a simply connected  nilpotent Lie group  $N$ with its Lie algebra  $\mathfrak n$ via the exponential map.  With this identification, the inverse
   (with respect to $*$)   of an element $X\in\mathfrak n$  is $-X$:  $X^{-1}=-X$.
     For this reason we denote by $0$ the identity element of a Lie group.  }


The following statement is known (see for example \cite{P11}, p.980). We include it here for completeness.

\begin{Le}\label{irine}
Let  $\phi, \tilde \phi: N_1\ra N_2$ be Lie group   homomorphisms between two simply connected nilpotent Lie groups.
 Let $N_2$ be equipped with a left invariant metric $d$ such that  $d$  induces the manifold topology and  closed balls are compact. Suppose
  $C:=\sup_{x\in N_1}\{d(\phi(x), \tilde \phi(x))\}<\infty$. Then $\phi=\tilde \phi$.

\end{Le}


\begin{proof}
  It suffices to show    the   corresponding Lie algebra homomorphisms   agree.    Suppose the contrary that there is some $X\in \mathfrak n_1$ such that
  $d\phi(X)\not= d\tilde \phi(X)$.  Let $\{\mathfrak n_2^{(j)}\}_j$ be the lower central series of $\mathfrak n_2$  and $\pi_j: \mathfrak n_2\ra  \mathfrak n_2/{\mathfrak n_2^{(j)}}$ the quotient homomorphism.   Let $k\ge 2$ be the integer satisfying
      $d\phi(X)-d\tilde \phi(X)\in \mathfrak n_2^{(k-1)}\backslash \mathfrak n_2^{(k)}$.
 Set $Y=\pi_k(d\phi(X))$,    $\tilde Y=\pi_k(d\tilde \phi(X))$  and $\mathfrak g=  \mathfrak n_2/{\mathfrak n_2^{(k)}}$.
   Then   $Y=\tilde Y+Z$ for some $Z\not=0$ in the center of $\mathfrak g$. The assumption implies
$\bar d(tY, t\tilde Y)\le C$  for all $t\in \mathbb R$, where $\bar d$ is the  distance on $\mathfrak g$  induced by $d$:
 $\bar d (\bar a, \bar b)=\inf\{d(a, b)| a\in \pi_k^{-1}(\bar a),   b\in \pi_k^{-1}(\bar b) \}$.   Here we identified $N_2$ with $\mathfrak n_2$ and so $d$ is a  distance on $\mathfrak n_2$.
    On the other hand, the BCH formula gives $(-t\tilde Y)*(tY)=tZ$ and so
 $C\ge \bar d(tY, t\tilde Y)=\bar d(0, (-t\tilde Y)*(tY))=\bar d(0, tZ)$.  Pick $z\in\pi^{-1}_k(Z)$.  Then there is some
  $a_t\in {\mathfrak n_2^{(k)}}$ satisfying $d(a_t, tz)\le C$.   We can write $tz=a_t*b_t$ for some $b_t\in B:=\{x\in \mathfrak n_2| d(0, x)\le C\}$.
     Since ${\mathfrak n_2^{(k)}}$  is an ideal of $\mathfrak n_2$,  by the BCH formula we can write $a_t*b_t=c_t+b_t$ for some $c_t\in {\mathfrak n_2^{(k)}}$.
      We have $tz=c_t+b_t$  and so $z=\frac{c_t}{t}+\frac{b_t}{t}$.  Since the  ball $B$ is compact, we have $\frac{b_t}{t}\to 0$ as $t\to \infty$.
       This implies $\frac{c_t}{t}\to z$ as $t\to \infty$.   As $\frac{c_t}{t}\in {\mathfrak n_2^{(k)}}$  and ${\mathfrak n_2^{(k)}}$ is closed in $\mathfrak n_2$, we conclude
       that $z\in {\mathfrak n_2^{(k)}}$, contradicting the  fact that $Z\not=0$.

\end{proof}

\subsection{Carnot algebras   and  Carnot groups  }\label{basics}

  Carnot algebras  form a particular class of nilpotent Lie algebras.
A   {Carnot Lie algebra} is a finite dimensional Lie algebra
$\mathfrak n$   together with  a direct sum   decomposition
    $\mathfrak n=V_1\oplus V_2\oplus\cdots \oplus V_r$
  of   non-trivial   vector subspaces
 such that $[V_1,
V_i]=V_{i+1}$ for all $1\le i\le r$,
    where we set $V_{r+1}=\{0\}$.    For every Carnot algebra
 $\mathfrak n=V_1\oplus V_2\oplus\cdots \oplus V_r$,  the \emph{Carnot derivation}
     $D: \mathfrak n\ra \mathfrak n$ is given by
  $D(x)=ix$ for $x\in V_i$.   The automorphisms of $\mathfrak n$ generated by the Carnot derivation are called  \emph{Carnot dilations}
   and are given by
  $\delta_t: \mathfrak n \ra \mathfrak n$, $t\in (0, +\infty)$, where
 $\delta_t(x)=t^i x$ for  $x\in V_i$.
  Let   $\mathfrak n=V_1\oplus V_2\oplus\cdots \oplus V_r$
    and $\mathfrak n'=V'_1\oplus V'_2\oplus\cdots \oplus V'_s$  be two  Carnot
    algebras.
  A Lie algebra  homomorphism
  $\phi: \mathfrak n\ra \mathfrak n'$
     is graded if 
  $\phi(V_i)\subset V'_i$ for all $1\le i\le
  r$.


 A  simply connected nilpotent Lie group is a  {Carnot group}
 if its Lie algebra is a Carnot algebra.
    Let $N$ be a Carnot group with Lie algebra
      $\mathfrak n=V_1\oplus \cdots \oplus V_r$.  The subspace $V_1$ defines
      a left invariant distribution  $H N\subset TN$ on $N$.   
           An
      absolutely continuous curve $\gamma$ in $N$  whose velocity vector
       $\gamma'(t)$  is contained in  $H_{\gamma(t)} N$ for a.e. $t$
        is called  a horizontal curve.
          By Chow's theorem ([BR, Theorem 2.4]),   any two points
  of $N$ can be  connected by horizontal curves.
    Fix a left invariant inner product on
          $HN$.
  Let $p,q\in N$, the
  \emph{Carnot-Caratheodory   metric} $d_{CC}(p,q)$   between them is defined as
  the infimum of length of horizontal curves that join $p$ and $q$.
      Since the inner product on   $HN$ is left invariant, the Carnot
  metric on $N$ is also left invariant.  Different choices of inner
  product on $HN$ result in Carnot metrics that are biLipschitz
  equivalent.
    The Hausdorff dimension of $N$ with respect to  a  Carnot metric
    is  given by $\sum_{i=1}^r i\cdot \dim(V_i)$.

       \begin{definition}\label{pansu-d}
 Let $G$ and $G'$
  be two Carnot groups endowed with Carnot-Caratheodory metrics,  and $U\subset G$, $U'\subset G'$ open subsets.
   A map $F: U\ra U'$ is {Pansu  differentiable}
    at $x\in U$  if there exists a   graded   homomorphism
     $L: G\ra G'$ such that
     $$\lim_{y\ra x}\frac{d(F(x)^{-1}*F(y),\, L(x^{-1}*y))}{d(x,y)}=0.$$
          In this case, the  graded  homomorphism
     $L: G\ra G'$  is called the {Pansu  differential} of $F$ at $x$, and
     is denoted by   $DF(x)$.
\end{definition}

Pansu  \cite{P89}  showed that  a Lipschitz map between Carnot groups is Pansu differentiable  a.e. and if the map is biLipschitz then   the Pansu differential is a graded isomorphism.


\subsection{Derivations, semi-direct products, Heintze groups, and SOL-like groups}\label{deri}

Let $\mathfrak n$ be a Lie algebra.
A linear map $D: \mathfrak n\ra \mathfrak n$  is called a derivation if $D[X,Y]=[DX, Y]+[X, DY]$ for all $X, Y\in \mathfrak n$.
    Let $\text{der}(\mathfrak n)$ be the set of derivations of $\mathfrak n$. It is clearly a vector space.
     It becomes a Lie algebra with the usual bracket for endomorphisms: for $D_1, D_2\in \text{der}(\mathfrak n)$,
      define $[D_1, D_2]=D_1\circ D_2-D_2\circ D_1$.

        Let $\mathfrak w$ and $\mathfrak h$ be two Lie algebras, and $\phi: \mathfrak h\ra  \text{der}(\mathfrak w)$
         a Lie algebra homomorphism.  The semi-direct product $\mathfrak w\rtimes_\phi \mathfrak h$ is the
         vector space $\mathfrak w\oplus \mathfrak h$ with the Lie bracket given by:
         $$[(W_1, H_1), (W_2, H_2)]=([W_1, W_2]+\phi(H_1)(W_2)-\phi(H_2)(W_1),   [H_1, H_2]).$$

         Let $\mathfrak n$ be a Lie algebra and $D\in  \text{der}(\mathfrak n)$.  Then   we can define  a homomorphism $\phi: \mathbb R\ra  \text{der}(\mathfrak n)$ by $\phi(t)=tD$. We denote
         $\mathfrak n\rtimes_D\mathbb R:=\mathfrak n\rtimes_\phi\mathbb R$.
            When $\mathfrak n$ is a nilpotent Lie algebra and the eigenvalues of $D\in \text{der}(\mathfrak n)$
             have positive real parts, we call   $\mathfrak n\rtimes_D\mathbb R$ a Heintze algebra.   In this case, the simply connected
              solvable Lie group  with Lie algebra  $\mathfrak n\rtimes_D\mathbb R$  is called  a Heintze group,  and   the pair $(\mathfrak n, D)$  (and also $(N,D)$)
                  is called a Heintze pair.
              A Heintze algebra (or Heintze group, or the pair $(\mathfrak n, D)$)  is called purely real if all the eigenvalues of $D$ are real numbers.   By \cite{Ale75},  every Heintze group is biLipschitz to a purely real Heintze group.
               Since we are only  interested in quai-isometries of Heintze groups in this paper, we    only need to consider purely real Heintze groups.


  A  purely real Heintze pair
$(\mathfrak n, D)$   is  called  diagonal if $D$ is diagonalizable.

\begin{Le}\label{derivation}\label{derivation lemma}
    Suppose  $(\mathfrak n, D)$ is a  diagonal  Heintze pair  and
    $0<\lambda_1<\cdots <\lambda_r$ are the  eigenvalues  of $D$.    Denote by $V_{\lambda_i}$ the eigenspace   associated to $\lambda_i$. Then\newline
  (1) $\mathfrak n=\oplus_{i=1}^r V_{\lambda_i}$  and $[V_{\lambda_i}, V_{\lambda_j}]\subset  V_{\lambda_i+\lambda_j} $;\newline
   (2)  Let $\mathfrak w$ be the Lie sub-algebra of $\mathfrak n$ generated by $V_{\lambda_1}$.   Then
     $\mathfrak w$ is a Carnot algebra and  $D(\mathfrak w)=\mathfrak w$.  If $\mathfrak w$ is an ideal of $\mathfrak n$, then $D$  projects to a derivation $\bar D$ of $\bar{\mathfrak n}:=\mathfrak n/\mathfrak w$,  which is
     diagonalizable  with positive eigenvalues.   Furthermore, the smallest eigenvalue  of $\bar D$ is  strictly larger than $\lambda_1$.

\end{Le}

\begin{proof}
(1)  $\mathfrak n=\oplus_{i=1}^r V_{\lambda_i}$    follows from linear algebra and
$[V_{\lambda_i}, V_{\lambda_j}]\subset  V_{\lambda_i+\lambda_j} $  follows from the definition of derivation.

  (2)   Set $W_1=V_{\lambda_1}$  and $W_j=[W_1, W_{j-1}]$ for $j\ge 2$.  By (1) we have
  $W_j\subset  V_{j\lambda_1}$. It follows that $\mathfrak w=\oplus_j W_j$ is a Carnot grading  and $D(\mathfrak w)=\mathfrak w$.
    Now assume $\mathfrak w$ is an ideal.   
     As $D(\mathfrak w)=\mathfrak w$, $D$ induces a   derivation $\bar D: \bar{\mathfrak n}\ra \bar{\mathfrak n}$.

 Since $\mathfrak w$ is graded,  that is,  $\mathfrak w=\oplus_i(\mathfrak w\cap V_{\lambda_i})$,
we have $\bar{\mathfrak n}=\oplus_i (V_{\lambda_i}/(\mathfrak w\cap V_{\lambda_i})$.  As $D|_{V_{\lambda_i}}$ is multiplication by $\lambda_i$,  $\bar D$ is also multiplication by $\lambda_i$ when restricted   to
  $V_{\lambda_i}/(\mathfrak w\cap V_{\lambda_i})$. It follows that $\bar D$ is also diagonalizable and  its set of eigenvalues is a
  subset of the set of eigenvalues of $D$.   Since $W_1=V_{\lambda_1}$ we have $V_{\lambda_1}/(\mathfrak w\cap V_{\lambda_1})=0$  and so $\lambda_1$ is   not  an eigenvalue of $\bar D$.  It follows that the smallest eigenvalue of $\bar D$ is $>\lambda_1$.

\end{proof}




 A Heintze pair $(\mathfrak n, D)$ (Heintze algebra, Heintze group) is  of Carnot type if it is purely real  with $D$ diagonal  such that $\mathfrak w=\mathfrak n$, where $\mathfrak w$ is as in Lemma \ref{derivation}.  A  Heintze pair   is of non-Carnot type otherwise.

   Let $(\mathfrak n_i, D_i)$ ($i=1, 2$)  be a Heintze pair.
               Let  $D$ be the derivation of $\mathfrak n_1\times \mathfrak n_2$ given by
              $D=(D_1, -D_2)$.   We call $(\mathfrak n_1\times \mathfrak n_2)\rtimes_D \mathbb R$ a SOL-like algebra  and
                the simply connected solvable Lie group  with Lie algebra $(\mathfrak n_1\times \mathfrak n_2)\rtimes_D \mathbb R$   a   SOL-like group.

 For any Lie sub-algebra $\mathfrak h$ of a Lie algebra $\mathfrak n$,
  the normalizer of $\mathfrak h$ in $\mathfrak n$ is defined by:
   $$N(\mathfrak h)=\{x\in \mathfrak n:  [x, \mathfrak h]\subset \mathfrak h\}.$$


   \begin{lemma} \label{normalizer}
    If a derivation $D$ of a Lie algebra $\mathfrak n$ satisfies  $D(\mathfrak h)\subset  \mathfrak h$, then
   $D(N(\mathfrak h))\subset N(\mathfrak h)$.
   \end{lemma}

   \begin{proof}
       Let $X\in N(\mathfrak h)$  and $Y\in \mathfrak h$   be arbitrary.  Then
         $$[D(X), Y]=D[X, Y]-[X, D(Y)]\in \mathfrak h$$
          implying $D(X)\in N(\mathfrak h)$.

   \end{proof}

\subsection{Homogeneous distances on nilpotent Lie groups}\label{homo distance}

Let $(N,D)$ be a   diagonal  Heintze pair.
  A distance   $d$  on $N$ is called $D$-homogeneous    if it is left invariant,   induces the manifold topology on $N$ and such that
 $d(e^{tD}x, e^{tD}y)=e^t d(x, y)$ for all $x, y\in N$ and $t\in \mathbb R$, where $\{e^{tD}|t\in \mathbb R\}$  denotes the automorphisms  of $N$ generated by the derivation $D$.    Let $\mathfrak n=\oplus_j V_{\lambda_j}$ be the decomposition of $\mathfrak n$ into the direct sum of eigenspaces of $D$.
  An inner product $\left<\cdot , \cdot \right>$ on $\mathfrak n$ is called a $D$-inner product if
  the eigenspaces corresponding to distinct eigenvalues are perpendicular with respect to
  $\left<\cdot , \cdot \right>$.  By the construction in Theorem 2 of \cite{HSi90}, given any $D$-inner product $\left<\cdot , \cdot \right>$ on $\mathfrak n$, there is a $D$-homogeneous distance $d$ on $N$ such that
  $d(0,x)=\langle x,x\rangle^{\frac{1}{2\lambda_j}}$ for $x\in V_{\lambda_j}$.     
        During the course of proof of  Theorem  \ref{main-uniform}, we will modify the $D$-inner products   several times  and so also  the corresponding $D$-homogeneous distances.

  It is easy to see that any two $D$-homogeneous distances on $N$ are biLipschitz equivalent.   We will always equip $N$ with a $D$-homogeneous distance. Hence it makes sense to
   speak of  a biLipschitz map of $N$ without specifying the  $D$-homogeneous distance.

        For computational purposes, we also define a function $\rho$ that is biLipschitz equivalent to a $D$-homogeneous
        distance
         $d$.  For any  $D$-inner product $\left<\cdot , \cdot \right>$ on  $\mathfrak n$  define a ``norm'' on $\mathfrak n$ by
                 $$||v||=\sum_i |v_i|^{\frac{1}{\lambda_i}} ,$$
                  where $v=\sum_i v_i$ with $v_i\in V_{\lambda_i}$.
                  Then   define $\rho$ by
                  $\rho(x,y)=||x^{-1}*y||$.  We identify $\mathfrak n$ and $N$.
                   Clearly $\rho$ is left invariant, induces the manifold topology and satisfies
                   $\rho(e^{tD}x, e^{tD}y)=e^t \rho(x, y)$ for all $x, y\in \mathfrak n$ and $t\in \mathbb R$.   It follows that for any $D$-homogeneous  distance $d$ on $\mathfrak n$,  there is a constant $L\ge 1$ such that   $d(x,y)/L\le \rho(x,y)\le L\cdot d(x,y)$ for all $x, y\in \mathfrak n$.
                       The explicit formula for $\rho$ will make the calculations much easier.    We will frequently use
                        $\rho$ rather than $d$ for  estimates.


 \begin{lemma}\label{bilip auto}
 Let $\phi$ be an  automorphism of $N$.
 Then $\phi$ is  biLipschitz  if and only if  $d\phi$ is
  ``layer preserving''; that is, $d\phi(V_{\lambda_j})=V_{\lambda_j}$ for each $j$.
 \end{lemma}

        \begin{proof}
        First suppose $\phi$ is biLipschitz. Let $0\not=v\in V_{\lambda_j}$ and write
        $d\phi(v)=\sum_i x_i$ with $x_i\in V_{\lambda_i}$.
         Then $d\phi(tv)=\sum_i{tx_i}$.
           We have $$\rho(0, tv)= |v|^{\frac{1}{\lambda_j}}|t|^{\frac{1}{\lambda_j}}$$
             and
             $$\rho(0, d\phi(tv))=\sum_i |x_i|^{\frac{1}{\lambda_i}}|t|^{\frac{1}{\lambda_i}}.$$
             The biLipschitz condition implies $x_i=0$ when $i\not=j$ by letting $t\ra \infty$ or $t\ra 0$.

        Conversely assume $d\phi$ is layer preserving.
          Then there is some constant $C\ge 1$ such that
          \begin{equation}           \label{phi(v)}
          |v|/C\le |d\phi(v)|\le C |v|
          \end{equation}
             for all $v\in V_{\lambda_j}$, $\forall j$.  Now let $v\in \mathfrak n$.  Write
          $v=\sum_j v_j$ with $v_j\in V_{\lambda_j}$.  Then
          $d\phi(v)=\sum_j d\phi(v_j)$.
           We have $\rho(0, v)=\sum_j|v_j|^{\frac{1}{\lambda_j}}$ and
            $\rho(0,d \phi(v))=\sum_j|d\phi(v_j)|^{\frac{1}{\lambda_j}}$.
        Now the claim follows from (\ref{phi(v)}).

                 \end{proof}

An automorphism $\phi$  of $N$ is called graded if it satisfies the condition in Lemma
\ref{bilip auto}.

Let $G$ be a connected Lie group with a left invariant distance $d$ that induces the manifold topology, and $H$ a  closed normal  subgroup of $G$.
 We define a distance on $G/H$ by $\bar d(xH, yH)=\inf\{d(xh_1, yh_2)| h_1, h_2\in H\}$.     Then $\bar d$ is a left invariant distance on $G/H$ that induces the manifold topology and the quotient map $(G, d) \to (G/H, \bar d)$ is $1$-Lipschitz.
Since $H$ is normal,   we have $\bar d(xH, yH)=d(xh_1, yH)=d(yh_2, xH)=d_H(xH, yH)$ for any $h_1, h_2\in H$, where
 $d_H$ denotes the Hausdorff distance.
If $F$ is a biLipschitz map of $(G, d)$ that permutes the cosets of $H$, then $F$ induces a   biLipschitz map $\bar F: (G/H, \bar d) \to (G/H, \bar d)$   with the same biLipschitz constant as $F$.

Let $(\mathfrak n, D)$ be a diagonal
   Heintze pair and $d$ a $D$-homogeneous distance on $N$. Assume $\mathfrak w$ is an ideal of $\mathfrak n$ such that
 $D(\mathfrak w)\subset
\mathfrak w$.   Then $D$ induces a derivation $\bar D$ of $\mathfrak n/\mathfrak w$   and  $({\mathfrak n/\mathfrak w}, \bar D)$   is a
 diagonal Heintze pair.
     In this case, the distance $\bar d$ on $N/W$ induced by $d$ is a $\bar D$-homogeneous distance, where $W$
       is the Lie subgroup of $N$ with Lie algebra $\mathfrak w$.

     \subsection{A  fiber Tukia theorem for  diagonal Heintze pairs}\label{fibertukia}

     Here we recall the Tukia-type theorem for Carnot groups and a fiber version of Tukia theorem for  diagonal Heintze pairs, see \cite{DFX} for more details.

      \begin{theorem}\label{tukia} (Theorem 1.1, \cite{DFX})
    Let $N$ be  a Carnot group  and $\hat N=N\cup \{\infty\}$ the one-point compactification of $N$.
      There is    a left invariant Carnot-Caratheodory metric $d_0$ on $N$ with the following property.
    Let $G$ be a uniform  quasiconformal group of   $\hat N$.
     If the action of $G$ on the space of distinct triples of $\hat N$ is co-compact, then
       there is some quasiconformal map $f: \hat N \ra \hat N$
         such that $fGf^{-1}$    consists of conformal  maps  with respect to $d_0$.
   \end{theorem}

   The metric $d_0$ has the largest conformal group  in the sense that  the conformal group of any  left invariant Carnot-Caratheodory metric is conjugated into the conformal group of $d_0$.  In general it is not possible to conjugate a uniform quasiconformal group into the conformal group of an arbitrary  left invariant Carnot-Caratheodory metric, see   Section 6, \cite{DFX}  for an example.

           Next let $(\mathfrak n, D)$ be  a diagonal Heintze pair.
     Then there is a sequence of $D$-invariant Lie sub-algebras
     $\{0\}=\mathfrak n_0\subset \mathfrak n_1\subset \cdots\subset \mathfrak n_s= \mathfrak n$  with the following properties:   each $\mathfrak n_{i-1}$ is an ideal of $\mathfrak n_{i}$  with the quotient
     $\mathfrak n_{i}/{\mathfrak n_{i-1}}$ a Carnot Lie algebra;  $D$ induces a derivation $\bar D:  \mathfrak n_{i}/{\mathfrak n_{i-1}}  \ra
      \mathfrak n_{i}/{\mathfrak n_{i-1}} $  which is a multiple of the Carnot derivation of
       $\mathfrak n_{i}/{\mathfrak n_{i-1}}$.
      Let $N_i$ be the connected Lie subgroup of $N$ with Lie algebra $\mathfrak n_i$.  Then $N/{N_i}$  is a homogeneous manifold and the natural map
       $\pi_i:  N/{N_{i-1}}   \ra N/{N_i}$  is a fiber bundle with fiber the Carnot group $N_i/{N_{i-1}}$.
        We call the sequence of subgroups $0=N_0<N_1<\cdots<N_s=N$  the preserved subgroup sequence.

        Let $d$ be  a $D$-homogeneous distance  on $N$.
       In general  $d$  does not induce any metric on the homogeneous   space $N/{N_i}$
         when $N_i$ is not normal in $N$.  Nonetheless,     it induces a metric on the fibers
        $ N_i/{N_{i-1}}$  of  $\pi_i:  N/{N_{i-1}}   \ra N/{N_i}$.
    Furthermore, every biLipschitz map $F$ of $N$ permutes the cosets of $N_i$ for each $i$.   Hence $F$ induces a   map $F_i: N/{N_i}  \ra N/{N_i}$  and a bundle map of  $\pi_i:  N/{N_{i-1}}   \ra N/{N_i}$.  The restriction of $F_i$ to  the  fibers of $\pi_i$ are biLipschitz maps of the Carnot group   $ N_i/{N_{i-1}}$ in the following sense.      For each $p\in N$, let $F_p=L_{F(p)^{-1}}\circ F\circ L_p$, where  $L_x$ denotes the left translation of $N$ by $x$.   Notice that the map
$(F_p)_{i-1}: N/{N_{i-1}}\ra N/{N_{i-1}}$ satisfies
$(F_p)_{i-1}(N_i/{N_{i-1}})=N_i/{N_{i-1}}$.  The statement above  simply means
  $(F_p)_{i-1}|_{N_i/{N_{i-1}}}:  N_i/{N_{i-1}}\ra  N_i/{N_{i-1}}$  is biLipschitz with respect to any  left invariant Carnot-Caratheodory metric on $N_i/{N_{i-1}}$.

\begin{theorem}\label{foliatedtheorem}  (Theorem 1.3, \cite{DFX})
Let  $(N,D)$ be  a diagonal Heintze pair  and
 $\G$ be a  uniform quasisimilarity group of $N$ that acts cocompactly on the space of distinct pairs of $N$ (or equivalently $\G$ a group that quasi-acts coboundedly on $S=N \rtimes_D \R$).
      Let $I=\{i| 1\le i\le s,  \dim(N_i/{N_{i-1}})\ge 2\}$.
    Then there exists a biLipschitz map $F_0: N \to N$    and   a left invariant
     Carnot-Caratheodory metric  $d_i$ on  $N_i/{N_{i-1}}$ for each $i\in I$
    such  that     for each $p\in N$ and   each  $g\in F_0\G F^{-1}_0$, the map
      $(g_p)_{i-1}|_{N_i/{N_{i-1}}}:  (N_i/{N_{i-1}}, d_i)\ra  (N_i/{N_{i-1}}, d_i) $
          is a similarity.

\end{theorem}

       \section{Quasi-isometric rigidity of  solvable groups}\label{cdiff-sec}

       In this section we combine Tukia-type  theorems with the coarse differentiation method of Eskin-Fisher-Whyte \cite{EFW12}, \cite{EFW13} to establish quasi-isometric rigidity results for solvable groups.

       \subsection{Quasi-isometric rigidity of  lattices in  the isometry group of SOL-like  groups}

One of the main applications of Tukia-type  theorems for nilpotent groups   
  is quasi-isometric rigidity of lattices in the isometry group of SOL-like groups.
  Recall a SOL-like group has the form  $S=(N_1\times N_2)\rtimes \mathbb R$  and admits two foliations by Heintze groups
    $N_1 \rtimes_{D_1} \R$ and $N_2 \rtimes_{D_2} \R$.

    There is a two step outline for proving the quasi-isometric rigidity for  SOL-like groups:
 \begin{enumerate}
 \item Show that up to composition with an isometry any self quasi-isometry of a SOL-like group is bounded distance from a product map $$(x,y,t) \to (f(x),f(y),t)\textrm{ where }x\in N_1,\ y \in N_2,\ t \in \R.$$
 \item Prove a Tukia-type  theorem for groups acting by quasi-similarities on $N_i, \ i=1,2$.\\
 \end{enumerate}
This strategy comes from Eskin-Fisher-Whyte's seminal work on the quasi-isometric rigidity of SOL itself. In their work  \cite{EFW07, EFW12, EFW13} they focus primarily on Step  (1) and develop a new ``coarse differentiation" technique for understanding quasi-isometries of SOL-like groups. For clarity
their papers focus primarily on the case of the three dimensional SOL and its discrete cousins the Diestel-Leader graphs.
 Despite this, their results follow similarly for all SOL-like groups as already pointed out in [EF10, Section 4.4].  Indeed, in her work \cite{P11a,P11b} on abelian-by-abelian higher rank SOL-like groups (i.e groups of the form $\R^n \rtimes \R^m$ with appropriate conditions on the action of $\R^m$)  Peng already points out that Eskin-Fisher-Whyte's arguments work more generally in the case of SOL-like groups where $N_i$ is abelian.  Recently Ferragut in his thesis \cite{Fe22} has endeavored to find the most general framework for which Eskin-Fisher-Whyte's work applies. In doing so he has
 extended the work of \cite{EFW12}  to the class of \emph{horo-pointed metric measure spaces} and is looking to do the same for \cite{EFW13}. 



We will briefly review in the SOL case how combining steps  (1) and  (2) gives quasi-isometric rigidity for lattices in SOL.

{\bf Quasi-isometric rigidity outline for SOL.} Let $\Gamma$ be a finitely generated group quasi-isometric to SOL (hence to any lattice in SOL). Then there is an induced quasi-action of $\Gamma$ on SOL by $L\geq 1,\ C\geq 0$ quasi-isometries where $L,C$ are uniform over all $\gamma \in \Gamma$. Up to replacing $\Gamma$ with a subgroup of index two we can apply Eskin-Fisher-Whyte (i.e. Step   (1) ) to get that all $\gamma \in \Gamma$ act by maps of the form $\gamma(x, y, t) = (f_\gamma(x), g_\gamma(y), t + c_\gamma)$. This induces two uniform quasisimilarity actions on $\R$ (via $f_\gamma$ and $g_\gamma$ respectively). By Step  (2) (in our case by \cite{FM99}) both of these actions can be conjugated to actions by similarities on $\R$. That is,  after conjugation $f_\gamma$ scales distances on $\R$ by some constant $n_\gamma$ and $g_\gamma$  by some constant $m_\gamma$.
Note that these two conjugations together define a quasi-conjugation of the original quasi-action on SOL.  If  $n_\gamma= m_\gamma$ for all $\gamma \in \Gamma$ then the quasi-action is bounded distance from an action by isometries. The other case cannot occur because if $n_{\gamma_0} \neq m_{\gamma_0}$ for even a single $\gamma_0$ then the quasi-isometry constants of the map induced by $\gamma_0^k$ go to infinity as $k \to \infty$ thus violating uniformity of the quasi-action.
Now we have that $\Gamma$ up to finite index and finite kernel acts by isometries on SOL. But SOL  is an index $8$ subgroup of
   its isometry group so  $\Gamma$ must be virtually a lattice in SOL.


For the proof of Theorem \ref{theorem:QIrigid} we also follow this outline however there are several extra considerations that apply in our more general setting. First, our conjugation may only give us an isometric action with respect to a different metric than the one we started with. Second, the isometry group of a more general SOL-like group $S$ may be much larger than $S$ (they do not have the same dimension in general). 

 {\bf Proof of Theorem   \ref{theorem:QIrigid}.}    We now prove Theorem \ref{theorem:QIrigid} in detail.

       Let $S=(N_1\times N_2)\rtimes \mathbb R$ be a SOL-like group  equipped with a left invariant Riemannian metric $g$, where  $(N_i, D_i)$   is Carnot  or Carnot-by-Carnot.   
          Let $\Gamma$ be a finitely generated group quasi-isometric to
            $S$. 
            The quasi-isometry between $\Gamma$ and $S$  induces a quasi-action of $\Gamma$ on $S$. In particular,
         each $\gamma\in \Gamma$ gives rise to a   $(L,C)$-quasi-isometry  $\phi(\gamma): S\rightarrow S$,         where  $L\ge 1$, $C\ge 0$   are constants independent of $\gamma$.
          As explained above,  the arguments of Eskin-Fisher-Whyte  imply  that  $\phi(\gamma)$ is at a bounded distance from a product map.  After  replacing $C$  by a larger constant   and replacing $\phi(\gamma) $ by a product map  we may assume that  $\phi(\gamma)$ is a product map.  After taking an index two subgroup  if necessary we may further assume that each $\phi(\gamma)$ preserves the foliation of $S$ by  cosets of  $S_i=N_i\rtimes_{D_i} \mathbb R$.  Thus for each $\gamma\in \Gamma$, there are maps $\gamma_1: N_1\rightarrow N_1$, $\gamma_2: N_2\rightarrow N_2$, $h_\gamma:
          \mathbb R  \rightarrow \mathbb R$, such that $\phi(\gamma)$ is given by:
      $$\phi(\gamma)(n_1, n_2, t)=(\gamma_1(n_1), \gamma_2(n_2), h_\gamma(t)).$$

             Let
          $\tilde{\gamma}_i:  S_i\rightarrow S_i$ be  given by $\tilde\gamma_i(n_i,t)=
          (\gamma_i(n_i), h_\gamma(t))$. Then $\gamma\mapsto \tilde \gamma_i$ defines a quasi-action of $\Gamma$ on $S_i$.   Since this quasi-action is by height-respecting quasi-isometries,   $\Gamma$ induces a uniform quasisimilarity action on
           $(N_i, d_i)$,
         where $d_i$ is a $D_i$-homogeneous distance on $N_i$.   Due to  the particular form of  $\tilde \gamma_i$  this  induced action of $\Gamma$ on
         $N_i $  
          is given by $\gamma\mapsto \gamma_i$.

          Now we assume  each $(N_i, D_i)$  $(i=1, 2$) is either Carnot or  Carnot-by-Carnot with $\dim(\mathfrak w_i)\ge 2$, where
           $\mathfrak w_i$ is the Lie sub-algebra of  $\mathfrak n_i$ generated by the  eigenspace of the smallest eigenvalue of $D_i$.
          Then by  Theorem \ref{main-uniform} (Carnot-by-Carnot case) or  Theorem \ref{tukia} (Carnot case), there is a  biLipschitz map $f_i$  of $N_i$
              and   a maximally symmetric
        $D_i$-homogeneous distance  $d_i$   on $N_i$   such that $f_i\circ \gamma_i\circ f_i^{-1}$    is a    similarity of  $(N_i, d_i)$   for each $\gamma\in \Gamma$.
        \cite{KLD17} Theorem 1.2  implies that  $f_i\circ \gamma_i\circ f_i^{-1}$
        has the form $f_i\circ \gamma_i\circ f_i^{-1}=L_{a_i}\circ e^{v_{i, \gamma}D_i}\circ \phi_i$  for some $a_i\in N_i$, where $\phi_i$ is an  automorphism of $N_i$
          and is also an isometry of $(N_i, d_i)$.   .





       \begin{lemma} \label{sol-compatible}
       The equality $v_{1, \gamma}+v_{2,\gamma}=0$ holds for all $\gamma\in \Gamma$.
       \end{lemma}

       \begin{proof}
       Note that the left translation $L_{(0, -v_{1, \gamma})}: S_1\rightarrow S_1$  (which is an isometry of $S_1$)
          is given by
       $L_{(0, -v_{1, \gamma})}(x,t)=(e^{-v_{1, \gamma}D_1}x, t-v_{1, \gamma})$.
       It follows that $L_{(0, -v_{1, \gamma})}\circ  \tilde \gamma_1$  is a $(L, C)$-quasi-isometry and  is given by
          $L_{(0, -v_{1, \gamma})}  \circ \tilde \gamma_1 (x, t)=(e^{-v_{1, \gamma}D_1}\gamma_1(x),
           h_\gamma(t)-v_{1, \gamma})$.
       Since the boundary map $e^{-v_{1, \gamma}D_1}\circ \gamma_1$  of
        $L_{(0, -v_{1, \gamma})}\circ  \tilde \gamma_1$ is
              an isometry, Lemma 5.1  in \cite{SX12}
               implies $| h_\gamma(t)-v_{1, \gamma}-t|\le C_1$ for a constant $C_1$ that depends only on
           $L,C$ and $N_1$.   Similarly by considering $L_{(0, v_{2, \gamma})}\circ  \tilde \gamma_2$ we get
            $| h_\gamma(t)+v_{2, \gamma}-t|\le C_2$ for a constant $C_2$ that depends only on
           $L,C$ and $N_2$.  Combining these two inequalities  we get
            $|v_{1, \gamma}+v_{2, \gamma}|\le C_1+C_2$ for all $\gamma\in \Gamma$.
             Since $v_{i, \gamma^n}=n v_{i, \gamma}$ for any $n\ge 1 $,  the above inequality applied to $\gamma^n$ implies
             $|v_{1, \gamma}+v_{2, \gamma}|\le (C_1+C_2)/n$.
              Since this is true for all $n\ge 1$, the lemma follows.

       \end{proof}

       \begin{lemma}\label{distance-sol}
       Let $f_i: (N_i, d_i)\rightarrow (N_i, d_i)$ ($i=1,2$) be  a biLipschitz map.  Let $d$ be a  left invariant Riemannian metric on $S$ such that $N_1$, $N_2$ and $\mathbb R$ are perpendicular to each other.
        Define $F:(S, d)\rightarrow (S, d)$  by $F(n_1, n_2, t)=(f_1(n_1), f_2(n_2), t)$. Then
         $F$ is a $(1,\tilde C)$-quasi-isometry for some constant $\tilde C\ge 0$.

       \end{lemma}

       \begin{proof}
       After rescaling the metric we may assume  that vertical lines $c_{(n_1, n_2)}:
       \mathbb R\rightarrow S$, $c_{(n_1, n_2)}(t)=(n_1, n_2, t)$,
         are unit speed geodesics.   Let $\pi_1: S\rightarrow S_1$ be given by $\pi_1(n_1, n_2, t)=(n_1, t)$ and  $\pi_2: S\rightarrow S_2$ be  given by  $\pi_2(n_1, n_2, t)=(n_2, t)$.  Also let $h: S\rightarrow \mathbb R$ be the height function
          $h(n_1, n_2, t)=t$.
           By     { Corollary 4.13 of  \cite{Fe20}   or  }  Theorem 4.1  in \cite{LDPX22}
           there is a constant $C_1\ge 0$ such that for any $p,q\in S$ we have
          $$|d(p,q)-d^{(1)}(\pi_1(p), \pi_1(q))-d^{(2)}(\pi_2(p), \pi_2(q))+|h(p)-h(q)||\le C_1,$$
            where $d^{(1)}$ is the metric on $S_1\sim (N_1\times\{0\})\rtimes \mathbb R  \subset S$ induced by $d$ and similarly for $d^{(2)}$.  Replacing $p,q$ with $F(p)$ and $F(q)$ respectively we  get
         $$|d(F(p), F(q))-d^{(1)}(\pi_1(F(p)), \pi_1(F(q)))-d^{(2)}(\pi_2(F(p)), \pi_2(F(q)))+|h(F(p))-h(F(q))||\le C_1.$$   The lemma follows from the following claim since by the definition of $F$ we have $h(F(x))=h(x)$ for any $x\in S$.

          Claim: There is a constant $\tilde C_i$ depending only on  the Gromov hyperbolicity constant of $S_i$ and the biLipschitz constant of $f_i$    such that  for all $p_1,  p_2\in S$:
           $|d^{(i)}(\pi_i(F(p_1)), \pi_i(F(p_2)))-d^{(i)}(\pi_i(p_1), \pi_i(p_2))|\le \tilde C_i$.

      Proof of the claim:  we will only consider the case $i=1$ as the case $i=2$ is similar.
       Let $p_1=(x_1, y_1, t_1)$, $p_2=(x_2, y_2, t_2)$. Then  the claim takes the form
        $$|d^{(1)}(f_1(x_1), t_1), (f_1(x_2), t_2))-d^{(1)}((x_1, t_1), (x_2, t_2))|\le \tilde C_1.$$
         Let $t_{x_1, x_2}$ be the height at which the two vertical geodesics $c_{x_1}$ and $c_{x_2}$ in $S_1$ diverge from each other (that is, the distance between $c_{x_1}(t_{x_1, x_2})$ and $c_{x_2}(t_{x_1, x_2})$ is $1$),    where for $x\in N_1$, $c_x: \mathbb R\rightarrow S_1$ is given by $c_x(t)=(x,t)$.
           Then we have
           $$|d^{(1)}((x_1, t_1), (x_2, t_2))-(t_{x_1, x_2}-t_1)-
           (t_{x_1, x_2}-t_2)|\le  C_2$$  if  $t_{x_1, x_2}>\max\{t_1, t_2\}$  and
           $|d^{(1)}((x_1, t_1), (x_2, t_2))-|t_1-t_2||\le C_2$      otherwise,
            for some constant $C_2$ depending only on     the Gromov hyperbolicity constant
           of      $S_1$.   Similarly
       $$|d^{(1)}((f_1(x_1), t_1), (f_1(x_2), t_2))-(t_{f_1(x_1), f_1(x_2)}-t_1)-
           (t_{f_1(x_1),  f_1(x_2)}-t_2)|\le C_2$$ if    $t_{f_1(x_1),  f_1(x_2)}>\max\{t_1, t_2\}$  and
           $|d^{(1)}((f_1(x_1), t_1), (f_1(x_2), t_2))-|t_1-t_2||\le C_2$   otherwise.
            It now suffices to show that there is a constant $C_3$ such that $|t_{f_1(x_1),  f_1(x_2)}-t_{x_1, x_2}|\le C_3$ for all $x_1, x_2\in N_1$.
         This follows from the  fact that $d_1(x_1, x_2)$ is  comparable with $e^{t_{x_1, x_2}}$    and  that $f_1$ is biLipschitz.

       \end{proof}

       {\bf{Completing the proof of Theorem \ref{theorem:QIrigid}}}.  
        Recall  that the    $D_i$-homogeneous distance  $d_i$   on $N_i$ is associated to   an inner product    $\langle,\rangle_i$  on $\mathfrak n_i$.
        Since  $\phi_i$ is an  automorphism of $N_i$ and also an isometry of $(N_i, d_i)$,
           Lemma \ref{bilip auto}  implies that   $d\phi_i$ is layer-preserving  and is an orthogonal transformation with respect to $\langle,\rangle_i$.
          Let $\langle,\rangle_0$ be the inner product on $T_e S=(\mathfrak n_1\times \mathfrak n_2)\rtimes \mathbb R$  that agrees with $\langle,\rangle_i$ on $\mathfrak n_i$, that satisfies $\langle(0,0,1), (0,0,1)\rangle_0=1$,   and such that
          $\mathfrak n_1$, $\mathfrak n_2$ and $\mathbb R$  are perpendicular to each other.
         Let $g_0$ be the  left invariant  Riemannian metric on $S$ determined by  $\langle,\rangle_0$.

          For each $\gamma\in \Gamma$, define a
           map $\Psi(\gamma):  S\rightarrow S$   by
         $$\Psi(\gamma)(n_1, n_2, t)= (f_1\circ \gamma_1\circ f_1^{-1}(n_1),  f_2\circ \gamma_2\circ f_2^{-1}(n_2), t+v_{1,\gamma}).$$   Since
          $f_i\circ \gamma_i\circ f_i^{-1}=L_{a_i}\circ e^{v_{i, \gamma}D_i}\circ \phi_i$,
           Lemma \ref{sol-compatible}    implies
         $\Psi(\gamma)(n_1, n_2, t)= L_{(a_1, a_2, v_{1, \gamma})}(\phi_1 n_1, \phi_2 n_2, t)$.
            The properties of $\phi_i$ imply that the map $S\rightarrow S$, $(n_1, n_2, t)\mapsto (\phi_1 n_1, \phi_2n_2,t)$ is an automorphism of $S$ and is an isometry
            of $(S,g_0)$.  It follows that  $\Psi(\gamma)$ is an isometry of $(S,g_0)$.

          Define $F: S\rightarrow S$  by
       $F(n_1, n_2, t)=(f_1(n_1), f_2(n_2), t)$ as in Lemma \ref{distance-sol}.   Then $F$ is a
          quasi-isometry of $S$.
         Notice that  $\Psi(\gamma)$   and  $F\circ \phi(\gamma)\circ F^{-1}$
           induce the same boundary map of $S_i$ . 
           It follows that $\Psi(\gamma)$ and  $F\circ \phi(\gamma)\circ F^{-1}$ are at a bounded distance from each other.   By replacing $F\circ \phi(\gamma)\circ F^{-1}$  with $\Psi(\gamma)$   we see that
          the original quasi-action of $\Gamma$ on $S$ is now quasi-conjugated to an isometric action of $\Gamma$ on $(S,g_0)$.  Since the original quasi-action is cobounded, the isometric action is cocompact and so $\Psi(\Gamma)$ is a uniform
             lattice in $\text{Isom}(S,g_0)$.  This finishes the proof when each $(N_i, D_i)$  $(i=1, 2$) is either Carnot or  Carnot-by-Carnot with $\dim(\mathfrak w_i)\ge 2$.

             If  $\dim(\mathfrak w_i)=1$ for at least one $i$ and $\Gamma$ is  amenable, then  Theorem \ref{main-uniform}   still applies and the above proof works.
             Now assume $\dim(\mathfrak w_i)=1$   and $\text{Isom}(S,g)$ admits a
  uniform  lattice   $\Gamma_0$   for some  left invariant Riemannian metric  $g$ on $S$.
    By Theorem 3 in \cite{W00}
   $\Gamma_0$ is virtually a  lattice in some simply connected solvable Lie group.  In particular, $\Gamma_0$ is a  finitely generated  amenable  group.
    It follows that $\Gamma$ is also amenable, being quasi-isometric to $\Gamma_0$.   Hence again  Theorem \ref{main-uniform}    applies and the above proof works.


         \qed

 \subsection{Quasi-isometric classification of a  class of solvable Lie groups}

  Here we provide the proofs of Theorem \ref{QIclassification}  and Theorem \ref{QIC-CarnotbyCarnot}.

 {\bf{Proof of Theorem \ref{QIclassification}}}.
 One direction is clear.  For the other direction, we assume $S_1,  S_2\in \mathcal S$   are quasi-isometric.  Then $S_1$ quasi-acts on $S_2$. Since $S_1$ is amenable,
 the proof of Theorem \ref{theorem:QIrigid}  applies (for $\Gamma:=S_1$, $S:=S_2$)
   and  we conclude that
   there is a left invariant Riemannian metric  $g_2$ on $S_2$ such that  the quasi-action of $S_1$ on $S_2$ is quasi-conjugate to an isometric action of $S_1$ on $(S_2, g_2)$.
    This induces a continuous homomorphism $\phi:  S_1\rightarrow \text{Isom}(S_2, g_2)$. Since a continuous homomorphism  between Lie groups is a Lie group homomorphism, we see that the isometric action of  $S_1$ on $(S_2, g_2)$ is a smooth action; that is, the corresponding map $F:  S_1\times S_2\rightarrow S_2$,  $F(s_1, s_2)=\phi(s_1)(s_2)$,   is a smooth map.   Fix some $x\in S_2$.
   Then  the map $f: S_1\rightarrow S_2$ given by $f(s_1)=F(s_1, x)=\phi(s_1)(x)$ is smooth.  It is easy to see that   $\phi(s_1') \circ f=f\circ L_{s_1'}$ for all $s_1'\in S_1$, where $ L_{s_1'}$   is the left translation of $S_1$ by $s_1'$.   This implies the
      differentials of the  map $f$ have constant rank. If this  rank  is less than $\text{dim}(S_1)$, then there exists  $h\in f^{-1}(x)\backslash\{e\}$.  Then all the powers $h^n$, $n\in \mathbb Z$, would fix $x\in S_2$.  This is a contradiction since
     $\{h^n|n\in \mathbb Z\}$ is unbounded in $S_1$, $\{\phi(h^n)(x)|n\in \mathbb Z\}=\{x\}$ is bounded in $S_2$, and the action of $S_1$ on $S_2$ is quasi-conjugate to the action of $S_1$ on itself by left translations.  Hence the rank of $f$ is at least $\text{dim}(S_1)$, which implies $\text{dim}(S_2)\ge \text{dim}(S_1)$.  By switching the roles of $S_1$ and $S_2$ we get the reverse inequality and so $\text{dim}(S_2)= \text{dim}(S_1)$.
     Hence $f$ is a local diffeomorphism from $S_1$ to $S_2$. The above argument also shows $f$ is injective  and so is a  diffeomorphism onto an open subset of $S_2$.  Since the image of $f$ is an orbit of the $S_1$ action on $S_2$, we see that $f$ must be surjective since otherwise $S_2$ would be a  disjoint union of  at least two open subsets.
      Consequently $f$ is a diffeomorphism  from $S_1$ to $S_2$.

       We next use an idea we found in   \cite{FLD21}, Proposition 2.1.
       Let $g_1$ be  the pullback Riemannian metric of $g_2$  by $f$. Then $f: (S_1, g_1)\rightarrow (S_2, g_2)$ is an isometry.   Finally the equality   $\phi(s_1') \circ f=f\circ L_{s_1'}$    implies
         $L_{s_1'}: (S_1, g_1)\rightarrow (S_1, g_1)$ is an isometry for any $s_1'\in S_1$. In other words, $g_1$ is left invariant.     Hence $S_1$ and $S_2$ can be made isometric.
          Since they are both of real type,    Theorem 4.21 in \cite{CKLGO21}
             implies they are isomorphic.

\qed

 By using Theorem \ref{NCHM} (proved next) and Corollary  1.2, \cite{DFX}
 instead of Theorem \ref{theorem:QIrigid}
    the above  argument yields  Theorem \ref{QIC-CarnotbyCarnot}.



\subsection{Quasi-isometric rigidity of quasi-actions on  certain Heintze groups}

      Finally we supply the proof of Theorem \ref{NCHM}.

     \noindent{\bf{Proof of Theorem \ref{NCHM}}}.
       Let $(N, D)$ be  Carnot-by-Carnot,  and $\Gamma$  a   group that quasi-acts coboundedly  on $N\rtimes_D \R$.
        We  further assume $\Gamma$ is amenable when $\dim(W)=1$.
         Then $\Gamma$ induces a    uniform quasisimilarity action  on $N$  such that every point in $N$ is a
       radial limit point (since the induced action of $\Gamma$ on the space of distinct pairs of $N$ is cocompact).  By   Theorem \ref{main-uniform}, there is a
         biLipschitz map $F_0$ of $N$ such that
         $F_0\Gamma F_0^{-1}\subset \text{Sim}(N, d_0)$,  where $d_0$ is a fixed
        maximally symmetric   $D$-homogeneous distance  on $N$.
         Recall  that  the  $D$-homogeneous distance  $d_0$ is   associated with  an inner product
          $\langle,\rangle_0$   on
          $\mathfrak n $.    
           Let
          $\langle,\rangle$ be the inner product on $T_e(N\rtimes_D \R)=\mathfrak n\times \mathbb R$ satisfying $\langle,\rangle|_{\mathfrak n \times \mathfrak n}=\langle,\rangle_0$, $\langle \mathfrak n, \{0\}\times \mathbb R\rangle=0$ and $\langle (0,1),(0,1)\rangle=1$.
            Let $g_0$ be the left invariant Riemannian metric on $N\rtimes_D \R$ determined by the inner product $\langle,\rangle$.  Let $\gamma\in \Gamma$.
            Recall that  $F_0\gamma F_0^{-1}$  acts on $N$ as an affine map and has the form  $F_0\gamma F_0^{-1}=L_{a_\gamma}\circ e^{t_\gamma D}\circ \phi_\gamma$ for some  $a_\gamma\in N$,   $t_\gamma\in \mathbb R$, where
            $L_{a_\gamma}$ is the left translation of $N$ by $a_\gamma$  and
            $\phi_\gamma$ is a
             graded automorphism such that $d\phi_\gamma$ is a linear isometry of
              $(\mathfrak n, \langle,\rangle_0)$.
                Now let $F_0\gamma F_0^{-1}$  act on  $N\rtimes_D \R$ by
             $F_0\gamma F_0^{-1}(n, t)=L_{(a_\gamma, t_\gamma)}(\phi_\gamma(n),t)$, where   $L_{(a_\gamma, t_\gamma)}$ is the left translation of $N\rtimes_D \R$  by
               $(a_\gamma, t_\gamma)$.
              It is easy to check that this defines an isometric action of $\Gamma$ on $(N\rtimes_D \R, g_0)$  that induces the given
                action   of     $F_0\Gamma F_0^{-1}$  on $N=\partial S\backslash\{\infty\}$.

      \qed

  \section{BiLipschitz maps of diagonal   Heintze pairs}\label{compatible}

       In this section we    study   individual  biLipschitz maps of   diagonal Heintze pairs (see section \ref{deri}).  We show that if a   biLipschitz map permutes the cosets of a connected graded normal  subgroup, then it has an expression (we call it  compatible expression) with some  nice properties, see Lemma \ref{sj}.



      For any  biLipschitz  map  $F: N\rightarrow N $
          and  $g\in N$, we  denote
          by $F_g:=  L_{F(g)^{-1}}\circ F\circ L_g$.    Note that $F_g(0)=0$.

       \begin{lemma}  \label{onfiber}
         Let $(\mathfrak n, D)$ be a  diagonal Heintze pair and $\mathfrak w $
            an ideal of $\mathfrak n$ such that $D(\mathfrak w)=\mathfrak w$.
         Let   $F: N\rightarrow N $   be  a  biLipschitz  map   
            that permutes the cosets of $W$, where $W$ is the connected Lie subgroup of $N$ with Lie algebra
           $\mathfrak w$.        In addition  assume for every $g\in N$, the map
       $F_g|_W:  W \to    W$ is an automorphism of $W$.    Then
       there is an automorphism $\phi: W\rightarrow W$
         such that    $d\phi: \mathfrak w\ra \mathfrak w$ is layer-preserving,
           $F(gw)=F(g) \phi(w)$ and
                             $(\chi_{{F_0(g)}}|_W)    \circ \phi=\phi\circ (\chi_{g}|_W)$  hold
                                for any $g\in N$, where $\chi_x$ denotes the conjugation by $x$.

       \end{lemma}

       \begin{remark}
           If we start with a uniform quasisimilarity group  $\Gamma$  of a  Carnot-by-Carnot group $N$    and $W$ is the connected Lie subgroup of $N$ with Lie algebra generated by the eigenspace of the smallest eigenvalue of $D$,
     we       can  apply the  fiber  Tukia  theorem (Theorem \ref{foliatedtheorem})    when $\dim(W)\ge 2$  such that
     every element in       a  biLipschitz  conjugate of $\Gamma$
     satisfies the assumption in the lemma. 
       \end{remark}


       \begin{proof}  By replacing $F$ with $F_0$ we may assume $F(0)=0$.
       Since $W$ is normal in $N$, $gW=Wg$ for any $g\in N$.
        The fact that $F$ permutes the cosets of $W$ implies that
            $F(gW)=F(g)W$ and $F(Wg)=W F(g)$.
          Hence  there are two functions
            $\phi_{ g}, \psi_{ g}: W\rightarrow W$ such that
              $F(gw)=F(g) \phi_g(w)$ and $F(wg)=\psi_g (w) F(g)$  for any $w\in W$.
       Notice that $\phi_g=F_g|_W$ and so    
       by assumption   
       is an automorphism of $W$.
          By definition $\psi_g(w) =F(wg)F(g)^{-1}=F(g\cdot g^{-1}wg)F(g)^{-1}
                    =F(g)\phi_g(g^{-1}wg)F(g)^{-1}$. It follows that
                    \begin{equation}\label{bgag}
                    \psi_g=(\chi_{F(g)}|_W)\circ \phi_g\circ (\chi_{g^{-1}}|_{W})
                    \end{equation}
                     is a composition of automorphisms of $W$  and so  is an automorphism of $W$.  Here we used the assumption that $W$ is normal in $N$ and so $\chi_x|_W$ is an automorphism of $W$ for any $x\in N$.


           We claim that           $\psi:=\psi_g$ is independent of  $g$.
                    Let $g_1,  g_2\in N$.  Then for any $w\in W$, $d(wg_1, wg_2)=
                     d(g_1, g_2)$. Since $F$ is $L$-biLipschitz for some $L\ge 1$,
                       we  have $d(F(wg_1), F(wg_2))\le L d(g_1, g_2)$.
                          Since $F(wg_i)=\psi_{g_i}(w) F(g_i)$,
                           we get
                            \begin{align*}
                            d(\psi_{g_1}(w), \psi_{g_2}(w))  & \le d(\psi_{g_1}(w), F(wg_1))
                            +d(F(wg_1), F(wg_2))+ d(F(wg_2), \psi_{g_2}(w))\\
                              &  \le d(0, F(g_1))+
                            L d(g_1, g_2)+d(F(g_2), 0)
                            \end{align*}
                             for all $w\in W$.     By Lemma \ref{irine}   we have  $\psi_{g_1}=\psi_{g_2}$.

                        We now    show that     $\phi:=\phi_g$   is also  independent of  $g$.
            Since $D(\mathfrak w)=\mathfrak w$,   the restriction of  a $D$-homogeneous distance to  $W$     is a $D|_{\mathfrak w}$-homogeneous distance on $W$.
            As the restriction of  a biLipschitz map, the automorphism  $\phi_g$ is a  biLipschitz map of $W$. 
               By Lemma \ref{bilip auto},
             $d\phi_g: \mathfrak w\ra \mathfrak w$ is layer preserving.    Write $\mathfrak w=W_{\mu_1}\oplus \cdots\oplus W_{\mu_m}$
              as the direct sum of eigenspaces of
             $D|_{\mathfrak w}$, where $0<\mu_1<\cdots< \mu_m$ are the distant eigenvalues of  $D|_{\mathfrak w}$.   Then  $d\phi_g(W_{\mu_j})=W_{\mu_j}$.
             Let $g, h\in N$.  We shall show that  for each $j$ the equality
              $d\phi_g|_{W_{\mu_j}}=d\phi_h|_{W_{\mu_j}}$  holds. To see this,
                notice that for any $i$ and $x\in N$,   if we denote  $ I_i=   \oplus_{k\ge i}W_{\mu_k}$,    then   (\ref{conjugationformula}) implies
                 $d\chi_x(I_i)\subset I_i$  
                  and  that   $d\chi_x$
                   induces the identity  map on $I_j/{I_{j+1}}\cong W_{\mu_j}$.  Hence   the map on $I_j/{I_{j+1}}$  induced by
                $d(\chi_{F(g)}|_W)\circ d\phi_g\circ d(\chi_{g^{-1}}|_{W})$  equals
                 $d\phi_g|_{W_{\mu_j}}$.  The same is true when $g$ is replaced by $h$.
                  Now  (\ref{bgag})  and the fact that $d\psi_g=d\psi_h$ implies
                   $d\phi_g|_{W_{\mu_j}}=d\phi_h|_{W_{\mu_j}}$.



                      Finally,          by picking $g=0$ in
                             (\ref{bgag})    we obtain $\phi=\psi$  as $F(0)=0$. 
                                The  equality   $(\chi_{{F_0(g)}}|_W)    \circ \phi=\phi\circ (\chi_{g}|_W)$ for any $g\in N$  also follows from     (\ref{bgag}).


     \end{proof}

                  Below when we say $F:\mathfrak n\ra\mathfrak n$ is a biLipschitz map,  we identify $\mathfrak n$ with $N$ and $\mathfrak n$ is equipped with a $D$-homogeneous distance.

               A self homeomorphism $F: G\to G$ of  a Lie group is called an affine map if
                $F=L_g\circ \phi$, where $\phi$ is an  automorphism  of $G$    and
                    $L_g$ is  left translation by $g\in G$.    We say $\phi$ is the automorphism part of  the affine map $F$.

                    Let $(\mathfrak n, D)$ and $\mathfrak w$ be as in  Lemma  \ref{onfiber}.  Denote by $\bar D: {\mathfrak n}/{\mathfrak w}\ra {\mathfrak n}/{\mathfrak w}$   the derivation of ${\mathfrak n}/{\mathfrak w}$ induced by $D$  and $\sigma(\bar D)$ the set of eigenvalues of $\bar D$.
                       Let $H\subset \mathfrak n$ be a graded subspace of $\mathfrak n$ complementary to $\mathfrak w$; that is, for each $\lambda\in \sigma(\bar D)$,
                   $H_\lambda\subset V_\lambda$ is a complementary linear subspace of
                    $W_\lambda$  in  $ V_\lambda$   (if $W_\lambda=\{0\}$, then  $H_\lambda= V_\lambda$), and $H=\oplus_{\lambda\in \sigma(\bar D)} H_{\lambda}$.
                  Notice  that for every  $g\in \mathfrak n$, there are unique
                     $h\in H$,  $w \in \mathfrak w$ such that $g=h*w$.



                  For $x\in \mathfrak n$, we use $\bar x$ to denote $\pi(x)$, where
                   $\pi: \mathfrak n\to \mathfrak n/\mathfrak w$ is the quotient map.
               For convenience, we introduce  the following terminology.

                  \begin{definition}\label{defcompatible}
                  Let $(\mathfrak n, D)$ be a  diagonal Heintze pair and $\mathfrak w $
            an ideal of $\mathfrak n$ such that $D(\mathfrak w)=\mathfrak w$.
         Let   $F: N\rightarrow N $   be  a  biLipschitz  map   that permutes the cosets of $W$, where $W$ is the connected Lie subgroup of $N$ with Lie algebra
           $\mathfrak w$.        In addition  assume that
            $F$ induces an affine map of $N/W$ and that
           for every $g\in N$, the map
       $F_g|_W:  W \to    W$ is an automorphism of $W$.
                    Let $\bar B$ be the automorphism part of the affine map of $N/W$ induced by $F$.
                    We say an expression
                    $F(h*w)=F(0)*Bh*Aw*A s(\bar h)$ ($h\in H$, $w\in \mathfrak w$),
                     with  $A=d\phi$, where $\phi$  is the automorphism   of $W$ from Lemma \ref{onfiber},
                      is a compatible expression of $F$ if the following hold:\newline
                     1. $B: H\ra \mathfrak n$ is a linear map satisfying $B(H_\lambda)\subset V_{\lambda}$ and $d\bar B\circ \pi|_H=\pi\circ B$,  where $\pi: \mathfrak n\ra \bar {\mathfrak n}$  is  the   quotient map;\newline
                      2.   $[Bh, Aw]=A[h,w]$ for any $h\in H$, $w\in \mathfrak w$;\newline
                      3.    $s$ is  a map $\mathfrak n/\mathfrak w\ra Z(\mathfrak w)$, where $Z(\mathfrak w)$ denotes the center of $\mathfrak w$.

                  \end{definition}

   We shall show (see Lemma \ref{drop5}) that conditions 1-2 above imply condition 3.  But to prove this we need the  existence
          of           a  compatible expression (Lemma \ref{sj}).

                In general a biLipschitz map $F$ does not have a unique  compatible expression.  The reason is that there might be more than  one linear map
                 $B: H\ra \mathfrak n$   satisfying Conditions 1 and 2 above. 
                    As a result, the map $s$ is also not unique.   However,
                    if  $\alpha$ denotes  the smallest eigenvalue of $\bar D$
                   and  $\pi_\lambda: \mathfrak w\ra W_\lambda$ is the projection with respect to the decomposition
                    $\mathfrak w=\oplus_{\lambda}W_\lambda$,  then
                    $\pi_\lambda\circ s$ for $ \lambda< \alpha$ is unique
                 since a change in $B$ only affects  $\pi_\lambda\circ s$ for $\lambda \ge \alpha$. 
                     Another way to see this is to notice  $(\pi_\lambda\circ s)(\bar h)=A^{-1}(\pi_\lambda(F(0)^{-1}*F(h)))$ for $ \lambda<\alpha$.


                  \begin{lemma}\label{sj}
                   Let $(\mathfrak n, D)$, $\mathfrak w$  and $F$ be as in Definition \ref{defcompatible}.
             Then $F$ has  a  compatible expression.

             \end{lemma}


                The proof of  Lemma \ref{sj} is tedious.
                  To improve readability we have placed the proof of Lemma \ref{sj} in   Appendix \ref{proof of compatible}.
                The main ingredient in the proof is the fact that
                  $\phi\circ \chi_g=\chi_{G(g)}\circ \phi$, where $G=F_0$.  See Lemma \ref{onfiber}.    We also  implicitly (and repeatedly) use the fact that
                  $[Z(\mathfrak w), \mathfrak n]\subset Z(\mathfrak w)$.   This  follows from  the Jacobi identity  and the fact that $\mathfrak w$ is an ideal of $\mathfrak n$.


          Notice that the order of  appearance  of  $Aw$ and $A s(\bar h)$  in (1) and (2) below are different.

          \begin{lemma}\label{drop5}
           Let $(\mathfrak n, D)$, $\mathfrak w$  and $F$ be as in Definition \ref{defcompatible}.
           Using the notation  $p=h*w \in \mathfrak{n}$ where $h\in H$, $w\in \mathfrak w$:
           \newline
          (1)    If
          an expression
                    $F(h*w)=F(0)*Bh*Aw*A s(\bar h)$
                     for $F$ satisfies conditions 1-2  of Definition \ref{defcompatible}, then it also satisfies  condition 3;\newline
                       (2)    If
          an expression
                    $F(h*w)=F(0)*Bh*A s(\bar h)*Aw$
                     for $F$ satisfies conditions 1-2  of Definition \ref{defcompatible}, then it also satisfies  condition 3 ($s(\bar h)\in Z(\mathfrak w)$) and so  $F(h*w)=F(0)*Bh*Aw* As(\bar h)$     is a    compatible expression.

          \end{lemma}

          \begin{proof}
            (1)  Let   $F(h*w)=F(0)*B_1h*Aw*A s_1(\bar h)$   be   a compatible expression.
           Condition 1 implies $B_1h-Bh\in \mathfrak w$ for all $h\in H$.
          From condition 2 we have
           $[Bh, Aw]=A[h,w]=[B_1 h, Aw]$ for all $h\in H$ and all $w\in \mathfrak w$.
            It follows that $[B_1h-Bh, Aw]=0$  and so $s'(h):=B_1h-Bh\in Z(\mathfrak w)$.
                   By using  the  BCH formula and the fact that $[\mathfrak n,  Z(\mathfrak w)]\subset  Z(\mathfrak w) $   we get
              $(-Bh)*B_1h\in  Z(\mathfrak w)$.  Finally from the two expressions for $F(h)$ we get  $As(\bar h)=(-Bh)*B_1h*As_1(\bar h)\in  Z(\mathfrak w)$.

              (2)   The same  proof as above also shows  in this case
               $s(\bar h)\in Z(\mathfrak w)$. This allows us  to switch
               $Aw$ and $As(\bar h)$  to get a compatible expression.

          \end{proof}

                   We observe that if $[Bh, Aw]=A[h,w]$ for all $h\in H$, $w\in \mathfrak w$, then 
                          \begin{equation}\label{cc}
                          Bh* Aw* (Bh)^{-1}=\sum_{i=0}^\infty \frac{1}{i!}(\text{ad}\,(Bh))^i(Aw)
                          =A(\sum_{i=0}^\infty \frac{1}{i!} (\text{ad}\,(h))^i(w))=A(h*w*h^{-1}).
                          \end{equation}


                  For later use (in Lemma \ref{fisauto}), we record the following lemma. The main point of the lemma is that the same linear map $B$ works for $F_p$.

                  \begin{lemma}\label{same B}
                  Let $F: \mathfrak n\ra \mathfrak n$ be  a biLipschitz map with a compatible expression
                   $F(h*w)=F(0)*Bh*Aw*A s(\bar h)$.  Then for any $p\in N$, the map $F_p$ admits a compatible expression of the form $F_p(h*w)=Bh*Aw*A \tilde{s}(\bar h)$  for some map  $\tilde{s}: \mathfrak n/\mathfrak w \ra Z(\mathfrak w)$.

                  \end{lemma}

                  \begin{proof}
                  Note that   $(F_p)_q=F_{pq}$   for any $p, q\in N$ so the same $A$ works for $F_p$. By Lemma \ref{onfiber}
                      we can write
                   $F_p(h*w)=F_p(h)*Aw$.  Since $F_p$ induces the automorphism $\bar B$ on $N/W$, we can write  $F_p(h)= Bh*A\tilde s(\bar h)$ for some map
                    $\tilde s: \bar{\mathfrak n}\ra \mathfrak w$. So we have $F_p(h*w)=Bh*A\tilde s(\bar h)*Aw$ with Conditions 1 and 2 of  Definition \ref{defcompatible}    satisfied.
                     By  Lemma \ref{drop5} (2)
                  $F_p(h*w)=Bh*Aw*A \tilde{s}(\bar h)$    is a compatible expression for $F_p$.


                \end{proof}

                  \section{ Characterization of biLipschitz shear maps}\label{shear-bilipschitz}

           {\bf{In Sections \ref{shear-bilipschitz}--\ref{cstructure}   and Appendix 
           \ref{proof of compatible} we shall often   identify a simply connected nilpotent Lie group with 
                  its Lie algebra  via the exponential map.}}  (See Section \ref{nilpprelim}.)

        From now on, unless stated otherwise,    $(N,D)$  will be  Carnot-by-Carnot and $W$ will be the Lie subgroup of $N$ with Lie algebra  $\mathfrak w$ generated by the  eigenspace of the smallest eigenvalue of $D$.  
        In this section we will  characterize     continuous  maps $s: N/W\ra  W$ such that the shear    maps
      $F: (N, d) \ra (N, d) $ given by $F(g)=g\,  s(gW) $       are 
          biLipschitz.    BiLipschitz  shear maps will be used to conjugate a uniform quasisimilarity group into a similarity group.   We are only interested in the case $s(0)=0$. In this case, $F(0)=0$.  
           We first introduce a useful differential one form on Carnot groups, 
               which will  be used in  the 
           characterization  of biLipschitz shear maps.

       \subsection{Horizontal tautological one form on Carnot groups}

        We first recall the tautological one form on  Lie groups and then define the horizontal tautological one form on Carnot groups. 
         Finally  we give an expression  for the horizontal tautological one form  in exponential coordinates.
       
         For  a Lie group $G$ with Lie algebra $\mathfrak g=T_eG$, the tautological one form  $\theta$ is a $\mathfrak g$ valued left invariant one form on $G$ defined as follows.    For each $x\in G$,  $\theta_x: T_xG\ra \mathfrak g$ is the linear map given by
          $\theta_x(X)=\tilde X_e$, where $\tilde X$ is the left invariant vector field on $G$ satisfying $\tilde X_x=X$ and $\tilde X_g$  denotes  the value of $\tilde  X$ at $g\in G$.  The name  ``tautological'' comes from the fact that at each $x\in G$,  $\theta_x$ is the identity map if one identifies the tangent vectors  with the corresponding  left invariant vector fields.
          
         The  horizontal tautological one   form is a counterpart of tautological one  form  in the setting of Carnot groups.
          Let $G$ be a  Carnot group with Lie algebra  grading $\mathfrak g=\oplus_{j=1}^n V_j$.    Then 
            the horizontal tautological one form $\theta_H$ is a $V_1$-valued left invariant one form on $G$.   
            Let $\pi_1: \mathfrak g\ra V_1$ be the projection with respect to the decomposition $\mathfrak g=\oplus_{j=1}^n V_j$.   
          Define $\theta_H=\pi_1\circ \theta$; that is, for each $x\in G$,     $(\theta_H)_x: T_x G\ra V_1$  is defined to be the composition $\pi_1\circ \theta_x$.

       Let $e_1, \cdots, e_{k_n}$ be  a basis of $\mathfrak g=T_eG$ such that $\{e_1, \cdots, e_{k_1}\}$ is  a basis of $V_1$.
        For each $i$, let $X_i$ be the left invariant vector field on $G$  that equals $e_i$ at $e$.    Let $\{\theta_1, \cdots, \theta_{k_n}\}$ be  the basis of   the space of left invariant one forms on $G$ dual to $\{e_1, \cdots, e_{k_n}\}$.  So we have 
         $\theta_i(X_j)=\delta_{ij}$.     Now it is easy to see that $\theta_H=\sum_{i=1}^{k_1} \theta_i  e_i$ as 
         $(\sum_{i=1}^{k_1} \theta_i  e_i)(\sum_{j=1}^{k_n} a_j X_j)=\sum_{i=1}^{k_1} a_i  e_i$.
         If we use the exponential coordinates  $\mathbb R^{\dim \mathfrak g}\ra G$, $(x_1, \cdots, x_{k_n})\mapsto
           \text{exp}(\sum x_j e_j)$, then $\theta_i$ is given by $\theta_i=dx_i$ for $1\le i\le k_1$.    Hence $\theta_H$ has the expression $\theta_H=\sum_{i=1}^{k_1} dx_i e_i$.  
         
        \subsection{Structure of Lie algebra when $\alpha$ is not an integer}\label{structure of n}

        We return to the case that $(N,D)$ is Carnot-by-Carnot. 
        We rescale so that the smallest eigenvalue of $D$ is $1$. Let $\alpha>1$ be the smallest eigenvalue of the induced derivation $\bar D: {\mathfrak n}/{\mathfrak w}\ra 
        {\mathfrak n}/{\mathfrak w}$.  
        Under the assumption that $\alpha$ is not an integer    we  clarify  the algebraic structure of $\mathfrak n$ and show that $\mathfrak n$ is a central product. 
       
       Recall $\bar{\mathfrak n}=\mathfrak n/\mathfrak w=\bar V_1\oplus\cdots
       \oplus \bar V_m$  is assumed to be a Carnot algebra   and $\pi: \mathfrak n\ra \bar {\mathfrak n}$ is the quotient map.  
     \begin{Le}\label{notinteger}
     (1) If $\alpha>1$ is irrational, then $\mathfrak  n$ has an ideal $\mathfrak h$ that is mapped by $\pi$ isomorphically onto  
     $\bar{\mathfrak n}$. In particular, $\mathfrak n$ is the direct sum of two ideals $\mathfrak w$ and $\mathfrak h$.\newline
      (2) Suppose $\alpha>1$ is rational but not an integer.  Let $k_0=\min\{ k\in \mathbb N| \,k\alpha\;\; \text{is an integer}\}$.  
      Then there exist a graded central ideal $I$ of $\mathfrak w$ contained in $\oplus_lW_{lk_0\alpha}$,   a Carnot algebra $\mathfrak h=\oplus_j H_j$ with a graded central ideal $J$ 
       contained in $\oplus_l H_{lk_0}$, and a linear isomorphism $\phi: I\ra J$ satisfying 
      $\phi(I\cap W_{lk_0\alpha})=J\cap H_{lk_0}$, such that $\mathfrak n$ is isomorphic to the central product
       $\mathfrak w\times_\phi {\mathfrak h}=(\mathfrak w\oplus {\mathfrak h})/K$, where $K=\{(x, \phi(x))|x\in I\}$. 
      
     \end{Le}
     
  \begin{proof}
  Let $H_1=V_\alpha$  and $\mathfrak h$ be the Lie subalgebra  of $\mathfrak n$  generated by $H_1$.     Then $\pi(H_1)=\bar V_1$
    and $\pi(\mathfrak h)=\bar{\mathfrak n}$.   Furthermore, the property  $[V_a, V_b]\subset V_{a+b}$ implies $\mathfrak h$ is a Carnot algebra.    \newline
  (1) Assume $\alpha$ is irrational.  Since $\mathfrak h\subset \oplus_l V_{l\alpha}$
    and $\mathfrak w\subset \oplus_{j\ge 1}V_j$, we have $\mathfrak w \cap \mathfrak h=\{0\}$.  
  As $\mathfrak w $ 
   is an ideal of $\mathfrak n$, the property $[V_a, V_b]\subset V_{a+b}$ implies
   $[\mathfrak w, \mathfrak h]=0$.  It follows that  $\mathfrak h$ is an ideal of $\mathfrak n$ and $\mathfrak n=\mathfrak w\oplus \mathfrak h$ is a direct sum of  two ideals.

  (2) Suppose $\alpha>1$ is rational but not an integer. The fact $[H_1, W_1]\subset V_{1+\alpha}\cap \mathfrak w=\{0\}$ implies 
    $[\mathfrak w, \mathfrak h]=0$  and so $\mathfrak h$ is also an ideal of $\mathfrak n$. However in this case $\mathfrak h$ and $\mathfrak w$ may have nontrivial intersection. The fact $[\mathfrak w,\mathfrak h]=0$ implies $\mathfrak w\cap \mathfrak h$ is central in $\mathfrak n$ and so is central in both $ \mathfrak w$ and $\mathfrak h$. Define $f: \mathfrak w\oplus \mathfrak h\ra \mathfrak n$ by   $f(w,h)=w+h$.      Then $f$ is a surjective Lie algebra homomorphism with kernel
     $f^{-1}(0)=\{(w,h)| w\in \mathfrak w\cap \mathfrak h, h=-w\}$.  Set $I=J=\mathfrak w\cap \mathfrak h$ and define $\phi: I\ra J$ by $\phi(w)=-w$. Then 
       $\mathfrak n\cong\mathfrak w\times_\phi \mathfrak h$.  Finally   as the intersection of two graded ideals $\mathfrak w$, $\mathfrak h$ of $\mathfrak n$,  $\mathfrak w\cap \mathfrak h$ ($=I$=$J$)   is  a graded ideal in both $\mathfrak w$ and $\mathfrak h$.


  \end{proof}
       
       Since $[\mathfrak w, \mathfrak h]=0$ and 
       $\mathfrak n=\mathfrak w+ \mathfrak h$, we see that $Z(\mathfrak w)$ lies in the center of $\mathfrak n$.

       \subsection{Characterization  of biLipschitz shear maps}\label{shear-sec}

       Let $s:  N/W\ra W$ be a map  satisfying $s(0)=0$  and
        $F: (N, d) \ra (N, d) $ be given by $F(g)=g s(gW)$.  
       
  \begin{lemma}\label{z(w)}
  Assume $F$ as above is biLipschitz. Then  $s$ takes values in the center $Z(W)$ of $W$.
  \end{lemma}
  
   \begin{proof}   As  $s(gW)\in W$,  $F$ maps each coset of $W$ to itself.  
   Let $g\in N$ and  $A_g:  W\to W$ be given by $A_g=F_g|_W$.  
    Observe that $A_g$ is an inner automorphism of $W$: $A_g(w)=s(gW)^{-1}ws(gW)$.   As $F$ is biLipschitz, so is $A_g$.   A biLipschitz  automorphism of a Carnot group is necessarily graded (see Lemma \ref{bilip auto})   
      and so is completely determined by its action on the first layer.  On the other hand, an inner automorphism induces  the trivial  map  on $\mathfrak w/[\mathfrak w,\mathfrak w]$.   It follows that $A_g$ is the trivial automorphism and therefore  $s(gW)$ must lie in $Z(W)$.

   \end{proof}            
       

       We need to introduce some function spaces before we can  give the statement of the    characterization.  
       
       {\bf The spaces $\mathcal H_j$.}  Let $\pi: \mathfrak n\ra \mathfrak n/\mathfrak w$ be the natural projection.
          For   any $\bar X\in  \mathfrak n/\mathfrak w$  and  $z\in Z(\mathfrak w)$, we define $[\bar X, z]:=[X, z]$, where $X\in \mathfrak n$ is such that 
           $\pi(X)=\bar X$.   This is well-defined since    $z\in Z(\mathfrak w)$  and different choices of $X$ differ by an element in $\mathfrak w$.  Notice  $[\bar X, z]\in Z(\mathfrak w)$.  
       
       Let $\theta_H$ be the horizontal tautological one form on $\mathfrak n/\mathfrak w$.  
          For any $j\ge 1$,  denote by  $P_j$ the space of all continuous maps  $c: \mathfrak n/\mathfrak w\ra Z_j(\mathfrak w)$   (recall $Z_j(\mathfrak w)=Z(\mathfrak w)\cap W_j$) satisfying  $c(0)=0$ and 
          $\int_\gamma [c(x),  \theta_H (x)]=0  $ for all closed horizontal curves  $\gamma$ in $\mathfrak n/\mathfrak w$.  
           Here we are using our definition of bracket $[z, \bar X]$  for $z\in Z(\mathfrak w)$, $\bar X\in \mathfrak n/\mathfrak w$.    Notice $[c(x),  \theta_H (x)]\in Z_{j+\alpha}(\mathfrak w)$ as $c(x)\in Z_j(\mathfrak w)$ and $\theta_H(x)\in \bar V_1$.  
       For any $c\in P_j$,   define a  map
       $c^{(1)}:  \mathfrak n/\mathfrak w\ra Z_{j+\alpha}(\mathfrak w)$   by
        $c^{(1)}(p)=\int_\gamma  [c(x), \theta_H (x)]$, where $\gamma$ is any horizontal path from $0$  to $p$. This is well-defined by the definition of $P_j$.    If $c^{(1)}\in P_{j+\alpha}$, then  we define 
        $c^{(2)}=(c^{(1)})^{(1)}$.   Similarly,  we can define 
        $c^{(k)}:      \mathfrak n/\mathfrak w\ra Z_{j+k\alpha}(\mathfrak w)$ if 
        $c^{(k-1)}\in P_{j+(k-1)\alpha}$.   
        
        For each  integer $1\le j\le \alpha$, let $E_j$ be the space of    $\frac{j}{\alpha}$ -Holder  continuous maps $c:  \mathfrak n/\mathfrak w\ra Z_j(\mathfrak w)$  satisfying $c(0)=0$,  and    $\mathcal H_j$ be the set of elements $c\in E_j$  such that $c^{(k)}$ is defined  for all
            $k\ge 1$ (i.e.,   $c^{(k-1)} \in P_{j+(k-1)\alpha}$).  Here the metric on $\mathfrak n/\mathfrak w$   is a Carnot metric and the metric on  $Z_j(\mathfrak w)$   is an Euclidean metric. 
             Note that $c^{(k)}\equiv 0$  for large enough $k$ since $\mathfrak n$ is nilpotent.  
        
     {\bf Special case.}   We notice that $\mathcal H_j=E_j$  for all $1\le j<\alpha$ when $\alpha$ is not an integer. This is because in this case $[Z(\mathfrak w),  \mathfrak n]=0$ and so $c^{(i)}\equiv 0$     for all  $c\in E_j$ and all $i\ge 1$.

        Given a continuous map  $s: \mathfrak n/\mathfrak w \ra Z(\mathfrak w)$ satisfying $s(0)=0$,  define 
       $K: \mathfrak n/\mathfrak w \times \mathfrak n/\mathfrak w   \ra Z(\mathfrak w)$  by  
       $$K(\bar g_1, \bar g_2)=s(\bar g_2)*(g_1^{-1}*g_2)^{-1}*(-s(\bar g_1)) *(g_1^{-1}*g_2).$$   Notice that $K$ is well-defined as $s$ takes values in $ Z(\mathfrak w)$.   Since both $s(\bar g_2)$ and $(g_1^{-1}*g_2)^{-1}*(-s(\bar g_1)) *(g_1^{-1}*g_2)$  lie in $Z(\mathfrak w)$, the BCH formula implies 
          $K(\bar g_1, \bar g_2)=s(\bar g_2)+(g_1^{-1}*g_2)^{-1}*(-s(\bar g_1)) *(g_1^{-1}*g_2)$.

        \begin{lemma}\label{shear-bilip-lemma}
          Write $s=\sum_j s_j$, where $s_j: \mathfrak n/\mathfrak w \ra Z_j(\mathfrak w)$   is the $W_j$ component of $s$. \newline
        (1)   Assume $\alpha$ is not an integer.  
          If there is a constant $C>0$ such that 
             $|\pi_{j}(K(\bar g_1, \bar g_2))|^{\frac{1}{j}}
     \le C d^{\frac{1}{\alpha}}_{CC}(\bar g_1,   \bar g_2)$ for  all $j \ge 1$ and all $\bar g_1, \bar g_2\in 
            \mathfrak n/\mathfrak w$,   
        then  $s_j\in E_j$ for $1\le j<\alpha$ and $s_j\equiv 0$ for $j>\alpha$;\newline
       (2)  Assume $\alpha$ is an integer and $j_0$ is an integer satisfying $1\le j_0\le \alpha$.      If  there is a constant $C>0$ such that 
             $|\pi_{k\alpha+j_0}(K(\bar g_1, \bar g_2))|^{\frac{1}{k\alpha+j_0}}
     \le C d^{\frac{1}{\alpha}}_{CC}(\bar g_1,   \bar g_2)$ for  all $k \ge 0$ and all $\bar g_1, \bar g_2\in 
            \mathfrak n/\mathfrak w$,      then  $s_{j_0}\in \mathcal H_{j_0}$   and 
            $s_{k\alpha+j_0}=s_{j_0}^{(k)}$ for all $k\ge 1$. 
        \end{lemma}
        
        \begin{proof}  
          Notice 
         \begin{align*}K(\bar g_1, \bar g_2)
         =s(\bar g_2)-s(\bar g_1)+[-(g_1^{-1}*g_2), -s(\bar g_1)]+
      \sum_{k\ge 2}\frac{1}{k!}(ad(-(g_1^{-1}*g_2)))^k(-s(\bar g_1)).
      \end{align*}
      Since $s$ takes values in $Z(\mathfrak w)$,     using our definition of bracket $[\bar X, z]$ with $\bar X\in \mathfrak n/\mathfrak w$, $z\in Z(\mathfrak w)$, 
       we can write
       $$K(\bar g_1, \bar g_2)=s(\bar g_2)-s(\bar g_1)+[-(\bar g_1^{-1}*\bar g_2), -s(\bar g_1)]+
      \sum_{k\ge 2}\frac{1}{k!}(ad(-(\bar g_1^{-1}*\bar g_2)))^k(-s(\bar g_1)).$$


                (1) Assume $\alpha$ is not an integer.  Since $[Z(\mathfrak w), \mathfrak n]=0$ in this case we have  
                $K(\bar g_1, \bar g_2)=s(\bar g_2)-s(\bar g_1)$.   The assumption implies 
                $|s_j(\bar g_2)-s_j(\bar g_1)|^{\frac{1}{j}}\le C d^{\frac{1}{\alpha}}_{CC}(\bar g_1,   \bar g_2)$ for  all $j \ge 1$ and all $\bar g_1, \bar g_2\in 
            \mathfrak n/\mathfrak w$. 
              When $1\le j<\alpha$, this implies $s_j\in E_j$.   
             When $j>\alpha$, this implies $s_j$ is a constant function and so $s_j\equiv 0$ as $s_j(0)=0$.

      (2) Assume $\alpha$ is an integer.       
      As    $[-(\bar g_1^{-1}*\bar g_2), -s(\bar g_1)]\in\oplus_{i\ge \alpha+1}W_i$,  
             we have  $\pi_jK(\bar g_1, \bar g_2)=s_j(\bar g_2)-s_j(\bar g_1)$ for $1\le j\le \alpha$.   
      The assumption implies $s_{j_0}\in E_{j_0}$.            
         Let $\bar \gamma: [a,b]\ra \mathfrak n/\mathfrak w$ be a rectifiable   horizontal curve parametrized by arc length.     Then  $d_{CC}(\bar \gamma(t_1),  \bar \gamma(t_2))\le |t_2-t_1|$ for any  $a\le t_1, t_2\le b$.     
            Write   $\bar\gamma(t_1)^{-1}*\bar \gamma(t_2)=\sum_j \bar p_j$ with   
           $\bar p_j=\bar p_j(t_1, t_2)\in \bar V_j$.  
           Since $\bar \gamma$ is a horizontal curve, we have $\bar p_j=o(t_2-t_1)$   for $j\ge 2$  (as $t_2\ra t_1$).   If $\bar \gamma$ has tangent at $t_1$, then  
           $\frac{\bar p_1}{t_2-t_1}\ra \bar\gamma'(t_1)$ as $t_2\ra t_1$.

            For $k\ge 1$,  
            we have  
            $$\pi_{k\alpha+j_0}K(\bar \gamma(t_1), \bar \gamma(t_2)) =s_{k\alpha+j_0}(\bar \gamma(t_2))-s_{k\alpha+j_0}(\bar \gamma(t_1))
            +[\bar p_1, s_{(k-1)\alpha+j_0}(\bar \gamma(t_1))]+o(t_2-t_1). $$
              Since  $d_{CC}(\bar \gamma(t_1),  \bar \gamma(t_2))\le |t_2-t_1|$,  
             the   assumption  implies that
               $$\frac{\pi_{k\alpha+j_0}K(\bar \gamma(t_1), \bar \gamma(t_2)) }{t_2-t_1}\ra 0$$ as $t_2\ra t_1$.  
             If $\bar \gamma$ has tangent at $t_1$, then 
             $$\frac{d}{dt}s_{k\alpha+j_0}(\bar \gamma(t))|_{t=t_1}=
             \lim_{t_2\ra t_1}\frac{s_{k\alpha+j_0}(\bar \gamma(t_2))-s_{k\alpha+j_0}(\bar \gamma(t_1))}{t_2-t_1}= [s_{(k-1)\alpha+j_0}(\bar \gamma(t_1)),   \bar\gamma'(t_1)].$$
        By the fundamental theorem of calculus, 
        we have
        \begin{equation}\label{path inde}
        s_{k\alpha+j_0}(\bar \gamma(b))-s_{k\alpha+j_0}(\bar \gamma(a))=\int_a^b
        [s_{(k-1)\alpha+j_0}(\bar \gamma(t)), \bar \gamma'(t)]dt=\int_{\bar\gamma}
         [s_{(k-1)\alpha+j_0}(x),  \theta_H(x)].
        \end{equation}
        In particular,  $\int_{\bar\gamma}
         [s_{(k-1)\alpha+j_0}(x),  \theta_H(x)]=0$ when $\bar \gamma$ is a closed horizontal curve in $\mathfrak n/\mathfrak w$  and so $s_{(k-1)\alpha+j_0}\in P_{(k-1)\alpha+j_0}$.   Furthermore, (\ref{path inde}) also shows $s_{k\alpha+j_0}=s_{(k-1)\alpha+j_0}^{(1)}$.  
        By induction this shows that $s_{j_0}\in \mathcal H_{j_0}$
          and  $s_{k\alpha+j_0}=s_{j_0}^{(k)}$ for all $k\ge 1$. 
        
        \end{proof}

        \begin{proposition}\label{shear-bilip}   
        Let $s: \mathfrak n/\mathfrak w \ra Z(\mathfrak w)$ be a continuous map  satisfying $s(0)=0$  and $F:\mathfrak n \ra \mathfrak n$, $F(g)=g*s(\bar g)$ be the associated  shear map. Write $s=\sum_j s_j$, where $s_j: \mathfrak n/\mathfrak w \ra Z_j(\mathfrak w)$   is the $W_j$ component of $s$.   
          Assume     $F$ is biLipschitz.\newline
           (1) If $\alpha$ is not an integer, then $s_j\in E_j$ for $1\le j<\alpha$ and $s_j\equiv 0$ for $j>\alpha$;\newline
        (2) If $\alpha$ is an integer,     then  $s_j\in \mathcal H_j$ for each $1\le j\le \alpha$
           and  $s_{k\alpha+j}=s_j^{(k)}$ for all $k\ge 1$ and $1\le j\le \alpha$.
        
        \end{proposition}

\begin{proof}
           There is some $L>0$ such that 
           $d(F(g_1), F(g_2))\le L\cdot d(g_1, g_2)$ for any $g_1, g_2\in \mathfrak n$. 
        Let $g_1,   g_2\in \mathfrak n$.  Then $F(g_1)=g_1*s(\bar g_1)$  and $F(g_2)=g_2*s(\bar g_2)$.   
     Since   $$(F(g_1))^{-1}*F(g_2)=(s(\bar g_1))^{-1} *g_1^{-1}*g_2*s(\bar g_2),$$
       we have 
       $K(\bar g_1, \bar g_2)
       =(g_1^{-1}*g_2)^{-1}*(F(g_1))^{-1}*F(g_2)$.  
        It follows that 
        \begin{align*}d(0, K(\bar g_1, \bar g_2))
        &=
        d(g_1^{-1}*g_2, (F(g_1))^{-1}*F(g_2))\\
        & \le d(g_1^{-1}*g_2, 0)
        +d(0, (F(g_1)^{-1}*F(g_2)))\\
        &=d(g_1, g_2)+d(F(g_1), F(g_2))\le (L+1) d(g_1, g_2).
        \end{align*}
         Hence     for any $i\ge 1$ we have
       \begin{align*}  |\pi_{i}(K(\bar g_1, \bar g_2))|^{\frac{1}{i}}
      \le  d(0, K(\bar g_1, \bar g_2))
      & \le (L+1) \inf_{w_1, w_2\in \mathfrak w}d(g_1*w_1, g_2*w_2)\\
      &=(L+1) d(g_1*\mathfrak w, g_2*\mathfrak w)\\
         &    \le  (L+1)C d_{CC}(\bar g_1,   \bar g_2)^{\frac{1}{\alpha}},
       \end{align*}
         for some $C>0$ as  the distance between cosets $d(g_1*\mathfrak w, g_2 *\mathfrak w)$ is comparable with 
           $d_{CC}(\bar g_1,   \bar g_2)^{\frac{1}{\alpha}}$.   
           The proposition now follows from Lemma \ref{shear-bilip-lemma}.  
           
           \end{proof}


       


        \begin{lemma}\label{zigzag path}
        Given any Carnot group $G$, there exists an integer $n_0$ and a constant $C>0$ 
         with the following property:
          for any $g\in G$, there is a horizontal curve $c$ from $0$ to $g$ 
           that is a concatenation of at most $n_0$ horizontal line segments  and such that 
           the length of $c$  is at most $C\cdot d(0,g)$.  Here a horizontal line segment in $G$ is  a path  $c:[0,a]\ra G$ of the form $c(t)=g \,{\text{exp}}(tX)$, where  $X$ lies in the first layer of the Lie algebra of $G$.  
        
        \end{lemma}
        
        \begin{proof}
        This follows from the proof of  Proposition 2.26 in \cite{BLD}, see also  Chapter 8 
        of \cite{AS04} (in particular Theorem 8.1 and
Proposition 8.5).

        \end{proof}
        
     Recall that $H_1\subset  V_\alpha$ is a linear subspace complimentary to $W_\alpha$.

        \begin{lemma}\label{conjugate}
          There is a constant $C>0$ depending only on  $N$ with the following property. 
        For any   $b_1, b_2>0$,  and any    $w\in Z(\mathfrak w)$,  $h\in H_1$ satisfying 
          $d(0,w)\le b_1^{\frac{1}{\alpha}}$ and $|h|<b_2$, the inequality $d(0, (-h)*w*h)\le C\cdot (\max\{b_1, b_2\})^{\frac{1}{\alpha}}$  holds.

        \end{lemma}
        
        \begin{proof}
        Write $w=w_1+\cdots +w_m$ with $w_j\in W_j$.  The assumption
          $d(0,w)\le b_1^{\frac{1}{\alpha}}$   implies $|w_j|\le b_1^{\frac{j}{\alpha}}$
        for each $j$.   We calculate
         $$(-h)*w*h=w+\sum_{k\ge 1}\frac{(-1)^k}{k!}(ad(h))^k w.$$
          As $N$ is nilpotent, the above is a finite sum. 
           We have 
         $$\pi_j((-h)*w*h)=w_j+\sum_{k\ge 1, j-k\alpha\ge 1}\frac{(-1)^k}{k!}(ad(h))^k w_{j-k\alpha}.$$
         By using $|[X,Y]|\le C_0\cdot |X|\cdot |Y|$, we get  (with $b=\max\{b_1, b_2\}$) 
          $$\huge|\frac{(-1)^k}{k!}(ad(h))^k w_{j-k\alpha}\huge|
          \le \frac{C_0^k}{k!} |h|^k |w_{j-k\alpha}|
          \le \frac{C_0^k}{k!} b_2^{{k}}b_1^{\frac{j-k\alpha}{\alpha}}
          \le \frac{C_0^k}{k!} b^{\frac{j}{\alpha}}.
          $$
        From this it is clear that $|\pi_j((-h)*w*h)|\le C b^{\frac{j}{\alpha}}$  for some $C$ depending only on $N$ and the lemma follows.

        \end{proof}


            \begin{proposition}\label{construction}   
            Let $s_j\in \mathcal H_j$  be given  for each integer  $1\le j\le \alpha$.\newline
             (1) If  $\alpha$ is not an integer,    set $ s_j\equiv 0$ for $j>\alpha$ and $s=\sum_j s_j$;  \newline
               (2)  If  $\alpha$ is an integer,    set $ s_{k\alpha+j}= s_j^{(k)}$ for all $k \ge 1$   and $s=\sum_j s_j$. \newline
               Then  the  shear map  associated to $s$ is biLipschitz. 
             
            
            \end{proposition}
            
            \begin{proof}
           We shall prove that    the  shear map  $F$ associated to $s$ 
             is Lipschitz.   The same argument shows $F^{-1}$ is also Lipschitz since $F^{-1}$ is the shear map associated to     $-s$.  

             By the triangle inequality,
             \begin{align*}
             d(F(g_1), F(g_2))&=d(0, F(g_1)^{-1}*F(g_2))\\
             &=d((g_1^{-1}*g_2)^{-1}, (g_1^{-1}*g_2)^{-1}*F(g_1)^{-1}*F(g_2))\\
             &\le d((g_1^{-1}*g_2)^{-1}, 0)+d(0, (g_1^{-1}*g_2)^{-1}*F(g_1)^{-1}*F(g_2))\\
             &=d(g_1, g_2)+d(0, K(\bar g_1, \bar g_2)).     
             \end{align*}
             Hence it  suffices to show there is some $C>0$ such that
             \begin{equation}\label{ine5.3}
             d(0, K(\bar g_1, \bar g_2))
             \le C\cdot d(g_1, g_2)\;\;\; \forall g_1, g_2\in \mathfrak n.
             \end{equation}
             
       (1) Assume $\alpha$ is not an integer. In this  case we have 
       $    K(\bar g_1, \bar g_2)=s(\bar g_2)-s(\bar g_1)$  and   (\ref{ine5.3}) follows from the assumption.  
             
             (2) Assume $\alpha$ is  an integer.
             We first show  (\ref{ine5.3}) in the case when $\bar g_1^{-1}*\bar g_2$ lies in $\bar V_1$.
             So assume $\bar g_1^{-1}*\bar g_2=t\bar h\in \bar V_1$  with $|\bar h|=1$ and $t>0$.    
                We have 
              \begin{align*}
         K(\bar g_1, \bar g_2)
         &=s(\bar g_2)-s(\bar g_1)+[-t\bar h, -s(\bar g_1)]+
      \sum_{k\ge 2}\frac{1}{k!}(ad(-t\bar h))^k(-s(\bar g_1))\\
      &=s(\bar g_2)-s(\bar g_1)+t[\bar h, s(\bar g_1)]+\sum_{k\ge 2}\frac{(-1)^{k+1}t^k}{k!}(ad(\bar h))^ks(\bar g_1).
      \end{align*}
             
            Since in our case $d(g_1, g_2)\ge d_{CC}(\bar g_1, \bar g_2)^{\frac{1}{\alpha}}=t^{\frac{1}{\alpha}}$, it suffice to show  that there is some constant 
             $C>0$    such that for each $j\ge 1$, the   following inequality  holds:
            \begin{equation}\label{e5.12}
            |\pi_jK(\bar g_1, \bar g_2)|^{\frac{1}{j}}\le C\cdot t^{\frac{1}{\alpha}}.
            \end{equation}

            When $1\le j\le \alpha$,  
            $\pi_jK(\bar g_1, \bar g_2)=s_j(\bar g_2)-s_j(\bar g_1)$  and  so  
               (\ref{e5.12})  holds in this case since by assumption $s_j\in E_j$.
             Now let $k_0\ge 1$ and 
             suppose there is some constant $C>0$ such that 
             $|\pi_{k\alpha+j}K(\bar g_1, \bar g_2)|^{\frac{1}{k\alpha+j}}\le C\cdot t^{\frac{1}{\alpha}}$  holds  for all $k<k_0$ and all $1\le j\le \alpha$.  We shall show  a similar inequality holds for $k=k_0$ (with a different constant).
              We have
            \begin{align*}
        & \pi_{k_0\alpha+j}K(\bar g_1,   \bar g_2)\\
        &=    s_{k_0\alpha+j}(\bar g_2)-s_{k_0\alpha+j}(\bar g_1)+t[\bar h, s_{(k_0-1)\alpha+j}(\bar g_1)]+\sum_{k= 2}^{k_0}\frac{(-1)^{k+1}t^k}{k!}(ad(\bar h))^k(s_{(k_0-k)\alpha+j}(\bar g_1)).
        \end{align*}
            
            Let $\bar \gamma: [0,t]\ra \mathfrak n/\mathfrak w$ be the horizontal curve given by
             $\bar \gamma(t_1)=\bar g_1*(t_1\bar h)$.   We have $\bar \gamma'(t_1)=\bar h$ 
              for all $t_1$.   By the definition of  $s_{k\alpha+j}$, we have 
            \begin{align*}
           & s_{k_0\alpha+j}(\bar g_2)-s_{k_0\alpha+j}(\bar g_1)+t[\bar h, s_{(k_0-1)\alpha+j}(\bar g_1)]\\
          &=\int_0^t [s_{(k_0-1)\alpha+j}(\bar g_1*(t_1\bar h)), \bar h]dt_1-\int_0^t[s_{(k_0-1)\alpha+j}(\bar g_1), \bar h]dt_1\\
          &=\int_0^t [s_{(k_0-1)\alpha+j}(\bar g_1*(t_1\bar h))-s_{(k_0-1)\alpha+j}(\bar g_1), \bar h]dt_1\\
          &=\int_0^t[\int_0^{t_1}[s_{(k_0-2)\alpha+j}(\bar g_1*(t_2\bar h)),  \bar h]dt_2, \bar h]dt_1\\
          &=\int_0^t\int_0^{t_1}[[s_{(k_0-2)\alpha+j}(\bar g_1*(t_2\bar h)),  \bar h],  \bar h]dt_2dt_1\\
          &=(-1)^2\int_0^t\int_0^{t_1} (ad(\bar h))^2(s_{(k_0-2)\alpha+j}(\bar g_1*(t_2\bar h))dt_2dt_1.
          \end{align*}
            On the other hand, notice  
             \begin{align*}
             \frac{(-1)^{k+1}t^k}{k!}(ad(\bar h))^k(s_{(k_0-k)\alpha+j}(\bar g_1))
          =-(-1)^{k}\int_0^t\int_0^{t_1}\cdots\int_0^{t_{k-1}}(ad(\bar h))^k(s_{(k_0-k)\alpha+j}(\bar g_1))dt_{k}\cdots dt_1.
          \end{align*}
         By induction we get
         \begin{align*}   
           &K_{k_0\alpha+j}(\bar g_1, \bar g_2)\\
           &=(-1)^{k_0}\int_0^t\cdots\int_0^{t_{k_0-1}}(ad (\bar h))^{k_0}
           (s_j(\bar g_1*t_{k_0}\bar h)-s_j(\bar g_1)) dt_{k_0}\cdots dt_1.
           \end{align*}
            There is a constant  $C_1>0$ depending only on $\mathfrak n$ such that $|[X, Y]|\le C_1 |X|\cdot |Y|$ for any $X, Y\in \mathfrak n$. 
            Since $s_j$ is $\frac{j}{\alpha}$-Holder   for $1\le j\le \alpha$, we have
            \begin{align*}
            |K_{k_0\alpha+j}(\bar g_1, \bar g_2)|
            &\le \int_0^t\cdots\int_0^{t_{k_0-1}}|(ad (\bar h))^{k_0}
           (s_j(\bar g_1*t_{k_0}\bar h)-s_j(\bar g_1))| dt_{k_0}\cdots dt_1\\
           &\le \int_0^t\cdots\int_0^{t_{k_0-1}}C_1^{k_0}|\bar h|^{k_0}
           |(s_j(\bar g_1*t_{k_0}\bar h)-s_j(\bar g_1))| dt_{k_0}\cdots dt_1\\
           &\le C_1^{k_0}\int_0^t\cdots\int_0^{t_{k_0-1}} C 
           d_{CC}(\bar g_1*(t_{k_0}\bar h),  \bar g_1)^{\frac{j}{\alpha}}dt_{k_0}\cdots dt_1\\
           &=  C_1^{k_0}C\int_0^t\cdots\int_0^{t_{k_0-1}} t_{k_0}^{\frac{j}{\alpha}}  
           dt_{k_0}\cdots dt_1\\
           &=C_2 t^{\frac{k_0\alpha+j}{\alpha}}.
           \end{align*}
             Hence  (\ref{e5.12}) and in turn  (\ref{ine5.3}) holds 
         when    ${\bar g_1}^{-1}*\bar g_2\in \bar V_1$.  

            Now consider the general case.   Let  $g, g'\in \mathfrak n$.  As $\mathfrak n/\mathfrak w$ is Carnot,  
            by Lemma \ref{zigzag path},    there exist
                   $\bar g={\bar g}_0,  {\bar g}_1, \cdots, {\bar g}_k={\bar g}'$   such that 
                     $k\le n_0$,     
                   ${\bar g}_{i-1}^{-1}*{\bar g}_i\in \bar V_1$ and $\sum_i d_{CC}({\bar g}_{i-1},   {\bar g}_i)\le C d_{CC}(\bar g,{\bar g}')$, where   $C>0$  and $n_0$ 
                    depend only on $\mathfrak n/\mathfrak w$.    In particular we have $d_{CC}({\bar g}_{i-1},   {\bar g}_i)\le C d_{CC}(\bar g,{\bar g}')$.   
                Let $g_0=g$. We inductively define $g_i$ for $1\le i\le k$ such that 
                  $d(g_i, g_{i-1})=d(g_i*\mathfrak w, g_{i-1}*\mathfrak w)$.      
                       There is a constant $C_4>0$ such that 
                     $\frac{1}{C_4} d_{CC}(\bar x, \bar y)^{\frac{1}{\alpha}}\le d(x*\mathfrak w, y*\mathfrak w)\le C_4 d_{CC}(\bar x, \bar y)^{\frac{1}{\alpha}}$ for all $x,y\in \mathfrak n$.  We have  
                      \begin{align*}
                      d(g, g_k)&\le \sum_{i=1}^k d(g_{i-1}, g_i)=\sum_{i=1}^k d(g_{i-1}*\mathfrak w, g_i*\mathfrak w)\\
                      &\le \sum_{i=1}^k C_4 d_{CC}({\bar g}_{i-1}, {\bar g}_i)^{\frac{1}{\alpha}}\le C_4 n_0 C^{\frac{1}{\alpha}} d_{CC}({\bar g}, {\bar g}')^{\frac{1}{\alpha}}\\
                      &\le C_4^2n_0C^{\frac{1}{\alpha}} d(g*\mathfrak w, g'*\mathfrak w)\le 
                      C_4^2n_0C^{\frac{1}{\alpha}} d(g, g').
                      \end{align*}
                    By the triangle inequality we have $d(g_k, g')\le (1+C_4^2n_0C^{\frac{1}{\alpha}}) d(g, g')$.   
                 Now by the special case, we have 
                     $$d(F(g), F(g_k))\le \sum_i d(F(g_{i-1}), F(g_i))\le \sum_i C_3 d(g_{i-1}, g_i)\le  
                     C_3C_4^2n_0C^{\frac{1}{\alpha}} d(g, g').$$
                      On the other hand it is easy to see that the restriction of $F$ on the cosets 
                      of $W$  are isometries  and so $d(F(g_k), F(g'))=d(g_k, g')\le (1+C_4^2n_0C^{\frac{1}{\alpha}}) d(g, g')$.  Finally by the triangle inequality 
                       $d(F(g), F(g'))\le d(F(g), F(g_k))
+d(F(g_k), F(g'))\le      (1+(1+C_3)C_4^2n_0C^{\frac{1}{\alpha}})
     d(g, g')$.

            \end{proof}

          \section{Eliminating $s_j$ for $j<\alpha$}\label{eliminate}

   The goal of  this section  is to prove the following result.

  \begin{proposition}\label{elliminatedim(W)ge 2}
      Let $(N, D)$ be    Carnot-by-Carnot  satisfying $\dim(W)\ge 2$, $\dim(N/W)\ge 2$,    and $\Gamma$ 
      a    
    uniform quasisimilarity group  of $N$.
        Then 
       there exists a biLipschitz  map $F_0$ such that every element in the conjugate  
            $F_0\Gamma F_0^{-1}$    has a compatible  expression  $a*Bh*Aw *A s(\bar h)$ where the $W_j$ component $s_j$ of  $s$ vanishes for every $1\le j<\alpha$. 
            \end{proposition}

    We  shall use the approach in \cite{DyX16}  to prove  Proposition \ref{elliminatedim(W)ge 2}.   The cases $\dim(W)=1$ and  $\dim(N/W)=1 $   will be considered in Sections \ref{dim(w)=1}  and \ref{n/w} respectively.

             Let $N$   and $\Gamma$ be as in Proposition \ref{elliminatedim(W)ge 2}.   
               After replacing $\Gamma$ with a biLipschitz conjugate we may assume $\Gamma$ is a fiber similarity group (see  Section \ref{outline} 
               for the definition of  fiber similarity map and fiber similarity group).
                          To   be more precise, after applying   Theorem \ref{foliatedtheorem}   (and Lemma \ref{onfiber}) 
                           we may assume 
                           there are Carnot metrics  $d_{CC}$   and $\bar d_{CC}$   on  $W$ and $N/W$   respectively  such 
                           that  
                              every 
                    $\gamma\in \Gamma$ induces a similarity 
        $\bar\gamma$  of $(N/W,  \bar d_{CC})$,   and there is a graded automorphism $A_\gamma$ of  $W$ that is  also a similarity of 
         $(W, d_{CC})$ such that  $A_\gamma=\gamma_p|_W$ for any $p\in N$.

                              Let   $\text{Aut}_I(W, d_{CC})$ be the group of isometric graded automorphisms of $(W, d_{CC})$,  and 
      $\text{Aut}_c(W,d_{CC})$ be the group of graded automorphisms of $W$ that are compositions of a Carnot dilation and an isometric graded automorphism.  
        Then $\text{Aut}_I(W, d_{CC})$  is compact and 
        $\text{Aut}_c(W, d_{CC}))\cong   \text{Aut}_I(W, d_{CC})\times \mathbb R$  is amenable.   
   Similarly $\text{Aut}_c(\mathfrak n/\mathfrak w, \bar d_{CC})$ is amenable and so is the group of similarities  $\text{Sim}(\mathfrak n/\mathfrak w, \bar d_{CC})=(\mathfrak n/\mathfrak w)\rtimes \text{Aut}_c(\mathfrak n/\mathfrak w, \bar d_{CC})$
    of $(\mathfrak n/\mathfrak w, \bar d_{CC})$.   
             Denote  $G_0=\text{Aut}_c(W, d_{CC})\times
        \text{Sim}(\mathfrak n/\mathfrak w, \bar d_{CC})$
                and  define a map $\Psi: \Gamma\rightarrow  G_0$ by 
         $\Psi(\gamma)=(A_\gamma,  \overline{\gamma})$.   It is easy to see that $\Psi$ is a homomorphism. 
     Let $G$ be the closure of $\Psi(\Gamma)$ in $G_0$. 
           We notice that  $G$ is amenable, being a closed subgroup of the amenable group   $G_0$.

            Let $1\le j<\alpha$ and $E_j=\{s: (\mathfrak n/\mathfrak w, \bar d_{CC})\ra Z_j(\mathfrak w)|   s \;\text{is } \; \frac{j}{\alpha}-\text{Holder}, \; s(0)=0\}$.   We shall define an affine action of $G$ on $E_j$, then show that this affine action  has 
            a fixed point  in  $ \mathcal H_j\subset E_j$,   and finally   use   this fixed point to construct a biLipschitz   shear 
                map of $N$ to conjugate $\Gamma$ into a group  with 
               the  desired property.

            \subsection{The affine action on $E_j$}
              We first define a norm on $E_j$, then   a  linear  isometric action on $E_j$, and finally the translational part of the affine action.

           
           
           Since   the compact group $\text{Aut}_I(W, d_{CC})$ leaves $W_j$ invariant,  
               there exists a  $\text{Aut}_I(W, d_{CC})$-invariant 
                inner product $\langle,\rangle_j$ on $W_j$. 
                   Let $|\cdot|_j$ be the associated norm on $W_j$.
               Notice that   $E_j$ is a Banach space with respect to the norm   
             $$||s||=\sup_{p\not=q}\frac{|s(p)-s(q)|_j}{(\bar d_{CC}(p,q))^{\frac{j}{\alpha}}}.$$
            
            For a similarity  $f: X\ra X$  of a  metric space $X$, we denote by $\lambda_f$ the similarity constant of $f$: 
             $d(f(x_1), f(x_2))=\lambda_f d(x_1, x_2)$.   
             Let
    $G_1:=\{(A, B)\in  G_0|  \lambda_B=\lambda_A^\alpha\}$.    
           We first define an action of $G_1$ on $E_j$.  Let $c\in E_j$ and $(A, B)\in   G_1$.   We define $\pi_{(A,B)} c: 
                           \mathfrak n/\mathfrak w \ra Z_j(\mathfrak w)$ to be the function given by 
                            $$(\pi_{(A,B)} c)(\bar h)=A^{-1} c(B(\bar h))-A^{-1} c(B(0)).$$
               Clearly $(\pi_{(A,B)} c)(0)=0$ and $\pi_{(A,B)} c$ is linear in $c$.

             \begin{lemma}\label{homo}
                                $\pi_{(A,B)}$ defines an action of the opposite group $G_1^*$ of $G_1$ on $E_j$  by linear isometries.
                                \end{lemma}
                                
                                \begin{proof}    Write $A=\delta_{\lambda_A}\circ A'$, where 
                                $\delta_{\lambda_A}$ is a Carnot dilation and $A'\in \text{Aut}_I(W, d_{CC})$. 
                                For $w\in W_j$ we have $A(w)=\lambda_A^j A'(w)$ and so 
                                $|A(w_1)- A(w_2)|_j=\lambda_A^j|w_1-w_2|_j$ for $w_1, w_2\in W_j$.                
                Now let $\bar h_1, \bar h_2\in \mathfrak n/\mathfrak w$   and $c\in E_j$.    Then
                \begin{align*}
                \frac{|(\pi_{(A,B)} c)(\bar h_1)-(\pi_{(A,B)} c)(\bar h_2)|_j}{  \bar d_{CC}(\bar h_1, \bar h_2)^{\frac{j}{\alpha}}}&=\frac{|A^{-1} c(B(\bar h_1))-A^{-1} c(B (\bar h_2))|_j}
                { \bar d_{CC}(\bar h_1, \bar h_2)^{\frac{j}{\alpha}}}\\
                 &=\frac{\lambda_A^{-j}| c(B(\bar h_1))- c(B (\bar h_2))|_j}
                { \bar d_{CC}(\bar h_1, \bar h_2)^{\frac{j}{\alpha}}}\\
                &=\frac{| c(B (\bar h_1))-c(B  (\bar h_2))|_j}
                {(\lambda_B  \bar d_{CC}(\bar h_1, \bar h_2))^{\frac{j}{\alpha}}}\\
                &=\frac{|c(B(\bar h_1))-c(B (\bar h_2))|_j}{( \bar d_{CC}(B(\bar h_1),   B(\bar h_2)))^{\frac{j}{\alpha}}}.
                \end{align*}
                                showing  that  $\pi_{(A,B)} c\in E_j$  and  $||\pi_{(A,B)} c||=||c||$.  
                                We already observed that $\pi_{(A,B)} c$ is linear in $c$.
                                 Therefore
                                  the map $E_j\ra E_j$,  $c\mapsto \pi_{(A,B)} c$, is  a linear isometry. 
    
    It is easy to check that 
                               $\pi_{(A_2, B_2)(A_1, B_1)}=\pi_{(A_1, B_1)}\circ \pi_{(A_2, B_2)}$  holds    for any $(A_1,  B_1), (A_2, B_2) \in G_1$  and so $\pi_{(A,B)}$ defines an action of the opposite group $G_1^*$ of $G_1$ on $E_j$.    
    
    \end{proof}
                         


           
            To obtain a linear isometric  action of $G^*$ on $E_j$, we will show that $G\subset G_1$.  
           
           \begin{lemma}\label{dilation-relation}
        The formula $\lambda_{\bar \gamma}=\lambda^\alpha_{A_\gamma}$ holds for all $\gamma\in \Gamma$.    
           
           \end{lemma}
           
       \begin{proof} Since both $d_{CC}$ and $d|_W$ ($d$ is a   fixed $D$-homogeneous distance on $N$) are homogeneous distances on $W$, 
       there is a constant $L\ge 1$ such that  $d_{CC}(w_1, w_2)/L\le d(w_1, w_2)\le L d_{CC}(w_1, w_2)$ for all $w_1, w_2\in W$.     Similarly, as both $\bar d$ (recall $\bar d$ is the distance on $N/W$ induced by $d$, see   end of  Section \ref{homo distance})  and $ \bar d_{CC}^{\frac{1}{\alpha}}$ are   $\bar D$-homogeneous distances on $N/W$,  
       there is a constant $\bar L\ge 1$ such that  $ \bar d^{\frac{1}{\alpha}}_{CC}(x, y)/{\bar L}\le 
       \bar d(x, y)\le \bar L  \bar d^{\frac{1}{\alpha}}_{CC}(x, y)$  for any $x, y\in N/W$.  
        After possibly replace both $L$ and $\bar L$ with $\text{max}(L, \bar L)$ we 
          may  assume $\bar L=L$.  
        
         Let $\gamma\in \Gamma$.   
        Let $w_1\not= w_2\in  W$.    As $\gamma$ is a $(\Lambda, C_\gamma)$-quasi-similarity with respect to $d$  for some $C_\gamma>0$,  we have 
          \begin{align*}
       \lambda_{A_\gamma} d_{CC}(w_1, w_2)&=  d_{CC}(\gamma_p(w_1), \gamma_p(w_2))\le L d(\gamma_p(w_1), \gamma_p(w_2))\\
          &=L d(\gamma(pw_1), \gamma(pw_2))
          \le L \Lambda C_\gamma d(pw_1, pw_2)\\
          &\le L^2 \Lambda C_\gamma d_{CC}(w_1, w_2),
          \end{align*}
           and so $\lambda_{A_\gamma}  \le  L^2 \Lambda C_\gamma$.  
            Similarly   we get $\lambda_{A_\gamma}  \ge  \frac{C_\gamma}{L^2\Lambda}$     from 
           \begin{align*}
     \lambda_{A_\gamma} d_{CC}(w_1, w_2)&=    d_{CC}(\gamma_p(w_1), \gamma_p(w_2)) \ge \frac{1}{L}\cdot  d(\gamma_p(w_1), \gamma_p(w_2))\\
          &=\frac{1}{L}\cdot     d(\gamma(pw_1), \gamma(pw_2))
          \ge \frac{C_\gamma}{L\Lambda} d(pw_1, pw_2)\\
          &\ge  \frac{C_\gamma}{L^2\Lambda}  d_{CC}(w_1, w_2). 
          \end{align*}
       
       Pick $p, q\in N$ so that $pW\not=qW$ and such that $p, q$ realize the distance between the cosets $pW$ and $qW$.  Now
        \begin{align*} 
      (\lambda_{\bar \gamma} \bar d_{CC}(pW,  qW))^{\frac{1}{\alpha}} &= \bar d^{\frac{1}{\alpha}}_{CC}(\bar \gamma(pW), \bar\gamma(qW)) \le L  \bar d(\bar \gamma(pW), \bar \gamma(qW))\\
      &=L d(\gamma(p)W, \gamma(q)W) 
        \le L d(\gamma(p), \gamma(q))\\
      &  \le L\Lambda C_\gamma d(p,q)
      =L\Lambda C_\gamma \bar d(pW,  qW)\\
      &\le L^2\Lambda C_\gamma  \bar d^{\frac{1}{\alpha}}_{CC}(pW,  qW),
        \end{align*}
        yielding   $\lambda_{\bar \gamma}^{\frac{1}{\alpha}}\le  L^2\Lambda  C_\gamma$. 
         Similarly by picking $p, q\in N$ so that $pW\not=qW$ and $d(\gamma(p), \gamma(q))=d(\gamma(p)W, \gamma(q)W)$  we get 
           $\lambda_{\bar \gamma}^{\frac{1}{\alpha}}\ge  \frac{C_\gamma}{L^2\Lambda }$. 
         Combining the above four inequalities we get 
          $\frac{1}{L^4\Lambda^2} \le     \frac{ \lambda_{\bar \gamma}^{\frac{1}{\alpha}}  }{\lambda_{A_\gamma} }    \le L^4\Lambda ^2$  for all $\gamma\in \Gamma$.
       Notice that for any integer $n\ge 1$, we have 
        $\lambda_{A_{\gamma^n}} =\lambda^n_{A_\gamma} $ and $\lambda_{\bar \gamma^n}=\lambda^n_{\bar \gamma}$.  
           The above inequality applied to $\gamma^n$  yields   $$\frac{1}{L^4\Lambda^2} \le     \left(\frac{ \lambda_{\bar \gamma}^{\frac{1}{\alpha}}  }{\lambda_{A_\gamma} }\right)^n    \le L^4\Lambda ^2$$   for all $n\ge 1$  and   so we must have 
           $\lambda_{\bar \gamma}=\lambda^\alpha_{A_\gamma}$.

       \end{proof}


             Lemma \ref{dilation-relation}   implies  that $\Psi(\Gamma)\subset G_1$.  
         Since $G_1$ is a closed subgroup of
        $G_0$, we have $G\subset G_1$.   
            By  restricting to $G^*$   the linear isometric action of $G^*_1$  on $E_j$   we get  a linear isometric action of $G^*$ on $E_j$.  
            To get an affine action on $E_j$ we next define the translational part.  

        Define  a map $b_j:  \Gamma\ra E_j$     by 
                     $b_j(\gamma)=s_{\gamma,j}$, where $s_{\gamma, j}=\pi_j\circ s_\gamma$  
                      and $s_\gamma$ is  as in a compatible expression 
                      $\gamma(h*w)=\gamma(0)*B_\gamma h*A_\gamma w*A_\gamma s_\gamma(\bar h)$
                     of $\gamma$.   Recall that $\pi_j\circ s_\gamma$  is unique for $j<\alpha$ even if $\gamma$ may have more than one compatible expression. 
                      In Lemma \ref{transpart}   we shall prove that $s_{\gamma, j}\in \mathcal H_j$;    in particular,  $s_{\gamma, j}\in E_j$.

                     \begin{Le}\label{b_jiscocycle}
                The equality          $b_j(\gamma_2\gamma_1)=b_j(\gamma_1)+\pi_{\Psi(\gamma_1)}b_j(\gamma_2)$   holds  for any $\gamma_1, \gamma_2\in \Gamma$. 
                     \end{Le}     
                          
                          \begin{proof} The proof may be tedious, but the idea is simple: 
                          it  follows from two different ways of computing $\gamma_2\gamma_1(h)$. We first notice that 
                           if $X, Y\in \mathfrak n$ and $X\in \oplus_{\lambda_j\ge \alpha}V_{\lambda_j}$, then  $\pi_j(X*Y)=\pi_j(Y)$ for $1\le j<\alpha$. We shall repeatedly use this fact implicitly.  Also recall  that $s_{\gamma, j}(\bar h)=A^{-1}_\gamma\pi_j (\gamma(0)^{-1}* \gamma(h))$.

                          Let $\gamma_1, \gamma_2\in \Gamma$  and $h\in H$. Write ${\gamma_1}(0)=h_1*w_1$   and 
                           ${\gamma_1}(0)*B_{\gamma_1} h=h_2*w_2$  with $h_1, h_2\in H$ and $w_1, w_2\in \mathfrak w$. 
                             Then   
                             $$w^{-1}_1*w_2=[w^{-1}_1*(h^{-1}_2*h_1)*w_1]* B_{\gamma_1} h$$ and it follows that 
                              $w^{-1}_1*w_2\in \oplus_{j\ge \alpha}W_j$.  
                                We have  
                                $(\gamma_2\circ\gamma_1)(0)={\gamma_2}(0)*B_{\gamma_2}h_1*A_{\gamma_2}w_1*A_{\gamma_2}  s_{\gamma_2}(\bar h_1)$  and 
                                \begin{align*}
                               \gamma_2(\gamma_1(h))&
                               =\gamma_2({\gamma_1}(0)*B_{\gamma_1} h*A_{\gamma_1}  s_{\gamma_1}(\bar h))\\
                               &= \gamma_2(h_2*w_2*A_{\gamma_1}  s_{\gamma_1}(\bar h))\\
                               &= {\gamma_2}(0)*B_{\gamma_2} h_2*A_{\gamma_2}w_2 *A_{\gamma_2}A_{\gamma_1}s_{\gamma_1}(\bar h)*A_{\gamma_2} s_{\gamma_2}(\bar h_2).  
                                    \end{align*}
                  Now we have
                    \begin{align*}
                   & s_{\gamma_2\gamma_1,j}(\bar h)\\
                   &=A^{-1}_{\gamma_2\gamma_1} \pi_j({\gamma_2\gamma_1}(0)^{-1} *(\gamma_2\gamma_1)(h))\\
                    &=A^{-1}_{\gamma_2\gamma_1} \pi_j\{A_{\gamma_2}(s_{\gamma_2}(\bar h_1))^{-1}*A_{\gamma_2}w_1^{-1}*
                    (B_{\gamma_2}h_1)^{-1}* B_{\gamma_2}h_2*
                    A_{\gamma_2}w_2 *A_{\gamma_2}A_{\gamma_1}s_{\gamma_1}(\bar h)*A_{\gamma_2} s_{\gamma_2}(\bar h_2)\}\\
                    &=A^{-1}_{\gamma_2\gamma_1} \pi_j\{A_{\gamma_2}(s_{\gamma_2}(\bar h_1))^{-1}*A_{\gamma_2}w_1^{-1}*
                    A_{\gamma_2}w_2 *A_{\gamma_2}A_{\gamma_1}s_{\gamma_1}(\bar h)*A_{\gamma_2} s_{\gamma_2}(\bar h_2)\}\\
                 &=A^{-1}_{\gamma_2\gamma_1} \pi_j\{A_{\gamma_2}(s_{\gamma_2}(\bar h_1))^{-1}*A_{\gamma_2}(w_1^{-1}*w_2)
                     *A_{\gamma_2}A_{\gamma_1}s_{\gamma_1}(\bar h)*A_{\gamma_2} s_{\gamma_2}(\bar h_2)\}\\
                   &=A^{-1}_{\gamma_2\gamma_1} \pi_j\{A_{\gamma_2}(s_{\gamma_2}(\bar h_1))^{-1}
                     *A_{\gamma_2}A_{\gamma_1}s_{\gamma_1}(\bar h)*A_{\gamma_2} s_{\gamma_2}(\bar h_2)\}\\
                    &=\pi_j\{A^{-1}_{\gamma_1}(s_{\gamma_2}(\bar h_1))^{-1}* s_{\gamma_1}(\bar h)   *A^{-1}_{\gamma_1}s_{\gamma_2}(\bar h_2)\}.
                    \end{align*}        
                           Since $s_\gamma$ takes values in $Z(\mathfrak w)$, by BCH formula
                            and noting $\bar h_2=\bar\gamma_1(\bar h)$, 
                            $\bar h_1=\bar\gamma_1(0)$  we obtain 
                            $$ s_{\gamma_2\gamma_1,j}(\bar h)=s_{\gamma_1,j}(\bar h) + A^{-1}_{\gamma_1} s_{\gamma_2,j}(\bar h_2)- A^{-1}_{\gamma_1}s_{\gamma_2,j}(\bar h_1)=
                            s_{\gamma_1,j}(\bar h) +(\pi_{\Psi(\gamma_1)} s_{\gamma_2, j})(\bar h).
                            $$
                          
                          \end{proof}


        Recall we have two maps $\Psi: \Gamma\ra G$ and $b_j: \Gamma \ra E_j$.  
       \begin{lemma}\label{uniquelimit}
       Let   $\{ g_i\}$ and $\{\tilde g_i\}$ 
       be two sequences in $\Gamma$   and $(A,B)\in  G$   such  that 
       $\Psi( g_i)  \ra (A, B)$   and   $\Psi(\tilde g_i) \ra (A, B)$.    If  $s, \tilde s\in E_j$ are such that   
         $b_j(g_i)\ra  s$     and 
           $b_j(\tilde g_i)\ra \tilde s$     pointwise  
            as $i\ra \infty$, then $s=\tilde s$.

       \end{lemma}

     \begin{proof}  
        By setting $\gamma_1=\gamma^{-1}$ and $\gamma_2=\gamma$ in Lemma \ref{b_jiscocycle},   
         we get 
      $b_j(\gamma^{-1})=-\pi_{\Psi(\gamma^{-1})} b_j(\gamma)$.   Similarly if we set 
      $\gamma_1=g_i$ and $\gamma_2=\tilde g_i^{-1}$ then we get 
       $$b_j(\tilde g_i^{-1} g_i)=b_j(g_i)+\pi_{\Psi(g_i)} b_j(\tilde g_i^{-1})=b_j(g_i)+\pi_{\Psi(g_i)} 
       (-\pi_{\Psi(\tilde g_i^{-1})} b_j(\tilde g_i))=b_j(g_i)- \pi_{\Psi(\tilde g_i^{-1}g_i)} b_j(\tilde g_i). $$ 
            The assumption implies $\Psi(\tilde g_i^{-1} g_i)\ra (\text{Id}, \text{Id})$.  
            On the other hand,   by  Corollary  \ref{bounded s}   there is a constant $C>0$ such that $||b_j(\gamma)||\le C$   for all $\gamma\in \Gamma$.
             It follows that 
             $b_j(\tilde g_i^{-1} g_i)  \ra s-\tilde s$   pointwise 
         as $i\ra \infty$.    
            Now it suffices to show that if $\{\gamma_i\}$ is a sequence in $\Gamma$ such that $\Psi(\gamma_i)\ra (\text{Id}, \text{Id})$  and  $b_j(\gamma_i)\ra s$  pointwise for some 
           $s\in E_j$, then $s\equiv 0$.

            We suppose $s\not\equiv 0$ and shall get a contradiction.  
              There is some $\bar h\in \mathfrak n/{\mathfrak w}$ such that $ s(\bar h)\not=0$.
              Fix some $m\ge 1$ such that 
                $m\cdot \frac{|s(\bar h)|_j}{2}>C\cdot \bar d_{CC}(\bar h,0)^{\frac{j}{\alpha}}.$
             Note $\Psi(\gamma^k_i)\ra (\text{Id}, \text{Id})$   for $1\le k\le m$.  On the other hand,   an easy  induction using Lemma \ref{b_jiscocycle}
               implies 
             $$b_j(\gamma^m_i)=b_j(\gamma_i)+\pi_{\Psi(\gamma_i)} b_j(\gamma_i)+\cdots +\pi_{\Psi(\gamma^{m-1}_i)} b_j(\gamma_i).$$
               Since $\Psi(\gamma^k_i)\ra (\text{Id}, \text{Id})$   and  $b_j(\gamma_i)\ra s$  pointwise,  we have 
                $|\pi_{\Psi(\gamma^{k}_i)} b_j(\gamma_i)(\bar h)-s(\bar h)|_j<\frac{|s(\bar h)|_j}{2}$  for all $0\le k\le m-1$ and all sufficiently large $i$.
              It follows from the triangle inequality that 
              $|b_j(\gamma^m_i)(\bar h) -m s(\bar h)|_j<m\cdot \frac{|s(\bar h)|_j}{2}$ and hence 
                   $|b_j(\gamma^m_i)(\bar h)|_j> m\cdot \frac{|s(\bar h)|_j}{2}>C\cdot \bar d_{CC}(\bar h,0)^{\frac{j}{\alpha}}$, contradicting the fact that  
                   $||b_j(\gamma)||\le C$ for all $\gamma\in \Gamma$.

     \end{proof}

     Lemma \ref{uniquelimit}  
      in particular implies   that  $b_j(\gamma)=b_j(\tilde \gamma)$   holds  whenever 
     $\Psi(\gamma)=\Psi(\tilde \gamma)$ (by taking $g_i=\gamma$ and $\tilde g_i=\tilde \gamma$).   
     Since $b_j(\gamma)$  lies in the closed ball $\bar B(0, C)\subset E_j$   for all $\gamma\in \Gamma$ and $\bar B(0, C)$ is compact in the topology of pointwise convergence, Lemma \ref{uniquelimit}      implies   that 
      for   any $(A, B)\in G$  and  any sequence $\{\gamma_i\}$ satisfying 
       $\Psi(\gamma_i)\ra (A,B)$,  the sequence $b_j(\gamma_i)$ converges to some $s\in \bar B(0, C)$ pointwise and $s$ is independent of the choice of the sequence $\{\gamma_i\}$.    We denote this $s$ by $\tilde b_j(A,B)$. It follows that the map $\tilde b_j: G\ra E_j$ is well-defined  and continuous, where $E_j$ is equipped with the topology of pointwise convergence.   
       This map $\tilde b_j$ is the translational part of the affine action.

        To see  that the map $\tilde b_j$ is actually the translational part of an affine action, we need to show 
               that $\tilde b_j$ is a $1$-cocycle  associated to the $G^*$-module $E_j$; that is, 
                $$\tilde b_j((A_2, B_2)(A_1, B_1))=\tilde b_j(A_1, B_1)+\pi_{(A_1, B_1)}\tilde b_j(A_2,B_2)$$  
                 holds  for all 
                $(A_1, B_1), (A_2, B_2)\in G$.   This follows  from   Lemma \ref{b_jiscocycle}  and the definition of $\tilde b_j$.  
                 The associated affine action of $G^*$ on $E_j$ is given by
          $$(A,B)\cdot c=\pi_{(A,B)}c +\tilde b_j(A,B),$$   where $c\in E_j$ and $(A,B)\in G$.  
  Since $\tilde b_j$ is continuous, it is now easy to see that the affine action $G^*\times E_j\ra E_j$ is separately continuous, where  $E_j$ is equipped with the topology of pointwise convergence.

       \subsection{Existence of fixed point in $\mathcal H_j$}

       Fix some  $c\in \mathcal H_j\subset E_j$, let $G^*\cdot c$ be the orbit of $c$ under the affine action of $G^*$. Let $K$ be the closure of  the convex hull of 
       $G^*\cdot c$  in $E_j$ with respect to the topology  of pointwise convergence.     In order to use Day's fixed point theorem to obtain a fixed point in  $\mathcal H_j$, we need to show that  $K$ is compact and lies in  $\mathcal H_j$.   
       The following lemma will be useful for this purpose. 
       
       \begin{lemma}\label{K lies in H_j}    
                       Let $Y\subset \mathcal H_j   $ be  bounded   with respect to the norm 
                       $||\cdot||$.     Then $\bar Y\subset \mathcal H_j$, where $\bar Y$ is the closure of $Y$ in $E_j$  in the topology  of pointwise convergence.

               \end{lemma}
               
               \begin{proof}
                Since $Y$ is bounded in $(E_j, ||\cdot||)$,  there is a constant $C_0>0$ such that 
                $||s||\le C_0$ for any $s\in Y$.  It follows that 
                  for any $i\ge 1$ and  any compact subset $M\subset  \mathfrak n/\mathfrak w$, there is a constant $C(C_0, M, i)>0$ such that $|s^{(i)}(x)|\le C(C_0, M,i)$  for all $x\in M$  and all $s\in Y$.    
                Let $s_k\in Y$, $k=1, \cdots$  and $s\in E_j$ 
               be such that $s_k(x)\to s(x)$ for any $x\in \mathfrak n/\mathfrak w$.  The dominated convergence theorem then implies  $\int_\gamma [s_k, \theta_H(x)]\to \int_\gamma [s, \theta_H(x)]$ for any horizontal curve $\gamma$ in $\mathfrak n/\mathfrak w$.  This implies $s^{(1)}$ is defined  and 
               $s_k^{(1)}\to s^{(1)}$  pointwise.  Now an induction shows $s^{(i)}$   is defined for all $i$ and so $s\in  \mathcal H_j   $.

               \end{proof}

           If  $\alpha$ is not an integer, then  $\mathcal H_j=E_j$ so $K\subset \mathcal H_j$ holds automatically. 
       To     show $K\subset \mathcal H_j$, we may assume $\alpha\ge 2$ is an integer.    
                            We first prove  $G^*\cdot c\subset \mathcal H_j$.
                          By the definition of the affine action we need to show that   for any $1\le j\le \alpha-1$: \newline
                         (1)  $\pi_{(A,B)} c\in \mathcal H_j$ for any $(A,B)\in G$  and any $c\in \mathcal H_j$;  see Lemma \ref{linearpart}.\newline
                         (2) $\tilde b_j(A,B)\in \mathcal H_j$    for any $(A,B)\in G$; see the paragraph before Lemma \ref{transpart}.  \newline

                          \begin{Le}\label{linearpart}
                          $\pi_{(A,B)} c\in \mathcal H_j$ for any integer $1\le j< \alpha$,  $(A,B)\in G$  and any $c\in \mathcal H_j$.
                          \end{Le}
                          
                     \begin{proof}   Let  $(A,B)\in G$  and  $c\in \mathcal H_j$.   Then there is a sequence $\{\gamma_i\}$ in $\Gamma$ such that $\Psi(\gamma_i)\ra (A,   B)$.   Notice that $\pi_{\Psi(\gamma_i)}c\ra \pi_{(A,B)}c$ pointwise.      Since 
                     $||\pi_{\Psi(\gamma_i)}c||=||c||$,   by Lemma \ref{K lies in H_j}  
                        to show $\pi_{(A,B)} c\in \mathcal H_j$   it suffices to show 
                     $\pi_{\Psi(\gamma)}c \in \mathcal H_j$  for all $\gamma\in \Gamma$.    
                     For this we shall show that there is a biLipschitz  shear map 
                      $f_2(g)=g*\tilde s(\bar g)$ with $\tilde s_j=\pi_{\Psi(\gamma)}c$. The lemma then follows from     Proposition \ref{shear-bilip}.

                      Pick   
                     $c\in \mathcal H_j$  and $\gamma\in \Gamma$.    Let $s: \mathfrak n/\mathfrak w\ra Z(\mathfrak w)$ be   given by $s=\sum_i s_i$, where 
                        $s_j=c$,  $s_i=0$  for 
                          $1\le i\le  \alpha$, $i\not= j$   and $s_{k\alpha+i}=s_i^{(k)}$ for 
                            $k\ge 1$ and $1\le i\le \alpha$.   By Proposition \ref{construction},  
                                the shear map $f(g)=g*s(\bar g)$ is biLipschitz. 
                           We shall show $f_2:=L_{f_1(0)^{-1}}\circ f_1$ with $f_1:= \gamma^{-1} f \gamma$ is the desired biLipschitz shear map. 
                           
                           Notice that for any two maps $F, G: N\ra N$  and any $p\in N$, we have 
                            $(F\circ G)_p=F_{G(p)}\circ G_p$. 
                            It is clear that $f_1$ is a biLipschitz map of $N$ that   permutes the cosets of $W$.  
                            Since $f_p|_W=\text{Id}$ and 
                            $\gamma_p|_W=A_\gamma$, we see that $(f_1)_p|_W=\text{Id}$.  On the other hand, as $f$ induces the identity map on $N/W$,  so does $f_1$.  Hence the assumption of Lemma \ref{sj} are satisfied and $f_1$ has  a   compatible expression
                             $f_1(h*w)=f_1(0)*h*w*\tilde s(\bar h)$ for some map 
                             $\tilde s:  {\mathfrak n}/{\mathfrak w}\ra Z(\mathfrak w)$.   It follows that  $f_2(h*w)=h*w*\tilde s(\bar h)$ is a biLipschitz shear map.  It remains to show 
                              $\tilde s_j=\pi_{\Psi(\gamma)}c$.  Note 
                               $\tilde s_j=\pi_j(f_1(0)^{-1}*f_1(h))$.

                            Fix a compatible expression  $\gamma(h*w)=\gamma(0)* B_\gamma h * A_\gamma w* A_\gamma s_\gamma(\bar h)$ of $\gamma$.   We calculate 
                             $f_1(h)$:
                              \begin{align*}
            f_1(h)=         \gamma^{-1} f \gamma(h)&=\gamma^{-1} f (\gamma(0) * B_\gamma h * A_\gamma s_\gamma(\bar h))\\
                     &=\gamma^{-1}(\gamma(0) *B_\gamma h * A_\gamma      s_\gamma(\bar h) *s(\overline{\gamma(0)} *\bar{B}_\gamma \bar h) )\\
                     &=h*A^{-1}_\gamma s(\overline{\gamma(0)} *\bar{B}_\gamma \bar h),
                     \end{align*}
                         where the last equality follows from the fact that $\gamma$ is a bijection  and 
                        $$\gamma (h*\{A^{-1}_\gamma s(\overline{\gamma(0)} *\bar{B}_\gamma \bar h)\})=
                        \gamma(0)* B_\gamma h*s(\overline{\gamma(0)} *\bar{B}_\gamma \bar h)* A_\gamma s_\gamma(\bar h).$$
                           In particular,    $f_1(0)=A_\gamma^{-1}s(\overline{\gamma(0)})$.     Now 
                           $   f_2(h)=A^{-1}_\gamma  s(\overline{\gamma(0)})^{-1}*h*A^{-1}_\gamma s(\overline{\gamma(0)} *\bar{B}_\gamma \bar h)$ and from this we see
    $$s_{f_2, j}(\bar h)=A^{-1}_\gamma s_j(\overline{\gamma(0)} *\bar{B}_\gamma \bar h)- A^{-1}_\gamma  
   s_j(\overline{\gamma(0)})=
   A^{-1}_\gamma c(\overline{\gamma(0)} *\bar{B}_\gamma \bar h)- A^{-1}_\gamma  
   c(\overline{\gamma(0)})=(\pi_{\Psi(\gamma) }c)(\bar h).$$

                           \end{proof}

                        
                  Let     $(A,B)\in G$. By the definition of $\tilde b_j$,   there is a sequence $\{\gamma_i\}$ in $\Gamma$ such that
                   $b_j(\gamma_i)\ra \tilde b_j(A,B)$ pointwise.   Since $b_j(\gamma_i) $ 
                       lies in the  closed  ball $\bar B(0,C)$ (see Corollary \ref{bounded s}), Lemma \ref{K lies in H_j}      implies  $\tilde b_j(A,B)\in {\mathcal H}_j$ provided  $b_j(\gamma)\in \mathcal H_j$    for all $\gamma\in \Gamma$.      This is true by (2) of the following lemma.  
                        
                       \begin{Le}\label{transpart}
                           Let $F: \mathfrak n \ra \mathfrak n$  be  a fiber similarity  map
                            and $F(h*w)=F(0)* B h *Aw*A s(\bar h)$   a compatible expression of $F$.   Then \newline
                           (1)  for $1\le j<  \alpha$,   $s_{j}\in  E_j $   with  $||s_j||$  bounded above  by a constant depending only on   $H$ and the  biLipschitz constant of $F$.    
                             \newline
                 (2)       if $\alpha$ is an integer, then     $s_{j}\in \mathcal H_j$
                         for any $1\le j< \alpha$. \newline  
                          (3)  if $\alpha$ is an integer, then  $s_{\alpha}\in  E_\alpha $   with  $||s_\alpha||$  bounded above  by a constant depending only on   $H$,   the  biLipschitz constant of $F$  and the map $B$.
                       \end{Le}
                       
                       \begin{remark}
  The    map  $F$ may admit many different compatible expressions: both $B$ and $s_\alpha$ may change and as a result their 
       Lipschitz constants may be arbitrarily large while  $F$ is fixed.

       \end{remark}

                       \begin{proof}
                           We will use the fact that there is a constant $L_1\ge 1$  depending only on $H$ such that 
                        $(1/{L_1}) \cdot d^{\frac{1}{\alpha}}_{CC}(0, \bar h)\le d(0, h)\le L_1 \cdot d^{\frac{1}{\alpha}}_{CC}(0, \bar h)$
                         for any $h\in H$, where $d_{CC}$ is a  Carnot metric on $\mathfrak n/\mathfrak w$.   
                         Since $F$ is $L_2$-biLipschitz for some $L_2\ge 1$,  for any $h_0, h\in H$ we have
                          $$d(0, F(h_0)^{-1}* F(h_0*h))
                           =d(F(h_0),     F(h_0*h))
                           \le    L_2\cdot d(h_0, h_0*h)=L_2\cdot d(0, h)\le 
                          L_2 L_1 d^{\frac{1}{\alpha}}_{CC}(0, \bar h).$$

                          There is a constant $L>0$ depending on $B$ such that 
                          for any $h\in H$ we have 
                          $$d(0, Bh)\le  L \cdot d^{\frac{1}{\alpha}}_{CC}(0, \bar B \bar h)
                          =L \lambda_{\bar B}^{\frac{1}{\alpha}}  \cdot d^{\frac{1}{\alpha}}_{CC}(0, \bar h), $$ 
                             where $\lambda_{\bar B}$ denotes the similarity constant of $\bar B$.    We also have 
                          \begin{align*}
                           &d(0, (Bh)^{-1}*F(h_0)^{-1}* F(h_0*h))
                           =d(Bh,  F(h_0)^{-1}* F(h_0*h))\\
                           &\le d(Bh, 0)+d(0, F(h_0)^{-1}* F(h_0*h))
                     \le      (L  \lambda_{\bar B}^{\frac{1}{\alpha}}+L_2L_1) d^{\frac{1}{\alpha}}_{CC}(0, \bar h).
                          \end{align*}

                              Let $F_0: \mathfrak n\ra \mathfrak n$ be the shear map  given by $F_0(g)=g*s(\bar g)$.

                          \noindent
                          {\bf{Claim:}}   (a)  $\pi_j  (F(h_0)^{-1}* F(h_0*h))=
                          As_j(\overline{h_0}*\bar h)-A s_j(\overline{h_0})$  holds for $1\le j<  \alpha$;\newline 
                    (b)            $\pi_{i} ((Bh)^{-1}*F(h_0)^{-1}* F(h_0*h))
                             =A \pi_{i}(h^{-1}*F_0(h_0)^{-1}* F_0(h_0*h))$
                              for any integer $i$ that is not an integral multiple of $\alpha$,  where $\pi_i: \mathfrak n\ra V_{i}= W_{i}$  is the projection with respect to the decomposition  $\mathfrak n=\oplus  V_{\lambda_k}$.     \newline
                              (c)   $\pi_{\alpha} ((Bh)^{-1}*F(h_0)^{-1}* F(h_0*h))=
                               As_\alpha(\overline{h_0}*\bar h)-A s_\alpha(\overline{h_0})$.

                 We first assume the Claim and finish the proof of the lemma.              
                              Part (a) of the claim implies  for $1\le j<\alpha$,  
                              \begin{align*}
                              |s_j(\overline{h_0}*\bar h)- s_j(\overline{h_0})|^{\frac{1}{j}}  & \le L_2\cdot |\pi_j(F(h_0)^{-1}* F(h_0*h))|^{\frac{1}{j}}\\
                             & \le  L_2\cdot  L_2L_1 d^{\frac{1}{\alpha}}_{CC}(0, \bar h)\\
                     & =  L^2_2 L_1  d^{\frac{1}{\alpha}}_{CC}(\overline{h_0}, \overline{h_0}*\bar h),
                     \end{align*}     
                               hence    (1) holds.  
                              Now notice $h^{-1}*F_0(h_0)^{-1}* F_0(h_0*h)=K(\overline{h_0}, \overline{h_0*h})$, where  
                              $K(\bar g_1,   \bar g_2)$ is defined for $g_1, g_2\in \mathfrak n$  
                                before Lemma \ref{shear-bilip-lemma}. 
                         So  when $\alpha$ is an integer,  Part (b) of the claim implies 
                          \begin{equation}\label{kalpha}
                          | \pi_{k\alpha+j}(K(\overline{h_0}, \overline{h_0*h}))|^{\frac{1}{k\alpha+j}}
                          \le L_2( L   \lambda_{\bar B}^{\frac{1}{\alpha}}+L_2L_1)  \cdot d^{\frac{1}{\alpha}}_{CC}(0, \bar h)
                          \end{equation}
                            for all $k\ge 0$ and all $1\le j<\alpha$.  
                           Then   (2)  follows    
                     since     by   (\ref{kalpha})   the assumption of Lemma \ref{shear-bilip-lemma}  is satisfied.     Finally (3) follows from Part (c):
 $$|s_\alpha(\overline{h_0}*\bar h)- s_\alpha(\overline{h_0})|^{\frac{1}{\alpha}}\le 
 L_2 |\pi_{\alpha} ((Bh)^{-1}*F(h_0)^{-1}* F(h_0*h))|^{\frac{1}{\alpha}}\le L_2
                     (L  \lambda_{\bar B}^{\frac{1}{\alpha}}+L_2L_1) d^{\frac{1}{\alpha}}_{CC}(0, \bar h).$$   
                       
                       Next we prove the claim.   Write $h_0*h=h_1*w_1$ for some $h_1\in H$, $w_1\in \oplus_iW_{i\alpha}$.   Applying $\pi_\alpha$ to both sides we get
                       $\pi_\alpha(h_0)+\pi_\alpha(h)=\pi_\alpha(h_1)+\pi_\alpha(w_1)$. As 
                        $V_\alpha=W_\alpha\oplus H_1$ is a direct sum, we obtain
                           $\pi_\alpha(h_0)+\pi_\alpha(h)=\pi_\alpha(h_1)$ and $\pi_\alpha(w_1)=0$;  hence 
                            $\pi_\alpha(Bh_0)+\pi_\alpha(Bh)=\pi_\alpha(Bh_1)$  and 
                           $w_1\in \oplus_{i\ge 2}W_{i\alpha}$.  
                        Note
                        $\bar h_1=\overline{h_0}*\bar h$. 
                       We have 
                        $F(h_0)=F(0)*Bh_0*A s(\overline{h_0})$  and  $F(h_0*h)=F(h_1*w_1)=F(0)* B h_1* A w_1*A s(\overline{h_0}* \bar h)$. Using $Bh* Aw*(Bh)^{-1}=A(h*w*h^{-1})$ twice we get:
                        \begin{align*}
                       &(Bh)^{-1}* (F(h_0))^{-1}*F(h_0*h)\\
                        &= (Bh)^{-1}*A(-s(\overline{h_0}))* (B h_0)^{-1}*Bh_1*A w_1*A s(\overline{h_0}* \bar h)\\
                        & =(Bh)^{-1}*(Bh_0)^{-1}*\{Bh_0*A(-s(\overline{h_0}))* (B h_0)^{-1}\}*Bh_1*A w_1*A s(\overline{h_0}* \bar h)\\
                        &=(Bh)^{-1}*(Bh_0)^{-1}*A(h_0*(-s(\overline{h_0}))*h_0^{-1})*Bh_1*A w_1*A s(\overline{h_0}* \bar h)\\
                        &=(Bh)^{-1}*(Bh_0)^{-1}*Bh_1*\{(Bh_1)^{-1}*A(h_0*(-s(\overline{h_0}))*h_0^{-1})*Bh_1\}*A w_1*A s(\overline{h_0}* \bar h)\\
                        &=(Bh)^{-1}*(Bh_0)^{-1}*Bh_1*A(h_1^{-1}*h_0*(-s(\overline{h_0}))*h_0^{-1}*h_1)*Aw_1*A s(\overline{h_0}* \bar h)\\
                        &=(Bh)^{-1}*(Bh_0)^{-1}*Bh_1*A(w_1*h^{-1}*(-s(\overline{h_0}))*h*w^{-1}_1)*Aw_1*A s(\overline{h_0}* \bar h)\\
                        &=(Bh)^{-1}*(Bh_0)^{-1}*Bh_1*Aw_1*A(h^{-1}*(-s(\overline{h_0}))*h)*A s(\overline{h_0}* \bar h),\\
                        \end{align*}
                        where for  the  sixth equality we used $h_1*w_1=h_0*h$.   
                         We first prove Part (c): as $\pi_\alpha(Bh_0)+\pi_\alpha(Bh)=\pi_\alpha(Bh_1)$  and 
                           $w_1\in \oplus_{i\ge 2}W_{i\alpha}$, the above calculation yields
                            \begin{align*}
                            \pi_{\alpha} ((Bh)^{-1}*F(h_0)^{-1}* F(h_0*h))
                            &=\pi_\alpha(A(h^{-1}*(-s(\overline{h_0}))*h)*A s(\overline{h_0}* \bar h))\\ &=
                               As_\alpha(\overline{h_0}*\bar h)-A s_\alpha(\overline{h_0}).
                               \end{align*}

                         The above calculation also gives
                        $(Bh)^{-1}*(F(h_0))^{-1}*F(h_0*h)=P*Az$, where  $P=(Bh)^{-1}*(Bh_0)^{-1}*Bh_1*Aw_1$ and
                         $z=s(\overline{h_0}* \bar h)+h^{-1}*(-s(\overline{h_0}))*h$.
                              Since $z$ lies in $Z(\mathfrak w)$,  the BCH formula gives 
                               $(Bh)^{-1}*(F(h_0))^{-1}*F(h_0*h)= P+Az+Q$, where $Q$ is a sum of iterated brackets of $P$ and 
                                $Az$,   and $Az$ appears exactly once in each of these  iterated brackets.    Notice that the BCH formula also implies  $P$ is a sum of iterated brackets of the terms  $Bh$, $Bh_0$, $Bh_1$, $Aw_1$. Now $[Bh, Aw]=A[h, w]$ and the Jacobi identity imply that $Q=A\tilde Q$, where $\tilde Q$ is obtained from  the sum of iterated brackets that gives rise to  $Q$ by replacing  $Bh$,    $Bh_0$, $Bh_1$, $Aw_1$, $Az$ by  $h$,  $h_0$, $h_1$, $w_1$, $z$ respectively.    So  $(Bh)^{-1}*(F(h_0))^{-1}*F(h_0*h)= P+Az+A\tilde Q$.
                                  
                                  Now comparing the formulas $F(h*w)=F(0)* Bh*Aw* As(\bar h)$  and $F_0(h*w)=h*w*s(\bar h)$,
                               and     repeating  the above calculation for $h^{-1}* (F_0(h_0))^{-1}* F_0(h_0*h)$ (we only need to replace $A$ with $\text{Id}_W$ and $B$ with $\text{Id}_H$), we get
                                $$h^{-1}* (F_0(h_0))^{-1}* F_0(h_0*h)=\tilde P +z+\tilde Q,$$
                                  where $\tilde P$ is obtained from  the sum of iterated brackets that gives rise to  $P$ by replacing  $Bh$,    $Bh_0$, $Bh_1$, $Aw_1$  by  $h$,  $h_0$, $h_1$, $w_1$ respectively. 
                         Now  Part (b) of the claim follows since $P, \tilde P\in \oplus_i V_{i\alpha}$.   
                        
                      Next we prove Part (a) of the Claim.   
                      As $h, Bh\in \oplus_{\lambda_i\ge \alpha}V_{\lambda_i}$, for $1\le j<\alpha$,  by Claim (b) we have 
                      \begin{align*}
                      \pi_j  (F(h_0)^{-1}* F(h_0*h))
                      &=\pi_{j} ((Bh)^{-1}*F(h_0)^{-1}* F(h_0*h))
                             =A \pi_{j}(h^{-1}*F_0(h_0)^{-1}* F_0(h_0*h))\\
&=A \pi_{j}(F_0(h_0)^{-1}* F_0(h_0*h))
=A\pi_j((-s(\bar{h}_0)*h_0^{-1}*h_0*h*s(\bar{h}_0*\bar h))\\
&=A\pi_j(s(\bar{h}_0*\bar h)-s(\bar{h}_0))
=A(s_j(\bar{h}_0*\bar h)-s_j(\bar{h}_0)).  
\end{align*}

                       
                       \end{proof}

           \begin{corollary}\label{bounded s}
           Let $\Gamma$ be a  fiber   similarity  
            group of $N$. 
              Then there is a constant $C>0$ such that 
               $||s_{\gamma, j}||<C$ for all $\gamma\in \Gamma$ and all $1\le j< \alpha$.  
         
           \end{corollary}
           
           \begin{proof}  The inequality $\frac{C_\gamma}{L^2\Lambda }\le\lambda_{A_\gamma}\le L^2\Lambda  C_\gamma$ we obtained in the proof of 
           Lemma \ref{dilation-relation}  implies that there is a constant $K_0>0$ 
             such that $e^{-\log(\lambda_{A_\gamma}) D}\circ \gamma$ is $K_0$-biLipschitz for  every $\gamma\in \Gamma$.    
              An  element $\gamma\in \Gamma$   with compatible expression 
               $\gamma(h*w)=\gamma(0) * B_\gamma h * A_\gamma w* A_\gamma s_\gamma(\bar h)$ 
              can be written as  $\gamma(h*w)=\gamma(0)* e^{\log(\lambda_{A_\gamma}) D}(B'_\gamma h* A'_\gamma w* A'_\gamma 
           s_\gamma(\bar h))$, where $s_\gamma$ is as before  and $A'_\gamma$ is an isometric  graded automorphism of $  (W, d_{CC})$.   
              It follows that  $e^{-\log(\lambda_{A_\gamma}) D}\circ \gamma$ has  a compatible expression
             $e^{-\log(\lambda_{A_\gamma}) D}\circ \gamma(h*w)=a'_\gamma *B'_\gamma h* A'_\gamma w* A'_\gamma s_\gamma(\bar h)$. 
           The corollary  now  follows by applying Lemma \ref{transpart}
              to $e^{-\log(\lambda_{A_\gamma}) D} \circ \gamma$.

           \end{proof}


           


           
           We   have shown that $G^*\cdot c\subset  {\mathcal H}_j$.   Corollary \ref{bounded s}   implies   $G^*\cdot c$ is bounded.  
            By    Lemma \ref{K lies in H_j}      $K\subset  {\mathcal H}_j$. 
             We next show that      the affine action     has a fixed point  in    $K$. For this purpose we use

                   \noindent{\bf Theorem (Day's fixed point theorem) }\cite{Da61}
\emph{ Let  $K$ be a compact convex subset  of a locally convex  topological   vector 
space  $E$   and let $\Gamma$ be a locally compact group that acts on $K$ by  affine transformations. If $\Gamma$ is amenable 
  and the action  $\Gamma\times K\ra K,  (\gamma, x)\mapsto \gamma\cdot x,$ is separately continuous,     
then the action of $\Gamma$ has a global fixed point in $K$. }

                          \begin{Le}\label{fixedless}
                          The  affine action  associated to the $1$-cocycle $\tilde b_j$   has a fixed point $c\in \mathcal H_j$.  
                                         
                           \end{Le}
                           
           
           \begin{proof}             
                            We equip $E_j$ with the topology of pointwise convergence. Then $E_j$ is a 
                           locally convex topological vector space   and $K$ is a compact convex subset (as $K$ is bounded).   We have observed that the affine action $G^*\times K\ra K$ is separately continuous.  Since $G^*$ is amenable,  
                           by Day's fixed point theorem,    the affine action  has a fixed point in $K$.

                          \end{proof}




       \subsection{Eliminating $s_j$}


                           \begin{Le} Let $1\le j<\alpha$.\newline 
                            Suppose   
                              the  affine action  of $G^*$  on $E_j$  has a fixed point $c\in \mathcal H_j$.    Let     $s: \mathfrak n/\mathfrak w\ra Z(\mathfrak w)$ be the  map   provided by Proposition \ref{construction} 
                             satisfying   $s_j=c$ and  $s_i=0$ for $1\le i\le \alpha$, $i\not=j$
                                 such that the corresponding shear map 
                             $F_0(g)=g*s(\bar g)$ is  biLipschitz. 
                               Then 
                             every   element $\tilde \gamma\in F_0\Gamma F_0^{-1}$ has a  compatible  expression 
                             $\tilde \gamma(h*w)=\tilde \gamma(0)*B_{\gamma}h*A_{\gamma}w*A_{\gamma} \tilde s_{\gamma}(\bar h)  $ with 
                              $\tilde s_{\gamma, j}=0$.\newline
                           
                           \end{Le}


                           \begin{proof}

                            For $\gamma\in \Gamma$,  denote $\tilde \gamma=F_0\circ \gamma\circ F_0^{-1}$.        Consider a compatible expression  $\gamma(h*w)=\gamma(0) *B_\gamma h*A_{\gamma}w* A_\gamma s_{\gamma}(\bar h)$ of $\gamma$.   Since $c$ is a fixed point of the affine action, we have $\pi_{\Psi(\gamma)} c+s_{\gamma,j}=c$ for all $\gamma\in \Gamma$.  
                            Note $F_0$ and $F_0^{-1}$ have the expressions: $F_0(h*w)=h*w*s(\bar h)$,
                             $F_0^{-1}(h*w)=h*w*(-s(\bar h))$.   We now calculate 
                             \begin{align*}
                             \tilde \gamma(h)=F_0\circ \gamma (h*(-s(\bar h)))
                             &=F_0(\gamma(0)*B_\gamma h*A_\gamma[s_\gamma(\bar h)*(-s(\bar h))])\\
                             &=\gamma(0)*B_\gamma h*A_\gamma[s_\gamma(\bar h)*(-s(\bar h))]*s(\overline{\gamma(0)}*\bar B_\gamma \bar h).
                             \end{align*}
                           In particular,  $\tilde\gamma (0)=\gamma(0)*s(\overline{\gamma(0)})$.  Now
                           \begin{align*}
                           s_{\tilde \gamma,j}(\bar h)  &=A^{-1}_\gamma\pi_j(\tilde\gamma(0)^{-1}*\tilde \gamma(h))\\
                           &=A^{-1}_\gamma\pi_j \{(-s(\overline{\gamma(0)}))*B_\gamma h*A_\gamma[s_\gamma(\bar h)*(-s(\bar h))]*s(\overline{\gamma(0)}*\bar B_\gamma \bar h)\}\\
                           &=A^{-1}_\gamma\pi_j \{(-s(\overline{\gamma(0)}))*A_\gamma[s_\gamma(\bar h)*(-s(\bar h))]*s(\overline{\gamma(0)}*\bar B_\gamma \bar h)\}\\
                           &=A^{-1}_\gamma\{-c(\overline{\gamma(0)})+A_\gamma s_{\gamma,j}(\bar h)+A_\gamma (-c(\bar h))+c(\overline{\gamma(0)}*\bar B_\gamma \bar h)\}\\&=s_{\gamma,j}(\bar h)-c(\bar h)+A^{-1}_\gamma c(\overline{\gamma(0)}*\bar B_\gamma \bar h)-A^{-1}_\gamma c(\overline{\gamma(0)})\\
                           &=s_{\gamma,j}(\bar h)-c(\bar h)+(\pi_{\Psi(\gamma)} c)(\bar h)=0.
                           \end{align*}

                           \end{proof}
                           
        The proof of 
                  Proposition \ref{elliminatedim(W)ge 2}  is now complete.

       \section{Conformal structures in the $V_\alpha$ direction}\label{cstructure}

          We continue the proof of Theorem \ref{main-uniform}.  In  Section \ref{eliminate}
            we showed that we can get rid of $s_{\gamma, j}$ for $j<\alpha$ 
              after a conjugation.   
       In this section we  continue to assume    $\dim(W)\ge 2$, $\dim(N/W)\ge 2$.     We    will  first show that, if  $\alpha$ is an integer, then after a  conjugation, $s_{\gamma, \alpha}: \mathfrak n/\mathfrak w\ra Z_\alpha(\mathfrak w)$  is a  Lie group homomorphism for every  $\gamma\in \Gamma$,   and then we  complete the proof of Theorem \ref{main-uniform} under the  assumptions   $\dim(W)\ge 2$, $\dim(N/W)\ge 2$.
          Here we are using the  identification between  simply connected nilpotent Lie groups and their Lie algebras via the exponential map.  

        We  will  imitate the proof of Tukia's theorem, with the last step different.   As   biLipschitz maps of  Carnot-by-Carnot groups  in general are not differentiable,   we    can not directly work with the differentials.   Instead   we   look for differentiability in  the $V_\alpha$ direction. 
          For each biLipschitz map $F: \mathfrak n\ra \mathfrak n$ and each $p\in \mathfrak n$, we  consider the
           differential of the map 
           $$F_{p, \alpha}:=(\pi_{\alpha}\circ   F_p)|_{V_\alpha}: V_\alpha\ra V_\alpha,$$
            with  $F_p=L_{F(p)^{-1}}\circ F\circ   L_p$,    where  
              $L_g$ is left translation  by the element $g$  and  $\pi_{\alpha}: \mathfrak n\ra V_\alpha$ is the projection with respect to the decomposition $\mathfrak n=\oplus_j V_{\lambda_j}$.  
           
           In the next two subsections   we assume  $\alpha$ is an integer.  
           
        \subsection{Differential of $F_{p, \alpha}$}\label{differential}
        
       In this  subsection we first find a formula for the differential of $F_{p, \alpha}$ at $0$ and then show that this differential satisfies  the chain rule
        (Lemma \ref{chainrule}).

       Let $(N,D)$ be  Carnot-by-Carnot and 
        $F:   \mathfrak n\ra \mathfrak n$ be a  biLipschitz map  satisfying the 
        assumptions of Lemma \ref{sj}. 
          Then  $F$ has a compatible expression 
        $F(h*w)=F(0)*Bh*Aw *A  s(\bar h).$
          Write  
             $s=\sum_j s_j$ with  $s_j: \mathfrak n/\mathfrak w\ra Z_j(\mathfrak w)$.
             By Lemma \ref{transpart}  
      $s_j$ is 
             $\frac{j}{\alpha}$-Holder for each $1\le j\le \alpha$.                
             

       To simplify the calculations we will work in a suitable quotient of $\mathfrak n$.  
       Notice that $\oplus_{\lambda_j>\alpha} V_{\lambda_j}$ is an ideal of $\mathfrak n$.  
       Denote  $\bar{\mathfrak n}_\alpha= \mathfrak n/{(\oplus_{\lambda_j>\alpha} V_{\lambda_j})}$ and let 
       $P_\alpha: \mathfrak n\ra \bar {\mathfrak n}_\alpha$ be the quotient  homomorphism.  
        Observe that if $x\in \oplus_{\lambda_j\ge \alpha}V_{\lambda_j}$, then $P_\alpha(x)$ lies in the center of $\bar {\mathfrak n}_\alpha$. 
      Let $h\in H$, $w\in \mathfrak w$. 
 There is some  $w'\in \mathfrak w$ such that $h+w=h*w'$.
 By applying $P_\alpha$ to both sides  of  $h+w=h*w'$ and noting that $P_\alpha(h)$ lies in the center of 
 $\bar {\mathfrak n}_\alpha$, we get $P_\alpha(w)=P_\alpha(w')$.   
          Now write  $p=h_0*w_0$  with $h_0\in H$ and $w_0\in \mathfrak w$,   
         and  $h_0*h=\tilde h+\tilde w=\tilde h*{\tilde {\tilde w}}$  for some $\tilde h\in H$ and $\tilde w, {\tilde {\tilde w}}\in \mathfrak w$.
      Using BCH formula we see that $\tilde w, {\tilde {\tilde w}}\in \oplus_{j>\alpha}W_j$. 
      By applying $P_\alpha$ to both sides of $h_0*h=\tilde h+\tilde w$ we get $P_\alpha(\tilde h)=P_\alpha (h_0)+P_\alpha(h)$.  
         It follows that $P_\alpha(B\tilde h)=P _\alpha (Bh_0)+P_\alpha(Bh)$ and so  $P_\alpha((Bh_0)^{-1}* B \tilde h)=P _\alpha (Bh)$.

    We are ready to find a formula for $F_{p,\alpha}$.     Recall     $p=h_0*w_0$.     Write   
$$ p*(h+w)=h_0* w_0 *h*w' 
= h_0 * h * (h^{-1}*w_0*h) * w'
= \tilde h*{\tilde {\tilde w}} * (h^{-1}*w_0*h) * w' .$$
    So 
$$ F(p)=F(0) * B h_0 * Aw_0*  A s(\overline{h_0})$$  and
 $$F(p*(h+w))=F(0) * B  \tilde h * A {\tilde {\tilde w}} * A(h^{-1}*w_0*h) * Aw' * As(\overline{h_0} *\bar{h}).  $$
 Therefore     (see explanation after the display formula)
 \bea 
  & & P_\alpha( L_{F(p)^{-1}}\circ F\circ L_p(h+w))\\
&= & P_\alpha(L_{F(p)^{-1}}(F(0) * B  \tilde h * A {\tilde {\tilde w}} * A(h^{-1}*w_0*h) * Aw' * As(\overline{h_0}*\bar{h})))\\
 &=& P_\alpha( As(\overline{h_0})^{-1}* A w_0^{-1}* (B h_0)^{-1} *B  \tilde h * A {\tilde {\tilde w}} * Aw_0 * Aw*
  As(\overline{h_0} *\bar{h}))\\
  &=& P_\alpha( As(\overline{h_0})^{-1}* A w_0^{-1}* (B h_0)^{-1} *B  \tilde h * Aw_0 * Aw*
  As(\overline{h_0} * \bar{h}))\\
 &=&P_\alpha( A s(\overline{h_0})^{-1}* (B h_0)^{-1} *B \tilde h * Aw* A s(\overline{h_0} *\bar{h}))\\
&=& P_\alpha(A s(\overline{h_0})^{-1}* B h*A w* A  s(\overline{h_0} * \bar{h}))\\
&= &P_\alpha(B h + A w + A s(\overline{h_0}*\bar{h})- A s(\overline{h_0})).\\
 \eea 
     For the second equality we used $P_\alpha(h^{-1}*w_0*h)=P_\alpha (w_0)$ (as $P_\alpha(h)$ lies   in the center of $\bar {\mathfrak n}_\alpha$)  and $P_\alpha(w')=P_\alpha(w)$.  
     For  the third equality we used  $P_\alpha({\tilde {\tilde w}})=0$ 
      (as ${\tilde {\tilde w}}\in \oplus_{j>\alpha}W_j$).   For  the fourth equality we used $P_\alpha(A w_0^{-1}* B h_0^{-1} *B  \tilde h * Aw_0)=P_\alpha (B h_0^{-1} *B  \tilde h)$ (as $P_\alpha (B h_0^{-1} *B  \tilde h)$ lies in the center of $\bar {\mathfrak n}_\alpha$).   For the fifth equality we used  $P_\alpha((Bh_0)^{-1}* B \tilde h)=P_\alpha (Bh)$. For the last 
      equality we used the fact that $P_\alpha(Bh)$ lies in the center of $\bar {\mathfrak n}_\alpha$  and that 
      $Aw$, $A s(\overline{h_0}*\bar{h})$ and $A s(\overline{h_0})$ commute with each other (as  $s$ takes values in $Z(\mathfrak w)$).

      After applying $\pi_\alpha$ to both sides  of the above display formula we get
      \begin{equation}\label{pre-derivative}
      \pi_\alpha\circ  F_p(h+w)=B \pi_\alpha(h)+A \pi_\alpha(w)+A s_\alpha(\overline{h_0}*\bar{h})- 
       A s_\alpha(\overline{h_0}).
       \end{equation}  
      This calculation will be used in the proofs of  Lemmas  \ref{psi(F)}  and    \ref{chainrule}.
      When $h\in H_1$ and $w\in W_\alpha$, we have $\pi_\alpha(h)=h$ and $\pi_\alpha(w)=w$ and so 
       $F_{p, \alpha}: V_\alpha\ra V_\alpha$ is given by 
       the formula  
        $$F_{p,\alpha}(h+w)=B h+A w + A s_\alpha(\overline{h_0}*\bar{h})- A s_\alpha(\overline{h_0}).$$

       \begin{lemma}\label{psi(F)}
        Let $\psi(F): \mathfrak n/\mathfrak w\rightarrow V_\alpha$ be the map given by 
       $\psi(F)(\bar h)=Bh_1+As_\alpha(\bar h)$,   where $h_1$ is the $H_1$ component of $h\in H$. Then $\psi(F)$ is Lipschitz with the Lipschitz constant bounded above by a constant depending only on $H$ and  the 
       biLipschitz constant of $F$.
       \end{lemma}


       
        \begin{proof}
         By  setting $w=0$ in  formula   (\ref{pre-derivative})    we   get 
       $$\pi_\alpha(F(p)^{-1}*F(p*h))=B h_1+A s_\alpha(\overline{h_0}*\bar{h})- 
       A s_\alpha(\overline{h_0})=\psi(F)(\bar h_0*\bar h)-\psi(F)(\bar h_0). $$ 
       Since $F$ is $L$-biLipschitz for some $L\ge 1$ we have
       $$|\psi(F)(\bar h_0*\bar h)-\psi(F)(\bar h_0)|^{\frac{1}{\alpha}}\le 
       d(0, F(p)^{-1}*F(p*h))=d(F(p), F(p*h))\le L \,d(p, p*h)=L\, d(0,h).$$
       On the other hand, since $\pi|_H: H\ra {\bar{\mathfrak n}}$ is a linear bijection, $H$ is stable under the automorphisms $e^{tD}$ ($D$ acts by Euclidean dilation on each $V_{\lambda_j}$), and $(\pi|_H)\circ (e^{tD}|_H)=e^{t\bar D}\circ (\pi|_H)$, there is a constant $L_0\ge 1$   that depends on $H$ but independent of the biLipschitz map $F$ such that
        $(1/{L_0}) d^{\frac{1}{\alpha}}_{CC}(0, \bar h)\le d(0,h)\le L_0 d^{\frac{1}{\alpha}}_{CC}(0, \bar h)$ for any $h\in H$. It follows that   $|\psi(F)(\bar h_0*\bar h)-\psi(F)(\bar h_0)|^{\frac{1}{\alpha}}\le LL_0 
       d^{\frac{1}{\alpha}}_{CC}(0, \bar h)$.  
       

       \end{proof}

      Now we can try to compute the differential $dF_{p, \alpha}(0)$ of $F_{p, \alpha}$ at the origin $0$:
      \begin{align*}
      dF_{p, \alpha}(0)(h+w)&=\lim_{t\ra -\infty}\frac{ F_{p, \alpha}(e^{\alpha t}(h+w))}{e^{\alpha t}} \\
      &=\lim_{t\ra -\infty}\frac{B e^{\alpha t}h + A e^{\alpha t}w + A s_\alpha(\overline{h_0}*e^{\alpha t}\bar{h})- A s_\alpha(\overline{h_0}) }{e^{\alpha t}} \\
      &=Bh+Aw+A\lim_{t\ra -\infty}\frac{s_\alpha(\overline{h_0}*e^{t\bar D}\bar{h})- s_\alpha(\overline{h_0})}{e^{\alpha t}}.
      \end{align*}
       In the above we used the fact that $\bar D|_{\bar V_1}=\alpha \cdot \text{Id}_{\bar V_1}$.
        Notice that at a point  $\overline{h_0}$ where the Lipschitz map $s_\alpha: \mathfrak n/\mathfrak w\ra Z_\alpha(\mathfrak w)$ is   Pansu differentiable,
          we have that 
           $\lim_{t\ra -\infty}\frac{s_\alpha(\overline{h_0}*e^{t\bar D}\bar{h})-  s_\alpha(\overline{h_0})}{e^{\alpha t}}=Ds_\alpha(\overline{h_0})(\bar h)$, where $  Ds_\alpha(\overline{h_0}): \mathfrak n/\mathfrak w\ra Z_\alpha(\mathfrak w)$ is the Pansu differential of $s_\alpha$ at $\overline{h_0}$. 
           Recall  that   a Lipschitz map  between Carnot groups  is Pansu differentiable   a.e., 
       see \cite{P89}.  
        Since $s_\alpha$ is Pansu differentiable  at a.e. $\overline{h_0}\in \mathfrak n/\mathfrak w$, we see that at a.e. $p=h_0*w_0\in \mathfrak n$,
          the map $F_{p, \alpha}: V_\alpha \ra V_\alpha$  is differentiable at the origin with differential 
           $dF_{p, \alpha} (0)$ given by:
            \begin{equation}\label{derivative}
             dF_{p, \alpha}(0)(h+w)=  Bh+Aw+ A Ds_\alpha(\overline{h_0})(\bar h).
             \end{equation}
        We shall denote $D_\alpha F(p)=dF_{p, \alpha} (0)$.

      The differential  $D_\alpha F(p)$  should be thought as a counterpart for the restriction of Pansu differential to the first layer, although our groups are not Carnot.  It satisfies the chain rule.
      

 \begin{Le}\label{chainrule}
      Let $F, \tilde F: \mathfrak n\ra \mathfrak n$ be  
      biLipschitz maps  satisfying the 
        assumptions of Lemma \ref{sj}. 
            Then 
         $D_\alpha( F\circ \tilde F)(p)= D_\alpha  F (\tilde F(p)) \circ D_\alpha\tilde F (p)$ for a.e. $p\in \mathfrak n$.
          
      \end{Le}
      
      \begin{proof}  Let $p\in \mathfrak n$ be a point  such that $D_\alpha(F\circ \tilde F)(p) $,   $D_\alpha  F (\tilde F(p))$  and  $ D_\alpha\tilde F (p)$
       all exist.   
      We need to show 
          \begin{equation}\label{chain}
          D_\alpha(F\circ \tilde F)(p)(X)= D_\alpha F (\tilde F(p)) \circ D_\alpha\tilde F (p)(X), \;\; \forall X\in V_\alpha.
          \end{equation}
      Since $V_\alpha=W_\alpha\oplus H_1$, it suffices to establish (\ref{chain}) for $X\in W_\alpha$ and $X\in H_1$.

      The maps $F$ and $\tilde F$ have  compatible expressions
       $F(h*w)=F(0)*Bh*Aw*A s(\bar h)$,  $\tilde F(h*w)= \tilde F(0) *\tilde Bh*\tilde Aw*A \tilde s(\bar h)$.  
      Let $X=w\in W_\alpha$.   By (\ref{derivative}),    $D_\alpha F(p)(w)=Aw$. It follows that 
        ${D_\alpha}  F (\tilde F(p)) \circ D_\alpha\tilde F (p)(w)=D_\alpha F (\tilde F(p))(\tilde Aw)= A (\tilde Aw)$. 
          On the other hand, 
        $   D_\alpha(F\circ \tilde F)(p)(w)=A \tilde A(w)$ as on cosets of $W$  the map $F\circ \tilde F$ acts by $A  \tilde A$. 
         Hence (\ref{chain})   holds for $X\in W_\alpha$. We next consider the case $X=h\in H_1$.



        
        Consider the paths $\tilde c(t)=L_{\tilde F(p)^{-1}}\circ \tilde F(p* (th))$
          and $c(t)=L_{(F\circ \tilde F(p))^{-1}}\circ F \circ L_{\tilde F(p)}(\tilde c(t))$,    and set $\tilde c_\alpha(t)=\pi_\alpha (\tilde c(t))$,
              $c_\alpha(t)=\pi_\alpha  ( c(t))$.   
        Since $D_\alpha \tilde F(p)$   and  $D_\alpha( F\circ \tilde F)(p) $ exist, we have $\tilde c'_\alpha(0)=D_\alpha \tilde F(p)(h)$
          and $ c'_\alpha(0)=D_\alpha(F\circ \tilde F)(p)(h) $.

          Write $\tilde c(t)=h(t)+w(t)$ with $h(t)\in H$ and $w(t)\in \mathfrak w$.   
          Then $$D_\alpha \tilde F(p)(h)=\tilde c'_\alpha(0)=\lim_{t\ra 0}\frac{\pi_\alpha h(t)}{t} +
          \lim_{t\ra 0}\frac{\pi_\alpha w(t)}{t}.$$   
            Set $h_1=\lim_{t\ra 0}\frac{\pi_\alpha h(t)}{t}\in H_1 $  and $w_1=\lim_{t\ra 0}\frac{\pi_\alpha w(t)}{t} \in W_\alpha$.  Then 
            $D_\alpha \tilde F(p)(h)=h_1+w_1.$

           We also write  $\tilde F(p)=\tilde h_0*\tilde w_0$ with
           $\tilde h_0\in H$ and $\tilde w_0\in \mathfrak w$.  
            We have
            $$D_\alpha F(\tilde F(p))(D_\alpha \tilde F(p)(h))=
            D_\alpha F(\tilde F(p))(h_1+w_1)=Bh_1+A w_1+ A Ds_\alpha(\overline{\tilde{h}_0})(\overline{h_1}).$$
        By (\ref{pre-derivative}),  
          $$c_\alpha(t)=\pi_\alpha\circ L_{(F\circ \tilde F(p))^{-1}}\circ F \circ L_{\tilde F(p)}(\tilde c(t))
          =B\pi_\alpha h(t)+A\pi_\alpha w(t) +A  s_\alpha (\overline{\tilde{h}_0}*\overline{h(t)})- As_\alpha(\overline{\tilde{h}_0}).$$
          Hence 
          \begin{align*}  D_\alpha(F\circ \tilde F)(p)(h)=
          c'_\alpha(0)&=B\lim_{t\ra 0}\frac{\pi_\alpha h(t)}{t}+A \lim_{t\ra 0}\frac{\pi_\alpha w(t)}{t}+
          A\lim_{t\ra 0}\frac{s_\alpha (\overline{\tilde{h}_0}*\overline{h(t)})-A s_\alpha(\overline{\tilde{h}_0})}{t}\\
          &=   Bh_1+A w_1+A\lim_{t\ra 0}\frac{s_\alpha (\overline{\tilde{h}_0}*\overline{h(t)})- s_\alpha(\overline{\tilde{h}_0})}{t}.
          \end{align*}
          It now suffices to  show  
          $$\lim_{t\ra 0}\frac{s_\alpha (\overline{\tilde{h}_0}*\overline{h(t)})- s_\alpha(\overline{\tilde{h}_0})}{t}=Ds_\alpha(\overline{\tilde{h}_0})(\overline{h_1}).$$

      Since $\pi\circ \tilde F(x)=\overline{\tilde{F}(0)}*\bar{\tilde{B}}(\bar x)$, we get
       $$\overline{h(t)}=\pi(\tilde c(t))=((\pi\circ \tilde F)(p))^{-1}*(\pi\circ \tilde F)(p*th)=\bar{\tilde {B}}(\bar p)^{-1}* 
       \bar{\tilde{B}}(\bar p*t\bar h)=\bar{\tilde{B}}(t\bar h)=t\bar{\tilde{B}}(\bar h).$$
             Hence 
             $$\lim_{t\ra 0}\frac{s_\alpha (\overline{\tilde{h}_0}*\overline{h(t)})- s_\alpha(\overline{\tilde{h}_0})}{t}=
             \lim_{t\ra 0}\frac{s_\alpha (\overline{\tilde{h}_0}*t\bar{\tilde{B}}(\bar h))- s_\alpha(\overline{\tilde{h}_0})}{t}=
              Ds_\alpha(\overline{\tilde{h}_0})(\bar{\tilde {B}}(\bar h)).$$
               Finally we notice that 
          $D_\alpha \tilde F(p)(h+w)=\tilde Bh+\tilde Aw+\tilde A D\tilde s_\alpha(\bar p)(\bar h)$
        implies $\bar h_1=\pi(D_\alpha \tilde F(p)(h))=\bar{\tilde {B}}\bar h$.

      \end{proof}

        Now we are ready to show that if  $\Gamma$ is a fiber similarity 
        group of $N$,     
         then 
        the differentials   $\{D_\alpha \gamma(p)|p\in N, \gamma\in \Gamma\}$ 
        are ``uniformly quasiconformal''. This result is needed in order to run Tukia's argument for the existence of invariant conformal structure.

      \begin{lemma}\label{uniform-alpha}
      Let $\Gamma$ be a  fiber similarity 
        group of $N$.  
       Then there is a constant $C\ge 1$ such that for every $\gamma\in \Gamma$, 
        the differential $D_\alpha \gamma(p)$ is $C$-quasiconformal for a.e. $p\in N$.

      \end{lemma}
      
      \begin{proof}
      By the discussion in Section \ref{eliminate}, there is  a  constant $K_0\ge 1$ with the following property: for each $\gamma\in \Gamma$,   there is some $t_\gamma\in \mathbb R$ such that 
            $\gamma':=e^{-t_\gamma D}\circ \gamma$ is $K_0$-biLipschitz,  and 
            $\gamma'$ acts on cosets of $W$  by isometric graded  automorphism and induces an isometry of $N/W$.   
      If $\gamma(h*w)=\gamma(0) * B_\gamma h*A_\gamma w*A_\gamma s_\gamma(\bar h)$ is a compatible expression for $\gamma$, then  $\gamma'$
       has a compatible expression  given by 
         $\gamma'(h*w)=\gamma'(0)* B'_\gamma h*A'_\gamma w*A'_\gamma s_\alpha(\bar h)$  
           where  
        $A'_\gamma= e^{-t_\gamma D}\circ A_\gamma$ is an isometry of $(\mathfrak w, d_{CC})$ and  $B'_\gamma= e^{-t_\gamma D}\circ B_\gamma$  is such that $\bar B'_\gamma$ is an isometry of $(\mathfrak n/\mathfrak w, d_{CC})$.   
       By the formula for 
        $D_\alpha F(p)$ we have that $D_\alpha (\gamma')(p)$ is the composition of 
        $D_\alpha (\gamma)(p)$  with  a standard Euclidean dilation:  
         $D_\alpha (\gamma')(p)(h+w)=e^{-t_\gamma \alpha} D_\alpha \gamma (p)(h+w)$ for $h\in H_1, w\in W_\alpha$.     Hence it suffices to show that 
       there is a constant $C\ge 1$ such that for every $\gamma\in \Gamma$, 
        the differential $D_\alpha \gamma'(p)$ is $C$-quasiconformal for a.e. $p\in N$.

         We fix an inner product on $V_\alpha$ so that $H_1$ and $W_\alpha$ are perpendicular to each other.  
        Notice that the equality $D_\alpha \gamma'(p)(h+w)=A'_\gamma w+ D\psi(\gamma')(\bar p)(\bar h)$ holds, where $\psi(F)$ was defined in Lemma \ref{psi(F)}.  As $A'_\gamma$ is an isometry and $\psi(\gamma')$ is Lipschitz  with Lipschitz constant bounded above by a constant depending only on $H$ and the biLipschitz constant of $\gamma'$, we see that there is a constant $C_1>0$   such that for each $\gamma \in \Gamma$,  the norm of the linear map $D_\alpha \gamma'(p): V_\alpha\ra V_\alpha$ is bounded above by $C_1$  for a.e. $p\in N$.  
         In particular it also holds for $\gamma^{-1}$.  By   applying Lemma \ref{chainrule} 
           to the composition  $\gamma\circ \gamma^{-1}=\text{Id}$ we  conclude that
               $D_\alpha \gamma'(p)$ is $C_1$-biLipschitz and so is $C_1^2$
               quasiconformal.


      \end{proof}

       \subsection{Measurable conformal structure  in the  $V_\alpha$   direction}\label{mconformal}

        In this subsection we shall show that, after a  conjugation, $s_{\gamma, \alpha}: \mathfrak n/\mathfrak w\ra Z_\alpha(\mathfrak w)$  is a  Lie group homomorphism for every  $\gamma\in \Gamma$.   See Proposition \ref{tukiaV}.  
        We shall  modify the proof of Tukia's theorem  \cite{T86},    see  also \cite{D10} for the foliated version.

       Fix an inner product on $V_\alpha$ and an orthonormal basis of $V_\alpha$ with respect to this inner product. Denote $n_\alpha=\dim(V_\alpha)$.  Then we can identify a linear transformation of $V_\alpha$ with an $n_\alpha\times n_\alpha$ matrix.  
        Denote by $SL(V_\alpha)$ (the special linear group)  the group of linear transformations of $V_\alpha$ whose  matrices have determinant equal to $1$, and  $SO(V_\alpha)\subset SL(V_\alpha)$ the subgroup consisting of linear transformations  that preserve  the inner product.  
       Let $X=SL(V_\alpha)/SO(V_\alpha)$.  Recall that,  $X$ 
          is a symmetric space of non-compact type  (see table V on page 518 of \cite{Hel78})   and so has nonpositive sectional curvature.   We denote by $\rho$ the distance on $X$.

          A   measurable conformal structure on $N$ in the  $V_\alpha$   direction is an 
           essentially bounded    
  measurable map 
 $$\mu: U \ra     X$$
  defined on a full measure subset $U\subset N$.   This is just a measurable way of assigning inner products (up to a scalar multiple) 
     in the direction of  $V_\alpha$.  
       To simplify language, we  will drop ``in the $V_\alpha$ direction'' and will just say 
  ``measurable conformal structure''.

        Let $\mu$ be a  measurable conformal structure  on $ N$ 
           and     $F: N\ra N$  a fiber similarity  map. 
            The pull-back
  measurable 
  conformal structure 
  $F^*\mu$ is defined  by:
   $$(F^*\mu) (p)=(D_\alpha F(p))[\mu(F(p))]:=(\text{det} D_\alpha F(p))^{-\frac{2}{\text{dim} V_\alpha}}(D_\alpha F(p))^T \mu(F(p)) D_\alpha F(p) , \;\;\;\text{for a.e.}\; p\in N.$$
      This is   analogous  to the  pull-back  of a Riemannian metric under a  diffeomorphism.
      Here we are using the fact that   $D_\alpha F(p)$ exists a.e.,  see Section \ref{differential}.

       \begin{corollary}\label{compositionrule}
      $(\gamma_2\gamma_1)^*\mu=\gamma_1^*(\gamma_2^*\mu)$ holds for all $\gamma_1, \gamma_2\in \Gamma$.
      \end{corollary}
      
      \begin{proof}
      It  follows  immediately from the  chain rule  $D_\alpha(\gamma_2\gamma_1)(p)= D_\alpha \gamma_2 (\gamma_1(p)) \circ D_\alpha\gamma_1 (p)$.

      \end{proof}

      A   fiber similarity  map $F$  is called conformal with respect to the measurable  conformal structure $\mu$
            if $F^*\mu=\mu$.     Tukia's   argument  together with  Corollary \ref{compositionrule}   and Lemma \ref{uniform-alpha}  
              then yield  that   $\Gamma$ has  an invariant measurable conformal structure; that is, 
       there is a      measurable conformal structure $\mu$  on $ N$  
       such that every $\gamma\in \Gamma$ is conformal with respect to $\mu$. 
        We may  assume $\Gamma$ is countable: 
        let $\Gamma_0$ be  a countable subgroup of $\Gamma$ that is dense in $\Gamma$ in   the topology of uniform convergence on compact subsets  of $N$;    if  $\Gamma_0$ can be conjugated into the similarity group of $(N, d)$ for some $D$-homogeneous distance $d$, then  the same map conjugates $\Gamma$ into the group of similarities of 
         $(N, d)$    as the limits of similarities are similarities.

       We next recall the notion of radial limit points. Let $S=N\rtimes_D \mathbb R$  be the  Heintze group associated with $(N, D)$  and  $\mathcal H: S\ra \mathbb R$  the height function given by $\mathcal H(x, t)=t$.  
        Let $\mathcal P(N)=\{(\xi_1, \xi_2)\in N\times N| \xi_1\not=\xi_2\}$, where we view $N=\partial S\backslash\{\infty\}$.  
          Let     $\chi: \mathcal P(N)\ra S$ be the map   that assigns to each   pair
     $(\xi_1, \xi_2)\in \mathcal P(N)$    the highest point on the geodesic $\xi_1\xi_2$; that is,  
      $\mathcal H(\chi(\xi_1, \xi_2))=\max \{\mathcal H(p)|p\in \xi_1\xi_2\}$.    
         In a sense, $\chi(P)$ is  a center of the triple $(\infty, \xi_1, \xi_2)$: it is where the two  geodesics  that join $\infty$ to  $\xi_1$ and  $\xi_2$ respectively  diverge from each other.   
          We observe that for any compact $C\subset S$, the set $\chi^{-1}(C)$ is compact in 
            $\mathcal P(N)$.

        The group     $\Gamma$  acts diagonally on  $\mathcal P(N)$:  
         $g(\xi_1, \xi_2)=(g(\xi_1), g(\xi_2))$.      
         
    \begin{definition}  \label{radial:defn}
           A point $\xi\in N$ is said to be  a  radial limit point of
            $\Gamma$ if   there  exists  
            a sequence of elements $\{h_i\}_{i=1}^\infty$ of $\Gamma  $  with the following property:
                           for any  pair  $P=(\xi_1, \xi_2)\in 
             \mathcal P(N)$, 
             and       any   complete  geodesic $\sigma$    asymptotic to $\xi$,        
              there exists a constant $C>0$          with           $\chi(h_i(P))\ra \xi$  and   $d(\chi(h_i(P)), \sigma)\le C$.  
              \end{definition}

        \begin{proposition}\label{tukiaV}
       There exists a biLipschitz map $F$ of $N$ such that each element of $F\Gamma F^{-1}$ has a
        compatible expression 
       $h*w\mapsto a*Bh*Aw*As(\bar h)$ such that $s_\alpha: \mathfrak n/\mathfrak w\ra Z_\alpha(\mathfrak w)$ is a  Lie group homomorphism.
      
      \end{proposition}

      \begin{proof}
        We equip $S=N\rtimes_D \R$ with a left invariant Riemannian metric such that $N$ and $\mathbb R$ are perpendicular to each other.
        The left translations $L_{(0,t)}$ are isometries of $S$ and 
        translate the vertical geodesic 
         $\sigma (s)=(0, s)$  above  $0\in N$  and 
          the boundary homeomorphisms induced  by them are the  automorphisms  $e^{tD}$    of $N$ generated by the derivation $D$.  
          
          Let $\mu$ be 
          a  $\Gamma$-invariant measurable conformal structure on $N$. 
       As $\mu$ is measurable,    it is approximately continuous  a.e. in $N$, see Theorem 2.9.13 in  \cite{F69}.   Let $p\in N$ be     a radial limit point of $\Gamma$  and also a point at which $\mu$ is 
       approximately continuous.  
        After applying  a  left translation 
         we may assume $p=0$ is the origin of $N$.  
                Fix a pair  $P\in \mathcal P(N)$  and 
               let $\sigma$ be the vertical geodesic (in $S$) above $0$.  
              Since $0$ is a radial limit point of $\Gamma$, there exists 
            a sequence of elements $\{\gamma_i\}_{i=1}^\infty$ of $\Gamma$    and 
              a constant $C>0$          with           $\chi(\gamma_i(P))\ra 0$  and   $d(\chi(\gamma_i(P)), \sigma)\le C$.  
        Fix a point $x_0\in \sigma$.     
              For each $i$ there is some $t'_i\in \mathbb R$ with $t'_i\ra +\infty$ as $i\ra \infty$ such that 
               $d(L_{(0,t'_i)}(\chi(\gamma_i(P))), x_0)\le C$.
                Since $L_{(0,t'_i)}$ is an isometry of $S$, we have $L_{(0,t'_i)}\circ \chi=\chi\circ  e^{t'_iD}$.  
                  Hence $d(\chi\circ e^{t'_iD} \circ \gamma_i(P), x_0)\le C$  and  so the set $\{e^{t'_iD}\circ \gamma_i(P)\}_{i=1}^\infty$ lies in the compact subset
                   $\chi^{-1}\bar B(x_0, C)$.    It follows that $\{e^{t'_iD}\circ \gamma_i\}_{i=1}^\infty$ is a compact family of 
                     biLipschitz  maps of $N$.           Recall that  there is a Carnot metric 
                      $d_{CC}$  on $W$  with the property that 
                   for    each $\gamma\in \Gamma $ 
                    there is some $t_\gamma\in \mathbb R$ such that 
                    $\gamma':=e^{-t_\gamma D}\circ \gamma$ 
                    has a compatible  expression   $\gamma'(h*w)={\gamma'(0)}*B_{\gamma} h*A_{\gamma}w *
               A_{\gamma}  s_{\gamma}(\bar h)$, where   $A_{\gamma}$ is a graded isomorphism of $W$ that is also an isometry of $(W, d_{CC})$.  
                            The compactness of the family   $\{e^{t'_iD}\circ \gamma_i\}_{i=1}^\infty$  implies that $\{t'_i+t_{\gamma_i}\}$ is a bounded sequence and so the family  $\{e^{-t_{\gamma_i}D}\circ \gamma_i\}_{i=1}^\infty$ is also compact.  
      Set   $f_i=e^{-t_{\gamma_i}D}\circ \gamma_i$.   
        By passing to  a  subsequence, we may assume $f_i$ converges uniformly on compact subsets  to a biLipschitz map
        $f:  N\ra  N$.   Since each  $f_i$ is   a fiber similarity, so is $f$. 

     Let  $\gamma\in \Gamma$ and denote $\tilde \gamma=f\gamma f^{-1}$, $\tilde \gamma_i=f_i\gamma f_i^{-1}$.   Let  $\mu_i=(f_i^{-1})^*\mu$. 
                    Since $\mu$ is $\Gamma$-invariant,      $\tilde  \gamma_i$ is conformal with respect to $\mu_i$: 
                       $$\tilde \gamma_i^*\mu_i=(f_i^{-1})^*\gamma^*f_i^* (f_i^{-1})^*\mu=(f_i^{-1})^*\gamma^*\mu=(f_i^{-1})^*\mu=\mu_i.$$   
                         Note  $\mu_i=(f_i^{-1})^*\mu=({e^{t_{\gamma_i}D}})^* (\gamma_i^{-1})^*\mu
                         =(e^{t_{\gamma_i}D})^*\mu$. 
                          So for $q\in N$,   
                          $$\mu_i(q)=(e^{t_{\gamma_i}D})^*\mu (q)=D_\alpha e^{t_{\gamma_i}D}(q)[\mu(e^{t_{\gamma_i}D}(q))]
                          =\mu(e^{t_{\gamma_i}D}(q))$$
            since  $D_\alpha e^{t_{\gamma_i}D}: V_\alpha\ra V_\alpha$ is the  standard dilation by $e^{\alpha t_{\gamma_i}}$ and so is conformal.  

       Let $U\subset  N$ be  a  bounded open subset  containing $0$. 
             There is a bounded open subset $U_0$ such that $U\bigcup \cup_i \tilde{\gamma}_i(U)\subset U_0$.  
              Since $\mu$ is approximately continuous at $0$  and $t_{\gamma_i}\ra -\infty$, the equality $\mu_i(q)=\mu(e^{t_{\gamma_i}D}(q))$
               implies that    for any $\epsilon>0$    there are subsets $C_i\subset U_0$ with $|C_i|\ra 0$ as $i\ra \infty$
                  and $\rho(\mu_i(x), \mu(0))\le \epsilon$ for $x\in U_0\backslash C_i$.   
                  Here $|E|$ denotes the measure of a subset $E\subset N$ and  $\rho$ is the distance on the symmetric space $X$.  
                  
                   The maps $\tilde \gamma^{-1}_i$ and $\tilde \gamma^{-1}$ form a compact family of biLipschitz maps.   There is some $L \ge 1$ such that  these maps are all $L$-biLipschitz.  
                  Hence  $|\tilde \gamma_i^{-1}(C_i)|\ra 0$ as $i\ra \infty$. 
                            Set $D_i=C_i\cup \tilde \gamma_i^{-1}(C_i)$. 
                      Now we have $|D_i|\ra 0$ as $i\ra \infty$ and 
                       $\rho(\mu_i(x), \mu(0))\le \epsilon$ and $\rho(\mu_i(\tilde \gamma_i(x)), \mu(0))\le \epsilon$ for 
                        $x\in U\backslash D_i$.                    
     Since  $\tilde \gamma_i$ is $\mu_i$-conformal, we have 
      $\mu_i(x)=D_\alpha\tilde \gamma_i(x) [\mu_i(\tilde \gamma_i(x))]$ for a.e. $x$. 
      Now
   $$    \rho(\mu_i(x), D_\alpha\tilde \gamma_i(x)[\mu(0)])
   =    \rho(\mu_i(\tilde \gamma_i(x)), \mu(0))\le \epsilon    $$
    for a.e. $x\in U\backslash D_i$.         Combining this with  $\rho(\mu_i(x), \mu(0))\le \epsilon$ , we   get 
        \begin{equation}
        \rho(\mu(0),  D_\alpha\tilde \gamma_i(x)[\mu(0)])\le 2\epsilon
        \end{equation}   
     for        a.e.  $x\in U\backslash D_i$.


      We consider compatible expressions   $\tilde \gamma(h*w)=\tilde \gamma(0)*Bh*Aw*A  s(\bar h)$ and 
        $\tilde \gamma_i(h*w)=\tilde \gamma_i(0)*B_ih*A_i w*A_i  s_i(\bar h)$.   
         We know that $\tilde \gamma_i$ converges to $\tilde \gamma$ uniformly on compact subsets.   It follows that $\tilde \gamma_i(0)\ra  \tilde \gamma(0)$ and $A_i\ra A$.  In general 
          we can not conclude that $s_i\ra s$  and  $B_i\ra B$  due to the fact that 
             compatible expressions are not unique.   Define functions $g_i, g: \mathfrak n/\mathfrak w \ra V_\alpha$  by $g_i(\bar h)=B_i h_1+A_i s_{i, \alpha}(\bar h)$ and 
              $g(\bar h)=B h_1+A  s_{\alpha}(\bar h)$, where $h_1$ is the $H_1$ component of $h\in H$.    By considering the $V_\alpha$ component of $\tilde \gamma_i(h)$ and $\tilde \gamma(h)$ we see that $g_i\ra g$ 
        uniformly on compact subsets. 

        For the rest of the proof, when we talk about ``perpendicular'', ``orthonormal'', length $|X|$ of a vector $X\in V_\alpha$,
         inner product $\left<\cdot , \cdot \right>$,  they are all with respect to  $\mu(0)$.   
     Let $X_1, \cdots, X_k$  ($k=\dim(W_\alpha)$)  be an orthonormal  basis of
       $W_\alpha$.   
      For $1\le l \le k$, let $g_{l}(x)=\left<g(x), X_l\right>$, 
      $g_{i,l}(x)=\left<g_i(x), X_l\right>$.  
     Since $\tilde \gamma$ and $\tilde \gamma_i$ are $L$-biLipschitz,    
      Lemma \ref{psi(F)}  implies that 
     there is  some constant $L_1>0$ such that  $g$ and $g_i$ are $L_1$-Lipschitz for all $i\ge 1$.
       It follows that 
       $g_{l}$  and  $g_{i,l}$ are $L_1$-Lipschitz functions on $N/W$
          for all $i\ge 1$, $1\le l\le k$.

          Let $\phi: H_1\ra W_\alpha$   be the linear map such that  
           $\{h+\phi(h)| h\in H_1\}$    is perpendicular to $W_\alpha$.
             Denote by $\psi: \bar{\mathfrak n}\to H_1$ the linear map given by 
              $\psi(\bar h)=(\pi|_{H_1})^{-1}(\bar{\pi}_1(\bar h))$, where    $\bar{\pi}_1: {\bar{\mathfrak n}}\ra \bar{V}_1$ is  the projection with respect to the decomposition
               ${\bar{\mathfrak n}}=\oplus_j \bar{V}_j$  and $\pi:\mathfrak n \to 
               \bar{\mathfrak n}$  is the quotient map.   For $1\le l\le k$,  
          let $L_l: {\bar{\mathfrak n}}\ra \mathbb R$ be the linear map  defined  by 
           $$L_l(\bar h)=
             -\left<X_l,  A\phi(\psi(\bar h))\right>.$$
          
        {\bf{Claim}}:  For each $1\le l\le k$,  $Dg_{i,l}$  converges to $L_l$ in $L_{\text{loc}}^1(N/W)$  as $i\to \infty$.

          We first assume the claim and finish the proof of the Proposition. 
               The claim implies that   for any $h\in H_1$ and any smooth function  with compact support  $\varphi$  defined on $N/W$, 
    \begin{align*}
      \int_{N/W}\varphi(x) L_l(\bar h)dx  & \longleftarrow     \int_{N/W} \varphi(x) Dg_{i,l}(x)(\bar h)dx\\
      =&-\int_{N/W} g_{i,l}(x) D_{\bar h} \varphi(x) dx \longrightarrow -\int_{N/W} g_l(x) D_{\bar h}\varphi(x)dx=\int_{N/W} \varphi(x) D g_l(x) (\bar h)dx ,
      \end{align*}
     which yields   $\int_{N/W} \varphi(x)(Dg_{l}(x)(\bar h)-L_l(\bar h))dx=0$.  
       Here the first convergence follows from the claim;   the second one follows from the fact that $g_{i,l}$ converges to $g_l$ uniformly on  compact subsets;   and  the equalities follow from   integration  by parts.   
     It follows that $D g_{l}(x) =L_l$  for a.e. $x\in N/W$. 
     This implies  $\langle X_l, ADs_\alpha(x)(\bar h)+A \phi(\psi(\bar h))\rangle=-\langle X_l, Bh\rangle$ for a.e. 
      $x\in N/W$ and all $1\le l\le k$.  Hence $ADs_\alpha(x)(\bar h)+A \phi(\psi(\bar h))$ defines a  constant  Lie group homomorphism (independent of $x$)   from $N/W$ to $W_\alpha$.  As $A \phi(\psi(\bar h))$
       is also a group homomorphism, we see that the Pansu differential of  $s_\alpha$  
     is a constant  group homomorphism.  It now follows from $s_\alpha(0)=0$ that 
     $s_\alpha$  
     is a  Lie group homomorphism.


      We now prove  the claim.  
      Recall  $\rho(D_\alpha\tilde \gamma_i(x)[\mu(0)], \mu(0))\le 2\epsilon$ for $x\in U\backslash D_i$. This means for $x\in U\backslash D_i$ the matrix representation for 
      the linear map $D_\alpha\tilde \gamma_i(x): V_\alpha\ra V_\alpha$ with respect to an orthonormal basis of  $(V_\alpha, \mu(0))$ is a constant multiple of a matrix that is very close  to  an orthogonal matrix. Since 
      $D_\alpha\tilde \gamma_i(x)(W_\alpha)=W_\alpha$,     we see that if 
       $x\in U\backslash D_i$  then 
       $D_\alpha\tilde \gamma_i(x):  V_\alpha\ra V_\alpha$  sends vectors in $V_\alpha$ perpendicular to $W_\alpha$
         to vectors almost perpendicular to 
       $W_\alpha$.     On the other hand,  as  $D_\alpha\tilde \gamma_i(x)(w)=A_iw$ for $w\in W_\alpha$  and $\{A_i|i\ge 1\}$ has compact closure (as $A_i\ra A$), there is some $M>0$ such that the operator norm of  $D_\alpha\tilde \gamma_i(x)$ is bounded above by $M$ for all $i$ and  all $x\in U\backslash D_i$.  
         Fix any $h\in H_1$ with $|h|=1$.  
       By the definition of $\phi: H_1\ra W_\alpha$,   we have  $\left<W_\alpha, h+\phi(h)\right>=0$. 
       As  $D_\alpha\tilde \gamma_i(x)(h+\phi(h))=B_i h+A_i \phi(h) + A_i Ds_{i, \alpha}(\bar x)(\bar h)$, there is a constant $\delta>0$ depending only on $\epsilon$ with $\delta\ra 0$ as $\epsilon\ra 0$ such that 
           $$|\left<X_l,  B_i h+A_i \phi (h) + A_i Ds_{i, \alpha}(\bar x)(\bar h)\right>|<\delta$$
            holds for all $x\in U\backslash D_i$  and all $1\le l\le k$. 
              Note  $\psi(\bar h)=h$ (as $h\in H_1$).  
           Since   $A_i\to A$   and  
            $Dg_{i,l}(\bar x)(\bar h)=\left<B_i h+ A_i Ds_{i, \alpha}(\bar x)(\bar h), X_l\right>$,   we   see that  
            $|Dg_{i,l}(\bar x)(\bar h)-L_l(\bar h)|\le 2\delta$  for all $x\in U\backslash D_i$
         and all sufficiently large $i$.    This implies 
       $\int_{U\backslash D_i}   |Dg_{i,l}(\bar x)(\bar h)-L_l(\bar h)| dx\le 2\delta |U|$  for  sufficiently large $i$. 
       On the other hand, as $g_{i,l}$, $i\ge 1$, $1\le l\le k$,   are  $L_1$-Lipschitz, there is a constant $C>0$ such that 
        $|Dg_{i,l}(\bar x)(\bar h)-L_l(\bar h)|\le C$ for all $x\in N$. 
          From this  we get $\int_{D_i} |Dg_{i,l}(\bar x)(\bar h)-L_l(\bar h)| dx\le  C|D_i|$. 
       As $|D_i|\ra 0$, we conclude that 
      $${\overline{\lim}}_{i\ra \infty} \int_{U}   |Dg_{i,l}(\bar x)(\bar h)-L_l(\bar h)| dx\le 2\delta |U|.$$  
       As this holds for all $\epsilon>0$ we have 
       $\lim_{i\ra \infty} \int_{U}   |Dg_{i,l}(\bar x)(\bar h)-L_l(\bar h)| dx=0$. 
       Since this holds for all bounded open subset $U\subset N$ and all $h\in H_1$, we have 
        $Dg_{i,l}\ra L_l$ in $L^1_{\text{loc}}(N/W)$.

      \end{proof}
      
      \subsection{Completing the proof of 
      Theorem \ref{main-uniform} when $\dim(W)\ge 2$ and $\dim(N/W)\ge 2$}
      
      We assume the assumptions of Theorem \ref{main-uniform}   and   in addition  that
       $\dim(W)\ge 2$, and $\dim(N/W)\ge 2$.    
      We first use Section  \ref{eliminate}
          to  get rid of $s_{\gamma, j}$ for $j<\alpha$,   and if $\alpha$ is an integer then apply Proposition \ref{tukiaV}
           to conclude that $s_{\gamma, \alpha}$ is a Lie group homomorphism for all  $\gamma\in \Gamma$, after a possible  further biLipschitz conjugation.   
      We observe that the property $s_{\gamma, j}\equiv 0$ for $j<\alpha$ is preserved when we apply Proposition \ref{tukiaV}:   this is because the conjugating map $f$ in the proof   of Proposition \ref{tukiaV}  is the limit of a sequence $\{f_i\}$   and each $f_i$ is the composition of a group element $\gamma_i\in \Gamma$ with $e^{t_iD}$ for some $t_i$;   since $\gamma_i$ has the above property, so do  $f_i$ and the limit $f$; finally a calculation shows that  if  two biLipschitz maps  $f_1, f_2$ of $N$ have this property then so  does  the composition $f_1\circ f_2$.

      At this point every element in $\Gamma$ has a compatible expression
       $F(h*w)=F(0)*Bh*Aw* A s(\bar h)$, where $s:\mathfrak n/\mathfrak w\ra Z(\mathfrak w)$ has the  properties that
        $s_j=0$ for $1\le j< \alpha$ and if $\alpha$ is an integer then 
         $s_\alpha: \mathfrak n/\mathfrak w\ra Z_\alpha(\mathfrak w)$ is a Lie group homomorphism.  
         We next show that such a map $F$ is an affine map; that is, $L_{F(0)^{-1}}\circ F$ is a  Lie  group automorphism.

       \begin{Le} \label{FtildeF}
        Let  $F, \tilde F: \mathfrak n\ra \mathfrak n $ be   biLipschitz maps
        with compatible expressions
           $F(h*w)=Bh*Aw*A s(\bar h)$  and  $\tilde F(h*w)=Bh*Aw*A\tilde s(\bar h)$.
             If   $s$, $\tilde s$ satisfy
            $s_j=\tilde s_j$ for $1\le j\le \alpha$,    then $F=\tilde F$.

          \end{Le}
          
          \begin{proof}
          We observe that    
          $F(h*w*(s(\bar h))^{-1}*\tilde s(\bar h))=\tilde F(h*w)$.  It follows that 
           $(F^{-1}\circ \tilde F)(h*w)=h*w*( s(\bar h)^{-1}*\tilde s(\bar h))$ is a biLipschitz shear map
            with shear function ${\tilde{\tilde{s}}}$ given by   ${\tilde{\tilde{s}}}=\tilde s-s$.  The assumption implies 
            ${\tilde{\tilde{s}}}_j\equiv 0$ for all $1\le j\le \alpha$. 
           On the other   hand, by Proposition \ref{shear-bilip}, 
            if $\alpha$ is not an integer, then  ${\tilde{\tilde{s}}}_j\equiv 0$ for all $j>  \alpha$,  and if $\alpha$ is an integer, 
               then ${\tilde{\tilde{s}}}_{k\alpha +j}={\tilde{\tilde{s}}}_j^{(k)}\equiv 0$
            for each $k\ge 1$ and $1\le j\le \alpha$.
            It follows that  ${\tilde{\tilde{s}}}=0$  and $F=\tilde F$.

          \end{proof}

          \begin{Le}\label{fisauto}
       Let     $F: \mathfrak n\ra \mathfrak n $ be    a  biLipschitz map  
               with compatible expression
           $F(h*w)=Bh*Aw*A s(\bar h)$.  Suppose 
            $s_j =0$ for $1\le j< \alpha$   and if $\alpha$ is an integer $s_\alpha: \mathfrak n/\mathfrak w\ra Z_\alpha(\mathfrak w)$ is a Lie group homomorphism.    Then $F$ is a Lie group automorphism.

          \end{Le}
          
          \begin{proof}
          Notice that it suffices to show $F_p=F$  for any $p\in N$.    Let $p\in N$ and set $\tilde F=F_p$.  By Lemma \ref{same B},  
              $\tilde F$ has a compatible expression given by
             $\tilde F(h*w)=Bh*Aw*A\tilde s(\bar h)$ for some map $\tilde s: \mathfrak n/\mathfrak w\ra Z(\mathfrak w)$.   
             By Lemma \ref{FtildeF} it now suffices to show 
              $\tilde s_j=s_j$ for all $1\le j\le \alpha$. 
              
         Let $h\in H$.         As $\tilde F(h)=Bh*A\tilde s(\bar h)$, we have
            $\tilde s(\bar h)=A^{-1}((Bh)^{-1}*\tilde F(h))$.  
                  Write $p=h_0*w_0$  and  $h_0*h=h_1*w_1$  with 
                $h_0, h_1 \in H$, $w_0, w_1\in \mathfrak w$.  Notice   $w_1\in  \oplus_{j\ge 2\alpha}W_j$.    We have 
                  $p*h=h_0*w_0*h=h_0*h*(h^{-1}*w_0*h)=h_1*w_1*(h^{-1}*w_0*h)$.   
                  In the following calculations we use  the quotient homomorphism $P_\alpha:  \mathfrak n\ra \bar{\mathfrak n}_\alpha={\mathfrak n}/(\oplus_{\lambda_j>\alpha} V_{\lambda_j})$:  
                   \begin{align*}
                 &P_\alpha( (Bh)^{-1}* \tilde F(h))\\
                 &=P_\alpha((Bh)^{-1}*(F(h_0*w_0))^{-1}*F(h_1*w_1*(h^{-1}*w_0*h)))\\
                   &=P_\alpha( (Bh)^{-1}*As(\overline{h_0})^{-1}* Aw_0^{-1}*(Bh_0)^{-1}*Bh_1* Aw_1*A(h^{-1}*w_0*h)*A s(\overline{h_0}*\bar h))\\
                   &=P_\alpha((Bh)^{-1}*As(\overline{h_0})^{-1}* Aw_0^{-1}*(Bh_0)^{-1}*Bh_1* Aw_1*(Bh)^{-1}*Aw_0*(Bh)*A s(\overline{h_0}*\bar h))\\
                   &=P_\alpha( (Bh)^{-1}*As(\overline{h_0})^{-1}* Aw_0^{-1}*(Bh_0)^{-1}*Bh_1*Aw_0*A s(\overline{h_0}*\bar h))\\
         &=P_\alpha((Bh)^{-1}*(Bh_0)^{-1}*Bh_1*As(\overline{h_0})^{-1}*A s(\overline{h_0}*\bar h))\\     
                 &=P_\alpha(A(s(\overline{h_0}*\bar h)-s(\overline{h_0}))).
                   \end{align*}
                     For the third equality we used   (\ref{cc}).   
                    In the  4th equality we used $P_\alpha(Aw_1)=0$ and $P_\alpha(Bh)\in Z(\bar{\mathfrak n}_\alpha)$. 
                     For the 5th equality we used  $P_\alpha(Bh_0), P_\alpha(Bh_1)\in Z(\bar{\mathfrak n}_\alpha)$.   For the last equality we used  $P_\alpha((Bh_0)^{-1}*Bh_1)=P_\alpha(Bh)$.
                       It follows that 
                    $\tilde s_j(\bar h)= 
                    s_j(\overline{h_0}*\bar h)-s_j(\overline{h_0})$  for $1\le j\le \alpha$.  
                      By the assumption on $s$, we have $\tilde s_j=0$ for $1\le j< \alpha$
                       and  if $\alpha$ is an integer  then 
                       $\tilde s_\alpha (\bar h)=s_\alpha(\overline{h_0}*\bar h)-s_\alpha(\overline{h_0})=s_\alpha(\bar h)$ as 
                       $s_\alpha: \mathfrak n/\mathfrak w\ra Z_\alpha(\mathfrak w)$ is a homomorphism.  Now we have $\tilde s_j=s_j$ for $1\le j\le \alpha$. By Lemma \ref{FtildeF}, $\tilde F=F$ and so $F$ is an    automorphism.

          \end{proof}

          In this last paragraph we switch back to Lie group notation. 
      At this point $\Gamma$ acts on $N$ by affine maps and is a uniform quasisimilarity group of $N$.  
       We write $\gamma=L_{\gamma(0)}\circ \phi_\gamma$, where $\phi_\gamma$ is the automorphism 
        $L_{\gamma(0)^{-1}}\circ \gamma$.   By Lemma \ref{bilip auto} we know that $\phi_\gamma$ is layer preserving,  that is, 
         $d\phi_\gamma(V_{\lambda_j})=V_{\lambda_j}$ for all $j$.  
       Each $\gamma\in \Gamma$ acts on the cosets of $W$ by an automorphism $A_\gamma$ of $W$ and 
        $A_\gamma$ is the composition of a  Carnot dilation and an isometric   graded isomorphism of $W$. Hence for each $\gamma$, there is a unique $t_\gamma\in \mathbb R$ such that $e^{-t_\gamma D}\circ \gamma$ acts on $W$ by an isometric graded isomorphism.   Since $\Gamma$ is a uniform group of quasisimilarities of $N$,  
           there is a constant $L>0$ such that $e^{-t_\gamma D}\circ \gamma$ is $L$-biLipschitz for all $\gamma\in \Gamma$.         It follows that    for each $j\ge 1$   the linear isomorphism   $d(e^{-t_\gamma D}\circ \phi_\gamma)|_{V_{\lambda_j}}: 
           V_{\lambda_j}\ra V_{\lambda_j} $ is  
      $L$-biLipschitz.  Now we  see that the map $\Gamma\ra GL(V_{\lambda_j})$ given by $\gamma \mapsto 
      d(e^{-t_\gamma D}\circ \phi_\gamma)|_{V_{\lambda_j}}$ is a homomorphism whose image has  compact  closure in 
       $GL(V_{\lambda_j})$.  It follows that there is an inner product   $\left<\cdot ,\cdot \right>_j$  on $V_{\lambda_j}$  such that  each  $d(e^{-t_\gamma D}\circ \phi_\gamma)|_{V_{\lambda_j}}$
        is an isometry with respect to this inner  product. 
        Let $\left<\cdot , \cdot \right>$ be the inner product on $\mathfrak n$ that agrees with 
         $\left<\cdot ,\cdot \right>_j$  on $V_{\lambda_j}$  such that $V_{\lambda_i}$ and $V_{\lambda_j}$ are perpendicular to each other for $i\not=j$.  Let $d_1$ be a   $D$-homogeneous distance on $N$ associated to this inner product. 
           Although  $d(e^{-t_\gamma D}\circ \phi_\gamma)$ is a linear isometry of  $(\mathfrak n, \langle,\rangle)$, it is not clear that 
            $e^{-t_\gamma D}\circ \phi_\gamma$       
            is an isometry of $(N,d_1)$.  However, 
             $\{e^{-t_\gamma D}\circ \phi_\gamma| \gamma\in \Gamma\}$ is a subgroup of   the group $\text{Auto}_g(N)$ of graded automorphisms with compact closure (we denote the closure by $K$).  
             Let $m$ be a normalized Haar measure on $K$.  Define a new distance $d_2$ on $N$ by  $d_2(x,y)=\int_K d_1(k(x), k(y))dm(k)$.   Now it is easy to check that $d_2$ is a 
      $K$-invariant         $D$-homogeneous distance on $N$ associated to   $\langle,\rangle$. It follows that 
     $\Gamma$ acts on $(N,d_2)$ by similarities.    
      Finally we    
      conjugate $\Gamma$ into 
       $\text{Sim}(N, d_0)$ where $d_0$ is a fixed  maximally symmetric $D$-homogeneous distance on $N$.

We have finished the proof of Theorem \ref{main-uniform} in the case when  $\dim(W)\ge 2$, $\dim(N/W)\ge 2$.


       \section{Case $\text{dim}(W)=1$}\label{dim(w)=1}

         In this section we prove Theorem \ref{main-uniform} in  the case when              
          $\text{dim}(W)=1$,   $\dim(N/W)\ge 2$  and $\Gamma$ is  amenable.   In this case  Tukia's arguments can not be used to 
         prove a   fiber Tukia theorem.     Nonetheless  Day's fixed point theorem once again can be used to 
          ``straighten''   the action   along the cosets of $W$.   
        We point out that the argument in this section is   valid for uniform quasisimilarity groups $\Gamma$  of product metric spaces of the form $\mathbb R\times Y$,   see   Theorem \ref{rtimesY}.  The only properties used in the proof  below  are that 
         $H$ is a proper metric space, 
         the action of     $\Gamma$  on $\mathbb R\times H$ permutes  the subsets  $\{\mathbb R\times \{h\}: h\in H\}$  and the induced action on $H$ is by similarities.

         Throughout this section 
         we assume  $\text{dim}(W)=1$    and $\dim(N/W)\ge 2$. 
      Let $\mathfrak h=\oplus_{\lambda>1}V_\lambda$.  
        Since $W=V_1$ is an ideal of $\mathfrak n$,   
           it follows from  the property $[V_a, V_b]\subset V_{a+b}$    that 
         $\mathfrak n=W\oplus \mathfrak h$ is a direct sum of  ideals where $W\simeq \R$.  So in this case, our group $N$ is a direct product of  a copy of $\mathbb R$ and a Carnot group which we also refer to as $W$ and $H$, where $H$ is the simply connected Lie group with Lie algebra $\mathfrak h$.  
         Because of this we can write $w*h \in N$ as $(w,h)\in W\oplus H$ and our group $\Gamma$ acts on $N$   by maps of the form
       $$\gamma(w,h)= ( \gamma^{\mathbb R}(w,h), \gamma^{H}(h))$$
       where $\gamma^H:H \to H$ is 
         biLipschitz  and for each $h\in H$,    the map 
         $\gamma^{\mathbb R}(\cdot, h): \R \to \R$ is   also biLipschitz.  
The proof we give here follows the proof from Section 3.3 in \cite{D10} but we construct the conjugating map in a slightly different manner in order to fix an oversight in the original paper; namely it is unclear that the conjugating map in \cite{D10} is actually biLipschitz. To construct our map we need to use that $\Gamma$ is   amenable.

         Since $\dim(N/W)\ge 2$,
         we   can first apply Tukia-type theorem for Carnot groups (Theorem \ref{tukia}) to the induced action of $\Gamma$  on $N/W \simeq H$.
            So there is a biLipschitz map $f_0$ of $H$ such that   after   conjugation by $f_0$,  
                 the induced action of $\Gamma$  on  $H$  is by similarities. Set     $F_0=(\text{Id}, f_0): W\oplus H\ra W\oplus H$. Then 
         we    can  
        conjugate      the action of $\Gamma$ on $N$   by $F_0$  to get  an action where $\gamma^{\mathbb R}(\cdot,h)$ is still biLipschitz and $\gamma^H$ is a similarity  of $(H, \bar d_{CC})$.
          Let   $t_\gamma\in \mathbb R$  be such that $e^{\alpha t_\gamma}$ is the similarity constant of $\gamma^H: (H, \bar d_{CC})\ra  (H, \bar d_{CC})$   
          and let $ \tilde \gamma: W\times H \ra W$ be the map given by 
           $\tilde \gamma(w, h)=e^{-t_\gamma} \gamma^{\mathbb R}(w,h)$.   
            Then $\gamma^{\mathbb R}(w,h)=e^{t_\gamma} \tilde \gamma(w, h)$.  
          Since $\Gamma$ is a uniform quasisimilarity group, there is a constant $\Lambda\ge 1$ such that $\tilde \gamma(\cdot,h)$ is $\Lambda$-biLipschitz for all $\gamma\in \Gamma$ and all $h\in H$.  
         After taking an index two subgroup if necessary, we may assume 
         $\tilde \gamma(\cdot,h)$ is orientation-preserving and so has derivatives in the interval 
         $[1/\Lambda, \Lambda]$.  For each $\gamma\in \Gamma$ we define a function
          $u_\gamma: N\rightarrow \mathbb R$ by 
          $u_\gamma(w,h)=\frac{\partial \tilde\gamma}{\partial w}(w,h)$ when   $\frac{\partial \tilde\gamma}{\partial w}(w,h)$  exists and $u_\gamma(w,h)=1$ otherwise.
            Here  $\frac{\partial \tilde\gamma}{\partial w}(w,h)$ denotes 
            the derivative of the function $\tilde \gamma(\cdot, h):  \mathbb R\ra \mathbb R$ at the point $w$ if it exists. 
               Then 
           $u_\gamma\in L^\infty(N)$   
            with values in 
       $[1/\Lambda, \Lambda]$.   In this section we are using the Hausdorff measure on $N$ including in the definition of $L^p$ spaces.  Of course, the Hausdorff measure on $N$  is a  Haar measure.  The point is that the argument in this section  still works when $N=\mathbb R\times H$ is replaced with a product  metric space $\mathbb R\times Y$.

       Let   $L^\infty(N)=(L^1(N))^*$  be equipped with weak$^*$ topology.
         Then $L^\infty(N)$
           is a locally convex topological vector space (see  Section 3.14, \cite{R91}).   
        We also consider the $L^\infty$  norm $||\cdot||$ on  $L^\infty(N)$.  We stress that the topology induced by the   $L^\infty$  norm  is different from the 
       weak$^*$     topology.    In the following,  when we say a subset $X\subset L^\infty(N)$ is closed (compact)  we mean it is closed (compact) in the  weak$^*$  topology;  similarly for closure of subsets  and continuity of maps; 
         when we say  $X$ is bounded we mean it is bounded with respect to the  $L^\infty$  norm.  
            By the Banach-Alaoglu theorem,     bounded closed subsets of $L^\infty(N)$ are   compact.

       Next we define an action of the opposite group $\Gamma^*$ of $\Gamma$ on $L^\infty(N)$.   For $\gamma \in \Gamma$ and $\phi\in L^\infty(N)$, define $\gamma\cdot \phi\in L^\infty(N)$ by  $\gamma\cdot \phi=u_\gamma (\phi\circ \gamma)$,   that is,  
         $$(\gamma\cdot \phi) (w,h)=u_\gamma(w,h)  \phi(\gamma(w,h))=u_\gamma(w,h) \phi(\gamma^{\mathbb R}(w,h),   \gamma^H(h)).$$
        One checks  that this defines a linear action (in particular an affine action) 
         of $\Gamma^*$ on $L^\infty(N)$.

       { \begin{lemma}
          The map  $\Gamma\times L^\infty(N)\rightarrow L^\infty(N), (\gamma, \phi)\mapsto \gamma\cdot \phi$,    is separately continuous.  
         
         \end{lemma}
         
         \begin{proof}
         It is easy to see that for fixed $\gamma\in \Gamma$, the map 
          $L^\infty(N)\rightarrow L^\infty(N),  \phi\mapsto \gamma\cdot \phi$,    is  continuous.   We next show that for fixed $\phi\in  L^\infty(N)$, the map 
           $\Gamma\rightarrow L^\infty(N),  \gamma\mapsto \gamma\cdot \phi$,    is   also   continuous.     Let $\phi\in L^\infty(N)$ be fixed and $\gamma_j, \gamma\in \Gamma$ be such that $\gamma_j\rightarrow \gamma$. 
          We need to show  
          $\gamma_j\cdot \phi \stackrel{w^*}\longrightarrow \gamma\cdot \phi$. 
         Let $C_c(N)$ be the space of compactly supported continuous functions  on $N$.
          Since $C_c(N)$ is a dense subspace of $L^1(N)$, it suffices to show that for any fixed $f\in C_c(N)$, 
            $\int_N f(\gamma_j\cdot \phi)dm \rightarrow \int_N f(\gamma\cdot \phi)dm$ as $j\rightarrow   \infty$,  where $m$ denotes the  Hausdorff    measure on $N$.            Notice that the Jacobian of $\gamma\in \Gamma$ at a point 
          $(w, h)$ is given by 
        $J\gamma(w,h)=e^{{t_\gamma}d_0}u_\gamma(w,h)$, where $d_0$ is the   Hausdorff dimension of $(N, d)$.
         The area formula applied to the map $\gamma$ and the function $f(\phi\circ\gamma)$   yields
          $$\int_N f(\phi\circ \gamma) J\gamma dm=\int_N (f\circ \gamma^{-1})\phi dm.$$
           It follows that
            $$\int_N f(\gamma\cdot \phi)dm=\int_N fu_\gamma(\phi\circ \gamma)dm=
           \frac{1}{e^{{t_\gamma}d_0}}   
           \int_N f(\phi\circ \gamma) J\gamma dm=
             \frac{1}{e^{{t_\gamma}d_0}}   
             \int_N (f\circ \gamma^{-1})\phi dm.$$
             Similarly we have 
             $$ \int_N f(\gamma_j\cdot \phi)dm= 
             \frac{1}{e^{{t_{\gamma_j}}d_0}}   
             \int_N (f\circ \gamma_j^{-1})\phi dm.$$
             Now
           \begin{align*}
           & \int_N f(\gamma_j\cdot \phi)dm-\int_N f(\gamma\cdot \phi)dm\\
           &=
             \frac{1}{e^{{t_{\gamma_j}}d_0}}   
            \int_N(f\circ \gamma_j^{-1}-f\circ \gamma^{-1})\phi dm+
         \big(\frac{1}{e^{{t_{\gamma_j}}d_0}}-\frac{1}{e^{{t_{\gamma}}d_0}}\big)
            \int_N(f\circ \gamma^{-1})\phi dm.    
            \end{align*}
       As $t_{\gamma_j}\ra t_\gamma$, the second term above clearly goes to $0$ as $j\ra \infty$.    Since   $f$ is continuous with compact support and $\gamma_j$ converges to $\gamma$ uniformly on compact subsets of  $N$, there is a compact subset $F\subset N$ and a quantity $\epsilon_j\ra 0$ such that  
         $\sup\{|f\circ \gamma_j^{-1}(n)-f\circ \gamma^{-1}(n)|:  n\in F\}<\epsilon_j$   and 
       $f\circ \gamma_j^{-1}(n)-f\circ \gamma^{-1}(n)=0$ for all $n\in N\backslash F$ and all sufficiently large $j$.      We have 
        $$|\int_N(f\circ \gamma_j^{-1}-f\circ \gamma^{-1})\phi dm|\le
         \int_F |f\circ \gamma_j^{-1}-f\circ \gamma^{-1}| |\phi| dm \le \epsilon_j \int_F |\phi| dm\ra 0$$ 
          and   hence 
           $\int_N f(\gamma_j\cdot \phi)dm\ra \int_N f(\gamma\cdot \phi)dm$.

         \end{proof}
         
         }


        Since  $u_\gamma$      takes values in the interval $[1/\Lambda, \Lambda]$, we see that $|| \phi||/\Lambda  \le ||\gamma\cdot \phi ||\le \Lambda ||\phi||$ for any $\phi\in L^\infty(N)$  and any $\gamma\in \Gamma$.   It follows that every $\Gamma$ orbit is a bounded subset of $L^\infty(N)$.  Let $\phi_0$ be the constant function $1$ on $N$ and $K$ be the closure of the convex hull of $\Gamma^*\cdot\phi_0$. Then $K$ is a compact, convex subset in $L^\infty(N)$.  Since $\Gamma$ is amenable,    
        Day's fixed point theorem  implies that $\Gamma$ has a fixed point $u$  in $K$.

       Notice that $\gamma\cdot \phi_0=u_\gamma$ and so is a  measurable function  with values in $[1/\Lambda, \Lambda]$.  It follows that every element in $K$, in particular $u$,  takes  values in $[1/\Lambda, \Lambda]$  a.e. By Fubini's theorem, for a.e. $h\in H$, 
          the map $\mathbb R\ra \mathbb R, w\mapsto u(w,h)$, is measurable and takes values in $[1/\Lambda, \Lambda]$  for a.e. $w\in W$.   
       
        Since $u$ is a fixed point of $\Gamma$,   for each $\gamma\in \Gamma$, we have 
 $u(w,h)= 
         u_\gamma(w,h) u(\gamma^{\mathbb R}(w,h),   \gamma^H(h))$  
         for  a.e. $(w,h)\in W\times H$. 
       By Fubini's theorem, for a.e. $h\in H$, 
          $u(w,h)= 
         u_\gamma(w,h) u(\gamma^{\mathbb R}(w,h),   \gamma^H(h))$  
            for a.e. $w\in W$.

     {

       Let  $\Gamma_0\subset \Gamma$ be  a countable dense  subgroup.    There is a  $\Gamma_0$-invariant full measure subset $U\subset H$  with the following properties:\newline
        (1)  for each $h\in U$ and each $\gamma\in \Gamma_0$,  the equality   
        $u(w,h)= 
         u_\gamma(w,h) u(\gamma^{\mathbb R}(w,h),   \gamma^H(h))$  
          holds  for a.e. $w\in W$. \newline
          (2) for   each  $h\in U$, 
          the map $\mathbb R\ra \mathbb R, w\mapsto u(w,h)$, is measurable and takes values in $[1/\Lambda, \Lambda]$  for a.e. $w\in W$.

       For each $h\in U$ we define a  function $v_h: \mathbb R\rightarrow \mathbb R$ by 
        $v_h(w)=\int_0^w u(s, h)ds$.   Clearly  $v_h$ is  $\Lambda$-biLipschitz. We also define $G_0: W\times U\rightarrow  W\times U$ by   $G_0(w,h)=(v_h(w), h)$.   We shall show that\newline
        (1)   $G_0$ is biLipschitz (Lemma \ref{G_0bilip})   
        and so admits a biLipschitz extension $\bar G_0:  N\ra N$; \newline
         (2) 
         $\bar G_0\Gamma \bar G_0^{-1}$  acts by similarities along the cosets of $W$ (Lemma \ref{simialongW}).

       \begin{lemma}\label{v_h}
       The function $v_h$ satisfies the following equation for all $\gamma\in \Gamma_0$ and all $w_1, w_2\in W$:
       $$
         v_{ \gamma^H(h)}(\gamma^{\mathbb R}(w_1,h ))-
         v_{\gamma^H(h)}(\gamma^{\mathbb R}(w_2, h))
         =e^{t_\gamma}(v_h(w_1)-v_h(w_2)).$$
       
       \end{lemma}
       
       \begin{proof}   For each $h\in U$,  $u_\gamma(w,h)=\frac{\partial \tilde \gamma}{\partial w}(w,h)$  
             for a.e. $w\in W$.   
        On the other hand,  
       by the choice of $U$,   for each $h\in U$ and each $\gamma\in \Gamma_0$,  the equality  
         $u(w,h)= 
         u_\gamma(w,h) u(\gamma^{\mathbb R}(w,h),   \gamma^H(h))$  
           holds  for a.e. $w\in W$.  
        Now we have
     %
        \begin{align*}
       v_{ \gamma^H(h)}(\gamma^{\mathbb R}(w_1,h ))-
        v_{\gamma^H(h)}(\gamma^{\mathbb R}(w_2, h))
         &=\int_{\gamma^{\mathbb R}(w_2,h)}^{\gamma^{\mathbb R}(w_1, h)} u(s,   \gamma^H(h))ds\\
         &=e^{t_\gamma} \int_{w_2}^{w_1} u(\gamma^{\mathbb R}(t, h),   
         \gamma^H(h)) \frac{\partial \tilde \gamma}{\partial t}(t,h)  dt\\
         &=e^{t_\gamma}\int_{w_2}^{w_1} u(t,h) dt
         =e^{t_\gamma}(v_h(w_1)-v_h(w_2)).
         \end{align*}   
       \end{proof}

       
         Let $C_b(\Gamma)$ be the space of bounded continuous functions on $\Gamma$.  
       Let $M_\Gamma$ be  the set of means on $C_b(\Gamma)$ of the form 
       $$\sum_{j=1}^n t_j\delta_{\gamma_j}\;\;\;\;\; (n\in \mathbb N, \gamma_1, \cdots, \gamma_n\in \Gamma, \; t_1,\cdots, t_n\ge 0,\; t_1+\cdots +t_n=1),$$
        where  for $\gamma\in \Gamma$, $\delta_\gamma: C_b(\Gamma)\ra \mathbb C$ is the linear functional given by $\delta_\gamma(f)=f(\gamma)$.    Then $M_\Gamma$ is $w^*$-dense in the set of  all means on $C_b(\Gamma)$, see  \cite{R02},  page 29.   Let  $\mathbb Q M_{\Gamma_0}$  
       be  the set of means on $C_b(\Gamma)$ of the form 
       $$\sum_{j=1}^n t_j\delta_{\gamma_j}\;\;\;\;\; (n\in \mathbb N, \gamma_1, \cdots, \gamma_n\in \Gamma_0, \; t_1,\cdots, t_n\ge 0,\;  t_j\in \mathbb Q,\; t_1+\cdots +t_n=1).$$  
       Since $\Gamma_0$ is dense in $\Gamma$ and $\mathbb Q$ is dense in $\mathbb R$, we see that   $\mathbb Q M_{\Gamma_0} $ is   also $w^*$-dense in the set of  all means on $C_b(\Gamma)$.  Let $K_0$ be the set of points in $K$ of the form
 $$       \sum_{j=1}^n t_j u_{\gamma_j}\;\;\;\;\; (n\in \mathbb N, \gamma_1, \cdots, \gamma_n\in \Gamma_0, \; t_1,\cdots, t_n\ge 0,\;  t_j\in \mathbb Q,\; t_1+\cdots +t_n=1).$$  
       
        \begin{lemma}\label{p-con}  There is a   sequence  $\{u_i\}$ in $K_0$ 
       that  converges to $u$ in the
        weak$^*$  topology.     

       \end{lemma}
       
       \begin{proof}
         This follows from the proof of Day's fixed point theorem, see page 29 of   \cite{R02}.
           The reader should have a copy of that proof in front of him/her while reading this proof. 
              In that proof, we pick $x_0=\phi_0$.   
               Let $A(K)$ be the set of all continuous affine functions on $K$.  For $\psi\in A(K)$, define 
              $$\phi_\psi: \Gamma \rightarrow \mathbb C, \;\; \gamma \mapsto \psi(\gamma\cdot \phi_0) =\psi(u_\gamma). $$

               Let $m$ be  a left invariant mean on   $C_b(\Gamma)$.   
           Since  $\mathbb Q M_{\Gamma_0} $ is    $w^*$-dense in the set of  all means on $C_b(\Gamma)$, there is  a   net   $\{m_\beta\}$ in 
          $\mathbb Q M_{\Gamma_0} $   that  converges to $m$ in the
        weak$^*$  topology.    
              For each $m_\beta=\sum_{j=1}^n t_j \delta_{\gamma_j}$, define   
            $u_\beta\in K_0$ by $u_\beta=\sum_{j=1}^n t_j  u_{\gamma_j}$.  
         One checks that  $<\phi_\psi, m_\beta>=\psi(u_\beta)$ for all $\psi\in A(K)$.   Since $K$ is compact,  $u_\beta$ sub-converges to some $u\in K$. It is proved on page 30 of \cite{R02}
            that this $u$ is a fixed point of $\Gamma$.   
              Finally, since $L^1(N)$ is separable,  by the sequential  Banach-Alaoglu theorem,  
                closed balls in $L^\infty(N)=(L^1(N))^*$ are metrizable.  So $K$ is compact and metrizable.   Therefore we can pick   a sequence  $\{u_i\}$ from the net  $\{u_\beta\}$ that   converges to $u$ in the
        weak$^*$  topology.

       \end{proof}

       }

        \begin{lemma}\label{v_h control}
       There  is a full measure subset $U'\subset U$  and    a  constant $C>0$ such that $|v_{h_1}(w)-v_{h_2}(w)|\le C \cdot 
       \bar d_{CC}(h_1, h_2)^{\frac{1}{\alpha}}$  for all $h_1, h_2\in U'$ and all $w\in W$.          
       \end{lemma}
       
       \begin{proof} 
           We consider the following $D$-homogeneous  distance  $d$ on $N$:
         $$d((w_1, h_1), (w_2, h_2))=|w_1-w_2|+\bar d_{CC}(h_1, h_2)^{\frac{1}{\alpha}}.$$   
       Recall that each $\gamma\in \Gamma$ has the expression
           $\gamma(w,h)=(e^{t_\gamma}\tilde \gamma(w,h), \gamma^H(h))$   and 
            $\gamma^H: (H, \bar d_{CC}) \ra (H, \bar d_{CC}) $ is a  similarity with similarity constant $e^{\alpha t_\gamma}$.   Since $\Gamma$ is a uniform quasisimilarity group, by increasing $\Lambda$ if necessary we may assume each $\gamma$ is a 
              $(\Lambda, e^{t_\gamma})$ quasisimilarity of $(N,d)$.   
           It follows that  $|\tilde \gamma(w, h_1)-\tilde \gamma(w, h_2)|\le \Lambda \bar d_{CC}(h_1, h_2)^{\frac{1}{\alpha}}$ for all $h_1, h_2\in H$, all $w\in W$ and all $\gamma\in \Gamma$.  
       
       For each integer $n\ge 1$, let $B_n\subset H$ be  the ball with radius $n$ and center the origin (the identity element of $H$). Set $D_n=[-n,n]\times B_n\subset W\times H$.    
       Notice that $u_i|_{D_n}, u|_{D_n}\in L^\infty (D_n)$  and 
       that $u_i|_{D_n}\ra  u|_{D_n}$ in the weak$^*$ topology. Since $D_n$ has finite measure we see that  $u_i|_{D_n}\ra  u|_{D_n}$   weakly in $L^p(D_n)$ for any $1\le p<\infty$.   By Mazur's lemma, there is a sequence  $\tilde u_j$ of convex combinations of the $u_i's$ such that $\tilde u_j|_{D_n}$ converges to $u|_{D_n}$ in the $L^p$ norm.  
         Then   after taking a subsequence if necessary we may assume that 
        $\tilde u_j|_{D_n}$ converges to $u|_{D_n}$ a.e.    By Fubini, there is a null set $F_n\subset H$ such that for any $h\in B_n\backslash F_n$,   $\tilde u_j(w,h)\ra u(w,h)$ for a.e. 
        $w\in [-n,n]$.    On the other hand,  as  each $u_i$ is a convex combination of 
     the     $u_\gamma$'s,  each $\tilde u_j$ is also a convex combination of the  $u_\gamma$'s. Write $\tilde u_j=\sum_{i=1}^{k_j} t_{j,i}u_{\gamma_{j,i}}$ for some $t_{j,i}\ge 0$ with $\sum_i t_{j,i}=1$  and $\gamma_{j,i}\in \Gamma$. 
         Since each $u_\gamma$ is bounded above by $\Lambda$,  so is $\tilde u_j$.  It follows from dominated convergence theorem that 
        $\int_0^w\tilde u_j(s,h) ds\ra \int_0^w u(s,h) ds=v_h(w)$ for any $w\in [-n,n]$, and  any  $h\in B_n\backslash F_n$.   
       
       Set $U'=U-\cup_n F_n$. Then $U'$ has full measure in $H$.    Let $h_1, h_2\in U'$ and $w\in W$. Pick a sufficiently large  $n$ such that $(w,h_1), (w, h_2)\in D_n$.    
        Now 
\bea 
|\nu_{h_2}(w)-\nu_{h_1}(w)| 
&=&    \lim_{j \to \infty}  \left|  \int_0^w  \left( \tilde u_{j}(s,h_2) -  \tilde u_{j}(s,h_1)\right)ds\right|\\
&=&     \lim_{j \to \infty}  \left|  \sum_{i}  t_{j,i}\int_0^w(u_{\gamma_{j,i}}(s,h_2) -    u_{\gamma_{j,i}}(s,h_1))ds\right|\\
&=&     \lim_{j \to \infty}\left|  \sum_{i}  t_{j,i}
 ({\tilde{\gamma}_{j,i}}(w, h_2)-{\tilde{\gamma}_{j,i}}(0, h_2)-
 {\tilde{\gamma}_{j,i}}(w, h_1)+{\tilde{\gamma}_{j,i}}(0, h_1) )   
  \right|\\
&\le&  2\Lambda \bar d_{CC}^{\frac{1}{\alpha}}(h_1, h_2).  
\eea

       \end{proof}

        \begin{lemma}\label{G_0bilip}
       The map $G_0|_{W\times U'}:   (W\times U', d)\ra  (W\times U', d)$ is biLipschitz.

       \end{lemma}
       
       \begin{proof}
       We first show $G_0|_{W\times U'}$   is Lipschitz.  
       Let $(w_1, h_1), (w_2, h_2)\in W\times U'$.  Then 
       \begin{align*}
       d(G_0(w_1, h_1), G_0(w_2, h_2))&
       =d((v_{h_1}(w_1), h_1),  (v_{h_2}(w_2), h_2))\\
       &= |v_{h_1}(w_1)-v_{h_2}(w_2)|+ \bar d_{CC}^{\frac{1}{\alpha}}(h_1, h_2)\\
       &\le |v_{h_1}(w_1)-v_{h_2}(w_1)|+ |v_{h_2}(w_2)-v_{h_2}(w_1)|+
       \bar d_{CC}^{\frac{1}{\alpha}}(h_1, h_2)\\
       &\le (2 \Lambda+1)\bar d_{CC}^{\frac{1}{\alpha}}(h_1, h_2)+\Lambda |w_1-w_2|\\   &\le (2 \Lambda+1) d((w_1, h_1), (w_2, h_2)).
       \end{align*}

       Next we show $G_0^{-1}|_{W\times U'}$ is also Lipschitz.  First assume $|w_1-w_2|\le   4 \Lambda^2  \bar d_{CC}^{\frac{1}{\alpha}}(h_1, h_2)$.   Then
    $$   d(G_0(w_1, h_1), G_0(w_2, h_2))
       \ge 
       \bar d_{CC}^{\frac{1}{\alpha}}(h_1, h_2)
       \ge \frac{1}{4\Lambda^2+1}  d((w_1, h_1), (w_2, h_2)).  $$
        Now we assume $|w_1-w_2|\ge   4 \Lambda^2  \bar d_{CC}^{\frac{1}{\alpha}}(h_1, h_2)$.  
         Then 
       \begin{align*}
       d(G_0(w_1, h_1), G_0(w_2, h_2))&
       = |v_{h_1}(w_1)-v_{h_2}(w_2)|+ \bar d_{CC}^{\frac{1}{\alpha}}(h_1, h_2)\\
       & \ge |v_{h_1}(w_1)-v_{h_1}(w_2)|-|v_{h_1}(w_2)-v_{h_2}(w_2)|+ \bar d_{CC}^{\frac{1}{\alpha}}(h_1, h_2)\\
       &\ge \frac{1}{\Lambda} |w_1-w_2|-2\Lambda \bar d_{CC}^{\frac{1}{\alpha}}(h_1, h_2)+\bar d_{CC}^{\frac{1}{\alpha}}(h_1, h_2)\\
       &\ge \frac{1}{2\Lambda} |w_1-w_2|+\bar d_{CC}^{\frac{1}{\alpha}}(h_1, h_2)\\
       &\ge \frac{1}{2\Lambda}d((w_1, h_1), (w_2, h_2)).
       \end{align*}

       \end{proof}

       Since $W\times U'$ is dense in $N$, Lemma \ref{G_0bilip}  implies 
       $G_0|_{W\times U'}$ extends to a biLipschitz map $\bar G_0: N\ra N$.

        \begin{lemma}\label{simialongW}
       For each  $\gamma\in \Gamma$,  $\bar G_0\circ \gamma\circ \bar G_0^{-1}$ acts by similarities  with  similarity constant $e^{t_\gamma}$  along the cosets of $W$.  
       
       \end{lemma}
       
       \begin{proof}
       Notice that $G_0^{-1}$ has the expression $G_0^{-1}(w,h)=(v_h^{-1}(w), h)$.  
        As above write $\gamma$ as  
         $\gamma(w,h)=(\gamma^{\mathbb R}(w, h),  \gamma^H(h))$.  It follows that 
         $G_0\circ \gamma\circ G_0^{-1}(w,h)=(v_{\gamma^H(h)}(\gamma^{\mathbb R}(v_h^{-1}(w), h)),  \gamma^H (h)). $
         Let $h\in U'\cap  (e^{t_\gamma\bar D}\circ \bar \gamma)^{-1}(U')$, $w_1, w_2\in W$.   First assume  
          $\gamma\in \Gamma_0$.  
         By Lemma \ref{v_h}  we have 
      %
        $$|v_{\gamma^H(h)}(\gamma^{\mathbb R}(v_h^{-1}(w_1), h))-v_{\gamma^H(h)}(\gamma^{\mathbb R}(v_h^{-1}(w_2), h))|=e^{t_\gamma}|v_h(v_h^{-1}(w_1))-v_h(v_h^{-1}(w_2))|
       =e^{t_\gamma}|w_1-w_2|. $$
         Since $U'\cap  (e^{t_\gamma\bar D}\circ \bar \gamma)^{-1}(U')$ has full measure
          in $H$,  we see that the map $\bar G_0\circ \gamma\circ \bar G_0^{-1}$ restricted to almost every coset of $W$ is a similarity  with  similarity constant $e^{t_\gamma}$. Since   $\bar G_0\circ \gamma\circ \bar G_0^{-1}$ is biLipschitz,  we see that 
         $\bar G_0\circ \gamma\circ \bar G_0^{-1}$ restricted to every coset of $W$ is a similarity  with  similarity constant $e^{t_\gamma}$.
   
     Now  consider a general $\gamma\in \Gamma$.  There is sequence $\gamma_j\in \Gamma_0$ that converges to $\gamma$ uniformly on compact subsets of $N$. 
        Since each   $\bar G_0\circ \gamma_j\circ \bar G_0^{-1}$ restricted to every coset of $W$ is a similarity  with  similarity constant $e^{t_{\gamma_j}}$,  a similar statement  is true for 
          $\bar G_0\circ \gamma\circ \bar G_0^{-1}$. 
              \end{proof}

       By Lemma \ref{simialongW},  
       after a  conjugation if necessary,    we may assume 
        $\tilde \gamma(\cdot,h)$ is a translation of $\mathbb R$  
        for each $\gamma\in \Gamma$  and $h\in H$.
        Now we can write $\gamma$ as
        $$\gamma(w,h)=\gamma(0,0)\cdot e^{t_\gamma D}(w+s_\gamma(h), Bh),$$
         where   $s_\gamma: H\rightarrow \mathbb R$   is  a  map  satisfying $s_\gamma(0)=0$,  and 
          $B:H\rightarrow H$ is a graded automorphism of $H$ that is also an isometry with respect to $\bar d_{CC}$.  
          Since $\{e^{-t_\gamma D}\circ \gamma|\gamma\in \Gamma\}$ is a family of uniform biLipschitz maps,  we see that 
            $\{s_\gamma|\gamma\in \Gamma\}$ is a bounded subset of 
           $$E_1=\{s: H\rightarrow \mathbb R \;\text{is}\; \textstyle{\frac{1}{\alpha}}\text{-Holder}, \; s(0)=0\}.$$ 
           Now one can apply Day's  theorem to a compact convex subset $K$ of $E_1$ to eliminate   $s_\gamma$ as in Section \ref{eliminate} (much easier now). After this, each $\gamma$ has the form $\gamma(w,h)=\gamma(0,0)\cdot e^{t_\gamma D}(w, Bh),$  which is a similarity of $(N,d)$.   

       We have finished the proof of Theorem \ref{main-uniform} in the case  when $\dim(W)=1$,    $\dim(N/W)\ge 2$.

       As already indicated above, the arguments in this section are valid for product metric spaces $\mathbb R \times Y$:
       \begin{theorem}\label{rtimesY}
       Let $Y$ be a proper metric space with a distance $d$. Let $\alpha\ge 1$ and $\tilde d$ be the distance on $\mathbb R\times Y$ defined by 
       $\tilde d((w_1, y_1), (w_2, y_2))=|w_1-w_2|+d(y_1, y_2)^{\frac{1}{\alpha}}$.
         Let 
       $\Gamma$ be   
    an  amenable      uniform quasisimilarity group of   $(\mathbb R\times Y, \tilde d)$.  Suppose the action of $\Gamma$ on  $\mathbb R\times Y$  permutes the subsets
 $\{\mathbb R\times \{y\}: y\in Y\}$ and induces a similarity action on $(Y, d)$.   
      Then 
       $\Gamma$ can be biLipschitzly conjugated into  the similarity group of 
       $(\mathbb R\times Y, \tilde d)$.
       \end{theorem}
       
       \begin{proof}
       The argument in this section shows  that, after a conjugation,   $\Gamma$   acts on the ``fibers'' $\mathbb R\times \{y\}$ by similarities.    At this point the elements of 
       $\Gamma$ may still contain shear components $s$.   One can get rid of the shear components by modifying the above arguments or apply Theorem 3.3 from \cite{DyX16}.  
  \end{proof}

        \section{Case $\text{dim}(N/W)=1$}\label{n/w}

  In this section we consider the case when              $\text{dim}(N/W)=1$.

  As before we shall make a biLipschitz conjugation so that $\Gamma$ is a fiber similarity group.  
           We first need  to show that (by replacing $\Gamma$ with a biLipschitz   conjugate if necessary) $\Gamma$   induces an action on the quotient by affine automorphisms. When $\dim(N/W)\ge 2$, we can apply a Tukia  type theorem to achieve this.  
            We need to consider the case  $\dim(N/W)=1$    separately since 
            we  can not apply a Tukia type theorem in   this case.  The theorem of Farb-Mosher \cite {FM99}  says that  a uniform quasisimilarity group of the real line is biLipschitz conjugate to a similarity group.   Although the action of $\Gamma$ on $N$ induces a uniform quasisimilarity action  of $\Gamma$ on the quotient $N/W=\mathbb R$,  it is not clear how to lift the biLipschitz conjugating map of $\mathbb R$  to a biLipschitz map of $N$.   
            

      We consider two cases:\newline
            Case I.  There is a one dimensional ideal  $H$ of   $\mathfrak n$ 
              such that     $H\subset V_\alpha$  and   $H\cap W_\alpha=\{0\}$. 
               In this case   $\mathfrak n=\mathfrak w\oplus H$ is a direct sum of two ideals     and $N=W\times H$ is a direct product.\newline
             Case II.     $\mathfrak n$ does not admit a direct sum decomposition as in Case I. 
             
               Case I is easy to deal with.  Since $\Gamma$ permutes the cosets of $W$, it induces an action on the quotient $N/W=H=\mathbb R$  and yields a subgroup $\overline\Gamma\subset \text{Homeo}(\mathbb R)$ which is a  uniform quasisimilarity  group.   The result of  Farb--Mosher
               \cite{FM99} 
              says that there is a biLipschitz map $f_0: \mathbb R\ra \mathbb R$  such that $f_0\overline \Gamma f_0^{-1}$ 
               consists of similarities.  Now let $F_0=\text{Id}\times f_0: N=W\times H\ra W\times H=N$. Then $F_0$ is a biLipschitz map of $N$ and the conjugate $F_0\Gamma F_0^{-1}$ has the property that its induced action on the quotient $N/W=H$ is by similarities.   If $\dim(W)\ge 2$,   then we can apply   Theorem   \ref{foliatedtheorem} 
                to conjugate $\Gamma$ into a fiber similarity group.  If $\dim(W)=1$, then Section \ref{dim(w)=1}  can now be used to finish the proof.
                 This  finishes Case I.

          We next consider Case II.       In this case we shall prove that the induced action of $\Gamma$ on the quotient  $N/W=\mathbb R$ is already by similarities   and so there is no need to conjugate.   So we assume  $\mathfrak n$ does not admit a direct sum decomposition as in Case I.     We notice that in this case $\dim(W)\ge 2$.   After applying 
            Theorem   \ref{foliatedtheorem} we may assume that the action of $\Gamma$ on the cosets of $W$ are by similarities.   
       By the discussion in 
           Section \ref{compatible},    for each $\gamma\in \Gamma$, 
            there is a graded   automorphism $\phi_\gamma$ of $W$   such that 
                $ (L_{\gamma(g)^{-1}}\circ \gamma \circ L_g)|_W=\phi_\gamma$ for any $g\in N$.  The element $\gamma$ induces a 
                 biLipschitz map $\bar \gamma:  N/W=\mathbb R\ra N/W$.

                 {\bf{Claim}}   $\bar \gamma$ has  constant derivative  a.e.    and so   is a  similarity. 
                 
                  Suppose the contrary  holds. Then there are two points $\bar g_1\not=\bar g_2\in N/W$ such that $\bar \gamma$ is differentiable at both $\bar g_1$ and $\bar g_2$ and  $\bar b_1\not=\bar b_2$, where 
                   $\bar b_1$, $\bar b_2$    are the derivatives   of $\bar \gamma$ at $\bar g_1$ and $\bar g_2$ respectively.  Let $\phi_i: N\ra N$ be a  blow-up of $\gamma$ at $g_i$; that is,   there is a sequence $t_j\ra \infty$ such that the sequence of maps
                    $e^{t_jD}\circ L_{(\gamma(g_i))^{-1}}\circ \gamma\circ L_{g_i}\circ e^{-t_jD}: N\ra N$  converges uniformly on compact subsets to $\phi_i$.  Such a sequence  $\{t_j\} $   exists by Arzela-Ascoli as $\gamma$ is biLipschitz.   Different choices of the sequences $\{t_j\}$ may yield different limits $\phi_i$. But all we need is one such limit. 
                  The  $\phi_i$ still acts on the cosets of $W$ by $\phi_\gamma$, but now has the additional property that its induced action on $N/W=\mathbb R$ is      multiplication by  $\bar b_i$ (as $\bar \gamma$ is differentiable at $\bar g_i$ with derivative $\bar b_i$).  Let $X\in V_\alpha\backslash W_\alpha$.     The   Case II  assumption implies that $[X, W_1]\not=0$  and so $\alpha$ must be   an  integer  as $[X, W_1]\subset W\cap V_{1+\alpha}$.   
                   By  Lemma \ref{sj}, there is an injective linear map
                  $B_i: \mathbb  R X\ra V_\alpha$  ($B_i$ also depends on $\gamma$) such that $B_i$ induces the linear map 
                   $N/W=\mathbb R\ra N/W, t\mapsto \bar b_i t$ and 
                   $d\phi_\gamma [X, w]=[B_i X, (d\phi_\gamma) w]$   for all $w\in \mathfrak w$.  
                    It follows that $[(B_1-B_2)X, (d\phi_\gamma) w]=0$ for all $w\in \mathfrak w$.   
                     As $d\phi_\gamma$ is an automorphism of $\mathfrak w$, we have  $[(B_1-B_2)X, \mathfrak w]=0$.    
                    Notice that  $(B_1-B_2)X= (\bar b_1-\bar b_2) X+Y$ for some $Y\in W_\alpha$.   
                    Since $\bar b_1-\bar b_2\not=0$, $X\notin \mathfrak w$   and $\mathfrak w$ has codimension one in $\mathfrak n$,    we see that 
                      $\mathfrak n$ is spanned by $\mathfrak w$ and $    (B_1-B_2)X$.   
                         Let $H'$ be the subspace  of $\mathfrak n$   spanned by $(B_1-B_2)X$. Then we have that $\mathfrak n=\mathfrak w\oplus H'$ is a direct sum of two ideals, contradicting our assumption. This  completes Case II.   
                         
                          We have shown that $\Gamma$ is a fiber similarity group, after possibly replacing $\Gamma$ with a biLipschitz conjugate. 
                             Now the arguments in Sections \ref{compatible}--\ref{cstructure}   
                                can be applied to show that $\Gamma$   can be 
                                conjugated into  
       $\text{Sim}(N, d_0)$ where $d_0$ is a fixed  maximally symmetric $D$-homogeneous distance on $N$.

                         Now we have  finished the proof of Theorem \ref{main-uniform} in the case when $\dim(N/W)=1$.  Combining this with Section \ref{cstructure}
                             and  Section \ref{dim(w)=1}   we have   completed  the  proof of Theorem \ref{main-uniform}.

       One can use semi-direct product (see  Appendix  \ref{nilexample})          to  construct  many  examples  of   Case II.  For instance, let 
             $\mathfrak n=\mathfrak e\rtimes \mathbb R$, where $\mathfrak e$ is the    Engel  algebra (with basis $e_0, e_1, e_2, e_3$ and only non-trivial brackets $[e_0, e_j]=e_{j+1}$, $j=1, 2$)   and $\mathbb R=\mathbb R X$ acts on $\mathfrak e$ by
      $[X, e_1]=e_3$ (all other brackets are $0$).

      Finally we prove a Tukia-type theorem for the product of finitely  many Carnot groups (with possibly different scalings on the factors).
      
      \begin{theorem}\label{product of Carnot}
        For $1\le i\le m$, let $\alpha_i\ge 1$  be  a  constant,   $N_i$ a 
       Carnot group and   $d_i$  a maximally symmetric Carnot metric on $N_i$.  
       Set  $N=\prod_{i=1}^m N_i$  and let $d$ be the distance on $N$ given by $d((x_i), (y_i))=\sum_i d_i^{\frac{1}{\alpha_i}}(x_i, y_i)$.  
        Let $\Gamma$ be   a    
           uniform quasisimilarity group of  $(N,d)$.   
        If the action of $\Gamma$ on the space of distinct pairs of   $N$
        is cocompact, then 
         $\Gamma$ can be  conjugated  by a  biLipschitz  map  into the similarity group of  $(N,d)$.  
        If  $\dim(N_i)=1$ for some $i$ we further assume $\Gamma$ is amenable.

      \end{theorem}
      
      \begin{proof}    By combining those $N_i$ where the corresponding $\alpha_i$ are equal we may assume the $\alpha_i$'s are all distinct. 
     After reordering  we may  further assume $\alpha_1<\alpha_{2}<\cdots < \alpha_m$.   Set  $M_j=\prod_{i=1}^j N_i$.  
       Then $\{M_j\}_j$ form the preserved subgroup sequence
        of $N$.    By  Theorem \ref{foliatedtheorem}   
             we may assume  $\Gamma$ acts on the fibers $N_j=M_j/{M_{j-1}}$ 
        by similarities  for all $j$ satisfying $\dim (N_j)\ge 2$.     Denote $H_j=\prod_{i=m-j}^m N_i$ and let $\tilde d_j$ be   the distance on $H_j$ given by 
         $\tilde d_j((x_i), (y_i))=\sum_{i=m-j}^m d_i^{\frac{1}{\alpha_i}}(x_i, y_i)$.            We next prove by induction on $j$,  starting from $j=0$,   that after a biLipschitz conjugation  of $\Gamma$  the induced action of $\Gamma$  on   $(H_j, \tilde d_j)$ 
         is by similarities. 
         For the base case we let $j=0$. If $\dim (N_m)\ge 2$, then   the induced action on 
           $(H_0, \tilde d_0)=(N_m,  d_m^{\frac{1}{\alpha_m}})$  is already by similarities.  If $\dim (N_m)=1$, then  the argument in Case I of this section applies. 

       For the induction step, assume $j\ge 1$ and 
         that the induced action on    $(H_{j-1}, \tilde d_{j-1})$ 
         is by similarities. We need to show that, after a further biLipschitz conjugation, the induced action on  $(H_j, \tilde d_j)$    
         is by similarities.   If $\dim(N_{m-j})=1$, then 
          Theorem \ref{rtimesY}   
           implies there is a biLipschitz map $f_0$ of   $(H_j, \tilde d_j)$ 
        which conjugates the induced action on $(H_j, \tilde d_j)$ 
         to an action by similarities.  Let $F_0$ be the biLipschitz map of $N$ which equals $f_0$ on $H_j$  and is the identity map on  $N_i$, $1\le i<m-j$. Then $F_0\Gamma F_0^{-1}$ induces a similarity action on   $(H_j, \tilde d_j)$. 

         Next we assume $\dim(N_{m-j})\ge 2$.   
         By the induction hypothesis the induced action on $(H_{j-1},  \tilde d_{j-1})$ is by similarities. By the first paragraph, 
         the action on  the fibers $N_{m-j}=M_{m-j}/{M_{m-j-1}}$  is 
        by similarities.   
         It follows that the induced action of each $\gamma$  on  $H_j$   is given by:
           $\gamma(x_{n-j}, y)=\gamma(0) (A_\gamma(x_{n-j} s_\gamma(y)), B_\gamma  y)$, where  $x_{m-j}\in N_{m-j}$,  $y\in H_{j-1}$,  $A_\gamma$ is an automorphism of $N_{m-j}$ and is also a similarity with respect to 
           $d_{m-j}$,  $B_\gamma$ is a similarity of    $(H_{j-1}, \tilde d_{j-1})$   
            and 
           $s_\gamma:   (H_{j-1}, \tilde d_{j-1}) \ra (Z(N_{m-j}),  d_{m-j}^{\frac{1}{\alpha_{m-j}}})$ is  Lipschitz  and satisfies $s_\gamma(0)=0$.  
             We  identify the Lie group with its Lie algebra via the exponential map   and 
            write $s_\gamma=\sum_i s_{\gamma,i}$, where 
            $s_{\gamma,i}:  H_{j-1}\ra Z_i(\mathfrak n_{m-j})=Z(\mathfrak n_{m-j})\cap \mathfrak n_{m-j, i} $ is the $i$-layer component of $s_\gamma$  and $\mathfrak n_{m-j}=\oplus_i \mathfrak n_{m-j, i} $ is the Carnot grading of $\mathfrak n_{m-j}$.     For $i\ge 1$, let
             $$E_i=\{c:     (H_{j-1}, \tilde d_{j-1})  \ra (Z_i(\mathfrak n_{m-j}), |\cdot|^{\frac{1}{i\alpha_{m-j}}})\;\;\text{is}\;\; \text{Lipschitz  and }\;\; c(0)=0\},$$
              where $|\cdot|$ is a fixed  Euclidean metric  on  $Z_i(\mathfrak n_{m-j})$.    It is easy to see that  $s_{\gamma,i}\in E_i$ for all $\gamma\in \Gamma$. Furthermore, 
                  if $s=\sum_i s_i$ with $s_i:   H_{j-1}\ra Z_i(\mathfrak n_{m-j}) $, then the shear map    $F_s:  H_j\ra H_j $, $F_s(x_{n-j}, y)=(x_{n-j}s(y), y)$  is  biLipschitz  if and only if $s_i\in E_i$.  
           Now applying the arguments from Section \ref{eliminate}  
              we can find a biLipschitz  shear map $f_0$ of  $(H_{j}, \tilde d_{j})$   that  conjugates  the induced action of $\Gamma$ on    $(H_j, \tilde d_j)$   
           to an action  by similarities.  Letting  $F_0$ be the biLipschitz map of $N$ that agrees with  $f_0$ on 
          $H_j$  and is the identity on  $N_i$  with $1\le i< m-j$ we see that $F_0\Gamma F_0^{-1}$    induces  a similarity action on    $(H_j, \tilde d_j)$

      \end{proof}



             \appendix

             \section{Examples  of Carnot-by-Carnot  groups} \label{nilexample}
         

In this appendix we shall give examples of    Carnot-by-Carnot  groups.   
 The upshot is that such groups are abundant.  The groups are Carnot group by Carnot  group extensions. These correspond to Lie algebra extensions of the type Carnot algebra by Carnot algebra.
 The classical connection between group extension and group cohomology  has a counterpart for Lie algebras: there is a connection between Lie algebra extension and Lie algebra cohomology, see  \cite{S66}, \cite{HS}.  We briefly recall this connection here.   Any extension  $0\rightarrow \mathfrak w\rightarrow \mathfrak n \rightarrow \mathfrak h\rightarrow 0$   induces a Lie algebra homomorphism
 $\bar\alpha: \mathfrak h \rightarrow  \text{out}(\mathfrak w):=\text{der}(\mathfrak w)/\text{ad}(\mathfrak w)$. Conversely, any  Lie algebra homomorphism
 $\bar\alpha: \mathfrak h \rightarrow  \text{out}(\mathfrak w)$  induces a  Lie algebra homomorphism $\beta: \mathfrak h\rightarrow \text{der}(\mathcal Z(\mathfrak w))$,  and there exists an  extension  $0\rightarrow \mathfrak w\rightarrow \mathfrak n \rightarrow \mathfrak h\rightarrow 0$ inducing the given $\bar\alpha$ if and only if a particular cohomology class in $H^3(\mathfrak h, \mathcal Z(\mathfrak w))$ vanishes, where the $\mathfrak h$ module structure on $\mathcal Z(\mathfrak w)$ is given by $\beta$.   When an extension as above does exist,  the set of  equivalence classes of extensions of $\mathfrak w$ by $\mathfrak h$ inducing the given $\bar \alpha$ is  in one-to-one correspondence with the elements of  $H^2(\mathfrak h, \mathcal Z(\mathfrak w))$.
    It follows that for every example of extension 
 $0\rightarrow \mathfrak w\rightarrow \mathfrak n \rightarrow \mathfrak h\rightarrow 0$, we get a collection of extensions parametrized by $H^2(\mathfrak h, \mathcal Z(\mathfrak w))$.   If one starts with a split extension (which corresponds to semi-direct product) and     the second cohomology is non-trivial, then one gets extensions that are no longer split. This applies to the semi-direct product examples below.

We next describe some examples.  Let $\mathfrak w=W_1\oplus \cdots \oplus W_m$ and $\mathfrak h=H_1\oplus \cdots H_n$ be two Carnot algebras.

\vspace{3mm}

\noindent
{\bf{Trivial extension.}}     This is the direct sum $\mathfrak w\oplus \mathfrak h$.

\vspace{3mm}

\noindent
{\bf{Semi-direct product.}} 
Recall that each Lie algebra homomorphism $\mathfrak h\rightarrow \text{der}(\mathfrak w)$ determines a semi-direct product $\mathfrak w\rtimes \mathfrak h$.  
   The trivial homomorphism yields the direct sum $\mathfrak w\oplus \mathfrak h$.
    We shall construct  nontrivial 
 Lie algebra homomorphisms $\mathfrak h\rightarrow \text{der}(\mathfrak w)$.   
We first recall a result about free nilpotent  Lie algebras, see \cite{S71}. 

\begin{prop}\label{freeuniversal}
  Let $F=F_1\oplus \cdots \oplus  F_t$ be a $t$-step free nilpotent Lie algebra, and $L: F_1\rightarrow F$ be any linear map.
 Then\newline
 (1) $L$ extends  uniquely  to a  derivation $d: F\rightarrow F$;\newline
 (2) $L$ extends  uniquely to a   Lie algebra homomorphism $\phi: F\rightarrow F$.  
\end{prop}

 Assume $\mathfrak h=F/{I}$ and $\mathfrak w=\tilde F/{\tilde I}$,   where $F=F_1\oplus \cdots \oplus F_t$ and 
$\tilde F=\tilde F_1\oplus \cdots \oplus \tilde F_{\tilde t}$  are
   free nilpotent Lie algebras,  and $I\subset F$,   $\tilde I\subset \tilde F$   are  graded ideals  
satisfying  $I\subset F_{k+1}\oplus \cdots \oplus F_t$, $\tilde I\subset \tilde F_{\tilde k+1}\oplus \cdots \oplus \tilde F_{\tilde t}$ for some 
    positive integers $k, \tilde k $.  
  For each integer $s\ge 0$, let $\text{der}_s(\mathfrak w)$ be the linear subspace of $\text{der}(\mathfrak w)$ defined by:
 $$\text{der}_s(\mathfrak w)=\{\rho: \mathfrak w\rightarrow \mathfrak w \;\;\text{is a derivation satisfying }\;\;   \rho(W_i)\subset W_{i+s}, \forall i\}.$$

\begin{Le}
Let   $\alpha\ge 2$ be an integer and 
$L: H_1\rightarrow \text{der}_\alpha(\mathfrak w)$ be a linear map. If $m\le (k+1)\alpha$, then $L$ extends uniquely to a Lie algebra homomorphism $\mathfrak h\rightarrow \text{der}(\mathfrak w)$.

\end{Le}

\begin{proof}
Let $L_\alpha$ be the Lie subalgebra of $\text{der}(\mathfrak w)$ generated by 
 $\text{der}_\alpha(\mathfrak w)$.  Since $\mathfrak w$ is $m$-step, the assumption 
$m\le (k+1)\alpha$  implies that $L_\alpha$ is a nilpotent Lie algebra with step at most $k$.  
As  $\mathfrak h=F/{I}$  with  $I\subset F^{(k+1)}$, we may identity $H_1$ with the first layer $F_1$  of $F$. Now $F$ is free nilpotent with step at least $k$ and $L_\alpha$ has step at most $k$.  By the universal property of free nilpotent Lie algebra the linear map
 $L: H_1=F_1\rightarrow L_\alpha$ extends uniquely to a Lie algebra homomorphism
 $\rho: F\rightarrow L_\alpha$.  Since $L_\alpha$   has step at most $k$, we have $\rho(F^{(k+1)})=0$.   As     $\mathfrak h=F/{I}$  with  $I\subset F^{(k+1)}$, 
 $\rho$ induces a Lie algebra homomorphism $\mathfrak h\rightarrow L_\alpha\subset \text{der}(\mathfrak w)$ that extends $L$.

\end{proof}

\begin{Le}
Let $L(W_1, W_{\alpha+1})$ be the vector space of linear maps from $W_1$ to $W_{\alpha+1}$, and
 $\Phi:  \text{der}_\alpha(\mathfrak w) \rightarrow L(W_1, W_{\alpha+1})$ be the restriction map,  $\Phi(\rho)=\rho|_{W_1}$. If $m\le \alpha+\tilde k$, then $\Phi$ is a linear  isomorphism.

\end{Le}

\begin{proof}
Since $\mathfrak w$ is generated by $W_1$, we see that $\Phi$ is injective.  
 We shall establish surjectivity of $\Phi$ by showing that each map in $L(W_1, W_{\alpha+1})$ extends to a derivation, which necessarily lies in $\text{der}_\alpha(\mathfrak w)$.   Let $L\in L(W_1, W_{\alpha+1})$.   As $\mathfrak w=\tilde F/{\tilde I}$  with  $\tilde I\subset \tilde F^{(\tilde k+1)}$, we may identity $W_1$ with $\tilde F_1$  and
 view  $L$ as a linear map from $\tilde F_1$ to $W_{\alpha+1}$.  Since 
 $W_{\alpha+1}=\tilde F_{\alpha+1}/(\tilde F_{\alpha+1}\cap \tilde I)$, we can 
 lift $L$ to a map into $\tilde F_{\alpha+1}$, that is, there is a 
linear map $\tilde L: \tilde F_1\rightarrow  \tilde F_{\alpha+1}$ such that $\pi\circ \tilde L=L$, where 
$\pi: \tilde F\rightarrow  \mathfrak w$  is the projection. By Proposition \ref{freeuniversal}  the linear map $\tilde L$ extends to a 
   derivation $\tilde d: \tilde F\rightarrow \tilde F$.  Notice that 
    $\tilde d(\tilde F^{(\tilde k+1)})\subset \tilde F^{(\tilde k+1+\alpha)}$.  
 As $\mathfrak w$ is $m$-step, the assumption
  $m\le \alpha+\tilde k$ implies  that  $(\pi\circ \tilde d) (\tilde F^{(\tilde k+1)})=0$.  
As   $\mathfrak w=\tilde F/{\tilde I}$  with  $\tilde I\subset 
    \tilde F^{(\tilde k+1)}$,   it is easy to see that $\tilde d$ induces a derivation
   $d: \mathfrak w\rightarrow \mathfrak w$  that extends $L$.


\end{proof}

The following Corollary provides many examples of nontrivial semi-direct products
of  Carnot algebras and so     nontrivial semi-direct products of   Carnot groups.

\begin{corollary}\label{semi-example}
Suppose $m\le \min\{\tilde k+\alpha, (k+1)\alpha\}$. Then for  any linear map  $f: H_1\rightarrow  L(W_1, W_{\alpha+1})$, the map   $\Phi^{-1}\circ f$ extends uniquely to a Lie algebra homomorphism
 $\mathfrak h\rightarrow \text{der}(\mathfrak w)$.  
 In particular,  for any  nonzero linear map  $f: H_1\rightarrow  L(W_1, W_{\alpha+1})$, there is a nontrivial 
 Lie algebra homomorphism  $\mathfrak h\rightarrow \text{der}(\mathfrak w)$ and so a nontrivial semidirect product $\mathfrak w\rtimes \mathfrak h$.

\end{corollary}

We mention a few special  cases of Corollary \ref{semi-example}. \newline

  (1) $\alpha=m-1$.  In this case,  the condition   $m\le \min\{\tilde k+\alpha, (k+1)\alpha\}$
 is automatically satisfied.  So every nonzero linear map $f: H_1\rightarrow L(W_1, W_m)$  will yield a nontrivial semidirect product $\mathfrak w\rtimes \mathfrak h$.    \newline
  (2)  $\alpha=m-2$, $m\ge 4$ and $\tilde k\ge 2$, In this case,   the condition  
 $m\le \min\{\tilde k+\alpha, (k+1)\alpha\}$
 is  satisfied.   And so  every nonzero linear map $f: H_1\rightarrow L(W_1, W_{m-1})$  will yield a nontrivial semidirect product $\mathfrak w\rtimes \mathfrak h$.


\vspace{3mm}

\noindent
{\bf{Central product.}}   Let $\alpha\ge 2$ be an integer and  $\mathfrak h=H_1\oplus \cdots\oplus  H_n$,    $\mathfrak w=W_1\oplus \cdots \oplus W_{\alpha n}$   be two Carnot algebras.  
Let $W'\subset W_{\alpha n}$ and $H'\subset H_n$ be linear subspaces and $\phi: W'\rightarrow H'$ a linear isomorphism.  The corresponding central product   $\mathfrak w\times_\phi \mathfrak h$   is the quotient of the direct sum 
 $\mathfrak w \oplus \mathfrak h$ by the central ideal $\{(w, -\phi(w))| w\in W'\}$.   
 Clearly there is a short exact sequence $0\rightarrow \mathfrak w\rightarrow \mathfrak w\times_\phi \mathfrak h\rightarrow \mathfrak h/{H'}\rightarrow 0$ and so $\mathfrak w\times_\phi \mathfrak h$ is an extension of $\mathfrak w$ by $\mathfrak h/{H'}$.   Notice that $\mathfrak h/{H'}=H_1\oplus \cdots \oplus (H_n/{H'})$ is a 
 Carnot algebra.

         \section{Lattices in SOL-like groups}\label{lattice}


                   In this appendix  we give examples of SOL-like groups that admit lattices.


           { There are necessary and sufficient conditions for the existence of lattices in solvable Lie groups, see  Theorem 6.2 in \cite{M70} and Chapter III, Section 6 of \cite{Au73}.
             But those conditions are not easy to check.
            Here we cite a result by Sawai--Yamada \cite{SY05} which  gives a simple sufficient condition for certain solvable Lie groups to admit lattices. }
           Let $\mathfrak n$ be  a   rational nilpotent Lie algebra. Here ``rational'' means   $\mathfrak n$ has  a basis with rational structure constants. Then it follows that $\mathfrak n$ has a basis $\{X_1,   \cdots, X_m\}$ with integer structure constants.
            Let  $\mathfrak n^{i}$  ($i=1, 2$)   be a copy of $\mathfrak n$ with corresponding basis  $\{X_1^i, \cdots, X_m^i\}$.
  So  $X_j\mapsto X_j^i$ extends to an isomorphism from $\mathfrak n$ to $\mathfrak n^i$.
      Suppose $k_j$, $1\le j\le m$ are integers and the map
      $X^1_j\mapsto k_j X^1_j$,  $X^2_j\mapsto -k_j X^2_j$,  extends to a derivation $D$ on
        $\mathfrak n^1\times \mathfrak n^2$.   Let $S$ be the semidirect product $(N^1\times N^2)\rtimes \mathbb R$,  where $N^i$ is the simply connected Lie group with Lie algebra $\mathfrak n^i$ and    the action of $\mathbb R$ on $N^1\times N^2$ is generated by the derivation $D$.

      \begin{theorem} (\cite{SY05},  Theorem 2) \label{SY}
      The solvable Lie group $S$ above admits a lattice.
      \end{theorem}

       We recall that   lattices in solvable Lie groups are always uniform \cite{M62}.

       We give two explicit examples.  The first is the so-called Benson--Gordon group \cite{BG}, which was also discussed in \cite{SY05}.
          In this example, $\mathfrak n^i$  ($i=1, 2$)  is a copy of the Heisenberg algebra with basis $X^i, Y^i, Z^i$
           and the only nontrivial bracket among basis elements  is $[X^i, Y^i]=Z^i$.   Let $k_1, k_2$ be integers.     Let $D_i:  \mathfrak n^i\ra \mathfrak n^i$ be the derivation given by   $D_i(X^i)=k_1X^i$,   $D_i(Y^i)=k_2Y^i$,
            $D_i(Z^i)=(k_1+k_2)Z^i$.
            Let $D=(D_1, -D_2)$ be the derivation of $\mathfrak n^1\times \mathfrak n^2$.
              The  semidirect product $S_{k_1, k_2}=(N^1\times N^2)\rtimes_D \mathbb R$ is a Benson-Gordon group.
               When $k_1, k_2$ are positive,    $S_{k_1, k_2}$   is a SOL-like  group that we are interested.
                  When $k_1=k_2$ is positive,
                  $(N^i, D_i)$  is of Carnot type.    Next we give an example  where $N$ is  Carnot-by-Carnot. 

               In the second example  $\mathfrak n$  is a semi-direct product $\mathfrak n=\mathfrak e\rtimes \mathfrak h$, where $\mathfrak e$ is the  Engel  algebra (with basis $e_0, e_1, e_2, e_3$  and only non-trivial brackets $[e_0, e_i]=e_{i+1}$, $i=1,2 $)   and $\mathfrak h$ is the Heisenberg algebra (with basis $X, Y,   Z$ and only non-trivial bracket $[X,Y]=Z$),  and the action of $\mathfrak h$ on $\mathfrak e$ is given by $[X, e_0]=e_3$, $[Y, e_1]=e_3$ (all other brackets are $0$).     This
                 semi-direct product    is of  Case (1) discussed after Corollary \ref{semi-example}.
            Let $D_0$ be the  derivation of $\mathfrak n$ given by $D_0(e_0)=e_0$, $D_0(e_j)=je_j$ ($j=1, 2,3$),  $D_0(X)=2X$, $D_0(Y)=2Y$,
             $D_0(Z)=4Z$.     Set $D_1=D_2=D_0$.
             Then $D=(D_1, -D_2)$ is a derivation of $\mathfrak n^1\times \mathfrak n^2$.
              Finally let $S=(N^1\times N^2)\rtimes_D \mathbb R$. 
            By
               Theorem   \ref{SY}   $S$ admits a    lattice.     In this example, $(N_i, D_i)$ is Carnot-by-Carnot.


      \section{Compatible expressions}\label{proof of compatible}

      In this appendix we prove Lemma \ref{sj}.  
      Recall the assumption: $(\mathfrak n, D)$ is a diagonal  Heintze pair, $\mathfrak w$ is an ideal of $\mathfrak n$ such that $D(\mathfrak w)=\mathfrak w$; $F:N\ra N$ is a biLipschitz map that permutes the cosets of $W$, where $W$ is the connected Lie subgroup of $N$ with Lie algebra $\mathfrak w$; for each $g\in N$, the map 
      $F_g|_W $ is an automorphism  $\phi$  of $W$, and $F$ induces an affine map  $\bar F$ of $N/W$.     Let $\bar B$ be the automorphism part of $\bar F$.


      The main ingredient in the proof is the fact that
                  $\phi\circ (\chi_g|_W)=(\chi_{G(g)}|_W)\circ \phi$, where $G=F_0$.  See Lemma \ref{onfiber}.   In the following proof  we will implicitly (and repeatedly) use the fact that  
                  $[Z(\mathfrak w), \mathfrak n]\subset Z(\mathfrak w)$.   This  follows from  the Jacobi identity  and the fact that $W$ is an ideal of $\mathfrak n$. 
               
               For a linear transformation  $T$ of a finite dimensional vector space, we denote by 
               $\sigma(T)$      the set of eigenvalues of $T$.   Recall  that 
                $\sigma(\bar D)\subset \sigma(D)$, where $\bar D: \mathfrak n/{\mathfrak w} \ra \mathfrak n/{\mathfrak w}$ is the derivation induced by $D$. Set $I=\sigma(D)-\sigma(\bar D)$.

      \begin{proof}
      Set $G=F_0=L_{F(0)^{-1}}\circ F$.   Then $G(0)=0$. By Lemma \ref{onfiber},  there is an automorphism $\phi$ of $W$ such that  if we set $A=d\phi$,   then 
        $A$ is layer-preserving,  
      $A\circ d(\chi_g|_W)\circ A^{-1}=d(\chi_{G(g)}|_W)$  and  $G(h*w)=G(h)*A w$ for 
                        any $g=h*w\in \mathfrak n$, where $h\in H, w\in \mathfrak w$.
                  Let $H'\subset \mathfrak n$ be a graded subspace of $\mathfrak n$ complementary to $\mathfrak w$; that is, for each $\lambda\in \sigma(\bar D)$, 
                   $H'_\lambda\subset V_\lambda$ is a complementary linear subspace of 
                    $W_\lambda$  in  $ V_\lambda$   (if $W_\lambda=\{0\}$, then  $H'_\lambda= V_\lambda$), and $H'=\oplus_{\lambda\in \sigma(\bar D)} H'_{\lambda}$.  
                       Denote by $B_0: H\ra  H'$ the linear isomorphism  $(\pi|_{ H'})^{-1}\circ d\bar B\circ (\pi|_{H})$.
                      Since $\pi(G(h))   
                      =d\overline B (\bar h)=\pi(B_0 h)$, there is a 
                        map  $S: \mathfrak n/\mathfrak w\ra \mathfrak w$ such that $G(h)=B_0h*S(\bar h)$.
                        It follows that the map $G$ has the form
                        $$G(h*w)=B_0h*S(\bar h)*A w.$$
                        In general,  $B_0$ and $S$ do not satisfy  conditions (2) and (3) in the definition of compatible expression. We need to modify the map $B_0$ (and so also $S$).     
                  
                      For each $\lambda\in \sigma(D|_{\mathfrak w})$,   denote
                  $Z_\lambda(\mathfrak w):=Z(\mathfrak w)\cap W_\lambda$        and let $Z_\lambda^\perp (\mathfrak w)\subset W_\lambda$ be a subspace  such that
                   $W_\lambda= Z_\lambda(\mathfrak w)\oplus Z_\lambda^\perp (\mathfrak w)$. 
                  The map $S$ can be written as $S(\bar h)=\sum_{\lambda\in \sigma(D|_{\mathfrak w})} S_\lambda(\bar h)$, where $S_\lambda:=\pi_\lambda\circ S$  and $\pi_\lambda: \mathfrak w \ra W_\lambda$ is the projection with respect to  the decomposition   $\mathfrak w=\oplus_\lambda W_\lambda$. 
                  There are two maps $\tilde S_\lambda: \mathfrak n/\mathfrak w\ra Z_\lambda(\mathfrak w)$ and $S_\lambda^\perp: \mathfrak n/\mathfrak w\ra
                   Z_\lambda^\perp (\mathfrak w)$  such that 
                  $S_\lambda(\bar h)=\tilde S_\lambda(\bar h)+S_\lambda^\perp(\bar h)$.

                  We shall use the equation  
                  $A\circ d(\chi_g|_W)\circ A^{-1}=d(\chi_{G(g)}|_W)$.
                    Let $\mu\in \sigma(D|_{\mathfrak w})$,  $w\in W_\mu$ and $h\in H$ be arbitrary.
                   We next  calculate  $A\circ d(\chi_h|_W)\circ A^{-1}(w) $    and $ d(\chi_{G(h)}|_W)(w)$.  
                  By  (\ref{conjugationformula}) 
                  $$A\circ d(\chi_h|_W)\circ A^{-1}(w)=w+A[h,A^{-1}w]+\sum_{i=2}^\infty\frac{1}{i!}A(\text{ad}\,h)^i(A^{-1}w)$$
                  and 
                    $$d(\chi_{G(h)}|_W)(w)=w+[G(h), w]+
                    \sum_{i=2}^\infty\frac{1}{i!}(\text{ad(G(h))})^i(w).$$
                      From  $A\circ d(\chi_g|_W)\circ A^{-1}=d(\chi_{G(g)}|_W)$,  
                     we get
                     \begin{align}\label{equality5.2}
L:=A[h, A^{-1}w]+\sum_{i=2}^\infty\frac{1}{i!}A(\text{ad}\,h)^i(A^{-1}w)=[G(h),  w]+
                    \sum_{i=2}^\infty\frac{1}{i!}(\text{ad(G(h))})^i(w):=R.
                    \end{align}
                     Since $\mathfrak w$ is an ideal in $\mathfrak n$, every item in 
                     (\ref{equality5.2}) lies in $\mathfrak w$. 
                     By the BCH formula, 
                      \begin{align}\label{F(h)}
                      G(h)=B_0h* S(\bar h)=B_0h+S(\bar h)+\frac{1}{2}[B_0h, S(\bar h)]+\cdots.
                      \end{align}

                   {\bf{First Claim}}:   $S_\lambda^\perp=0$   for $\lambda\in I$; in other words, $S_\lambda(\bar h)\in Z_\lambda(\mathfrak w)$  for  $\lambda\in I$.

                   The proof is by induction.  Let $\lambda_0\in I$.  Assume that 
                  $S_\lambda^\perp=0$ for all  $\lambda\in I$  with $\lambda<\lambda_0$. 
                    To simplify notation, we use
                     $w^z$ to denote an element of $Z(\mathfrak w)$, $w^>$ to denote an element of $\oplus_{\lambda>\lambda_0}W_\lambda$,
                      use  $\bar x$, $\bar y$ and so on  to denote elements in $\sum_{\lambda\in\sigma(\bar D)}W_\lambda$,  
                      and use subscripts to denote different such  elements.  Using this we can write
                     $S(\bar h)=\bar x_1+w_1^z+S_{\lambda_0}(\bar h)+w^>_1$.    Then 
                      $[B_0h, S(\bar h)]=\bar x_2+ w^z_2+w^>_2$ and all the iterated brackets of  $B_0h$ and $S(\bar h)$ have this form.  Hence    we have 
                      $$G(h)=B_0h+\bar x_3+w^z_3+S_{\lambda_0}(\bar h)+w^>_3.$$
                       From this we get 
                        $[G(h), w]=[B_0h+\bar x_3, w]+[S_{\lambda_0}(\bar h), w]+[w_3^{>}, w]$  and for $i\ge 2$, $(\text{ad}\,G(h))^i(w)=y_1+y_2$, with $ y_1\in\oplus_{\lambda\in \sigma(\bar D)}W_{\lambda+\mu}$  and 
                        $y_2\in \oplus_{\lambda>\lambda_0+\mu}W_\lambda$.  
                          Since by assumption $\lambda_0\in I$ we see that 
                           $\pi_{\lambda_0+\mu}L=0$ and 
                        $\pi_{\lambda_0+\mu}R=[S_{\lambda_0}(\bar h), w]$.  It follows that 
                        $[S_{\lambda_0}(\bar h), w]=0$ for    any   $w\in W_\mu$ and any
                          $\mu\in   \sigma(D|_{\mathfrak w})$  and therefore    
                             $S_{\lambda_0}(\bar h)\in Z_{\lambda_0}(W)$.

                    {\bf{Second  Claim}}: 
              for each $\lambda\in \sigma(\bar D)$, there is a linear map $B_\lambda: H\ra \mathfrak n$ satisfying  \newline
             $(a)_\lambda$.   $d\bar B\circ \pi|_H=\pi\circ B_\lambda$  and $B_\lambda(H_{\lambda'})\subset V_{\lambda'}$ for any $\lambda'\in \sigma(\bar D)$;\newline
             $(b)_\lambda$.   $B_\lambda|_{H_{\lambda'}}=B_0|_{H_{\lambda'}}$ for $\lambda'>\lambda$;\newline
             $(c)_\lambda$.   $[B_\lambda h, A w]=A[h, w]$ for any $w\in \mathfrak w$, and $h\in H_{\lambda'}$ with $\lambda'\le \lambda$;\newline
             $(d)_\lambda$.    The map $G$ can be written 
             $G(h*w)=B_\lambda h*S^{(\lambda)}(\bar h)*A w$, where 
             $S^{(\lambda)}: \mathfrak n/\mathfrak w\ra \mathfrak w$ is a map  satisfying 
                 $S^{(\lambda)}_{\lambda'}(\bar h)\in Z_{\lambda'}(\mathfrak w)$  for 
                   $\lambda'\le \lambda$, 
                   where $S^{(\lambda)}_{\lambda'}=\pi_{\lambda'}\circ S^{(\lambda)}$.

                  The proof of  the Second Claim is also by induction. 
                  We  let $\lambda_0\in \sigma(\bar D)$ and assume that  $B_\lambda$    satisfying  $(a)_\lambda-(d)_\lambda$  are defined for all 
                   $\lambda<\lambda_0$.   Let $\lambda_1<\lambda_0$ be the largest $\lambda\in \sigma(\bar D)$ less than $\lambda_0$.  
                  We shall  first show  that ${S_{\lambda_0}^{(\lambda_1)}}^\perp(\bar h)$ depends only on the $H_{\lambda_0}$ component   $h_{\lambda_0}$   of $h$ and is linear in $h_{\lambda_0}$.   
                  
                   Let $h\in H$ and $w\in W_\mu$ for some $\mu\in \sigma(D|_{\mathfrak w})$.    
                    We will  apply  $\pi_{\lambda_0+\mu}$  to  
                    both sides of (\ref{equality5.2}).
                      Using the First Claim   and the induction hypothesis we can write
                       $S^{(\lambda_1)}(\bar h)=w_1^z+S^{(\lambda_1)}_{\lambda_0}(\bar h)+w_1^>$. From this we get    $[B_{\lambda_1}h,  S^{(\lambda_1)}(\bar h)]=w_2^z+w_2^>$;
                        $G(h)=B_{\lambda_1}h+w_3^z+S^{(\lambda_1)}_{\lambda_0}(\bar h)+w_3^>$; 
                     $[G(h), w]=[B_{\lambda_1}h,  w]+[S^{(\lambda_1)}_{\lambda_0}(\bar h), w]+[w_3^>,w]$;    and for $i\ge 2$, 
                      $(\text{ad} \,G(h))^i(w)=(\text{ad} B_{\lambda_1}h)^i(w)+y_3$  with $y_3\in  \oplus_{\lambda>\lambda_0+\mu}W_\lambda$.   
                      Hence we have       
       $$  \pi_{\lambda_0+\mu}L=A[h_{\lambda_0}, A^{-1} w]+
       \sum_{i\ge 2}\frac{1}{i!}\sum_{\lambda_{j_1}+\cdots +\lambda_{j_i}=\lambda_0}               
       A(\text{ad} h_{\lambda_{j_1}}\circ \cdots  \circ \text{ad}h_{\lambda_{j_i}}(A^{-1}w)),$$
                        
                    $$ \pi_{\lambda_0+\mu}R=[B_{\lambda_1} h_{\lambda_0}, w]+
                    [S^{(\lambda_1)}_{\lambda_0}(\bar h), w]+\sum_{i\ge 2}\frac{1}{i!}
                    \sum_{\lambda_{j_1}+\cdots +\lambda_{j_i}=\lambda_0}               
                    \text{ad} B_{\lambda_{1}}h_{j_1}\circ \cdots \circ \text{ad} B_{\lambda_{1}}h_{j_i}(w).$$

                  
                  On the other hand  $(c)_{\lambda_1}$   implies for $i\ge 2$ 
               $$  A(\text{ad} h_{\lambda_{j_1}}\circ \cdots  \circ \text{ad} h_{\lambda_{j_i}}(A^{-1}w))=
                \text{ad} B_{\lambda_{1}}h_{j_1}\circ \cdots \circ \text{ad} B_{\lambda_{1}}h_{j_i}(w). $$
                    It follows that 
                     \begin{align}\label{e5.5}
                    A[h_{\lambda_0}, A^{-1} w]=[B_{\lambda_1} h_{\lambda_0}, w]+
                    [S^{(\lambda_1)}_{\lambda_0}(\bar h), w].
                    \end{align}

                    
                    Since the two terms  $ A[h_{\lambda_0}, A^{-1} w] $ and  $[B_{\lambda_1} h_{\lambda_0},w]$  
              depend only on the $H_{\lambda_0}$ component    $h_{\lambda_0}$  
                of $h$  and   are linear in $h_{\lambda_0}$, we see that 
              ${S_{\lambda_0}^{(\lambda_1)}}^\perp(\bar h)={S_{\lambda_0}^{(\lambda_1)}}^\perp(\bar h_{\lambda_0})$ also depends only on   $h_{\lambda_0}$ 
              and is linear in $h_{\lambda_0}$. 

                    
               We define $B_{\lambda_0} $ as follows:        $B_{\lambda_0}|_{H_{\lambda}}=B_{\lambda_1}|_{H_\lambda}$ for $\lambda\not=\lambda_0$   and
                          $B_{\lambda_0} h=B_{\lambda_1}  h+{S_{\lambda_0}^{(\lambda_1)}}^\perp(\bar h)$ for $h\in H_{\lambda_0}$.  
                     We need to verify $(a)_{\lambda_0}-(d)_{\lambda_0}$.  The properties  $(a)_{\lambda_0}$   and 
                       $(b)_{\lambda_0}$   are easy to see and  $(c)_{\lambda_0}$   follows from   $(c)_{\lambda_1}$,   (\ref{e5.5})  and the definition of $B_{\lambda_0}$.  
                        For  $(d)_{\lambda_0}$:  use  $(d)_{\lambda_1}$  and 
                       write $G$ as
                      $$ G(h*w)=B_{\lambda_1}  h*S^{(\lambda_1)}(\bar h)*A w
                       =B_{\lambda_0}h*S^{(\lambda_0)}(\bar h)*A w,$$
                        where $S^{(\lambda_0)}(\bar h)=(-B_{\lambda_0}h)*B_{\lambda_1}h*S^{(\lambda_1)}(\bar h)$.   
                        We need to show   $S^{(\lambda_0)}_{\lambda'}(\bar h)\in Z_{\lambda'}(\mathfrak w)$  for $\lambda'\le \lambda_0$.  By the definition of $B_{\lambda_0}$  and the linearity of $B_{\lambda_0}$,  $B_{\lambda_1}$  
                          we obtain  
                         $B_{\lambda_0}h=B_{\lambda_1}h+{S_{\lambda_0}^{(\lambda_1)}}^\perp(\bar h_{\lambda_0})=B_{\lambda_1}h+{S_{\lambda_0}^{(\lambda_1)}}^\perp(\bar h)$  for any $h\in H$.   Using this   we get     
                   $$[-B_{\lambda_0}  h, B_{\lambda_1}  h]=-[{S_{\lambda_0}^{(\lambda_1)}}^\perp(\bar h), B_{\lambda_1}  h]\in \oplus_{\lambda\ge (\lambda_0+\alpha)}W_\lambda.$$
                  By the BCH  formula,     
                          $(-B_{\lambda_0}h)*B_{\lambda_1}h
                          =(-B_{\lambda_0}h)+B_{\lambda_1}h+\frac{1}{2}[-B_{\lambda_0}h, B_{\lambda_1}h]+\cdots
                          =-{S_{\lambda_0}^{(\lambda_1)}}^\perp(\bar h)+ w_4,$ 
                           with $w_4\in \oplus_{\lambda\ge (\lambda_0+\alpha)}W_\lambda$. 
                             By  $(d)_{\lambda_1}$, 
                               $[-{S_{\lambda_0}^{(\lambda_1)}}^\perp(\bar h),  S^{(\lambda_1)}(\bar h)]\in \oplus_{\lambda\ge 2\lambda_0}W_\lambda$.  
                  Finally,   
                   \begin{align*}
                            S^{(\lambda_0)}(\bar h)&=(-B_{\lambda_0}h)*B_{\lambda_1}h*S^{(\lambda_1)}(\bar h)\\
                            &=-{S_{\lambda_0}^{(\lambda_1)}}^\perp(\bar h)+ w_4+S^{(\lambda_1)}(\bar h)+
                            \frac{1}{2}[-{S_{\lambda_0}^{(\lambda_1)}}^\perp(\bar h)+ w_4,  S^{(\lambda_1)}(\bar h)]+\cdots\\
                            &=S^{(\lambda_1)}(\bar h)-{S_{\lambda_0}^{(\lambda_1)}}^\perp(\bar h)+w_5, 
                            \end{align*}
                             with $w_5\in \oplus_{\lambda\ge (\lambda_0+\alpha)}W_\lambda$.
                     From  this,   $(d)_{\lambda_1}$  and the First Claim   we see  that 
                     $S_\lambda^{(\lambda_0)}(\bar h)=S_\lambda^{(\lambda_1)}(\bar h)\in Z_\lambda(\mathfrak w)$  for $ \lambda<\lambda_0$,   
                       and 
                       $$S_{\lambda_0}^{(\lambda_0)}(\bar h)
                       =S_{\lambda_0}^{(\lambda_1)}(\bar h)-{S_{\lambda_0}^{(\lambda_1)}}^\perp(\bar h)=
                       \tilde S_{\lambda_0}^{(\lambda_1)}(\bar  h)\in Z_{\lambda_0 }(\mathfrak w).$$ 
                     This  verifies     $(d)_{\lambda_0}$   and completes
the induction argument  for the Second Claim.                   
                  

                  Denote by  $\bar \lambda$   the largest eigenvalue of $\bar D$.  
                  We set $B=B_{\bar \lambda}$   and $s=A^{-1}\circ S^{(\bar \lambda)}$.  We need to verify  conditions 1--3 in the definition of a compatible expression. 
                   Condition 1  follows from $(a)_{\bar \lambda}$ and  Condition 2  follows from $(c)_{\bar\lambda}$. 
                  From $(d)_{\bar \lambda}$  we  get an expression for $F$
                   (as $G=F_0=L_{F(0)^{-1}}\circ F$):
                   $$F(h*w)=F(0)*Bh*A s(\bar h)*A w.$$ 
                    By the First Claim and $(d)_{\bar \lambda}$  we see that $s(\bar h)\in Z(\mathfrak w)$. This allows us to switch $A s(\bar h)$ and $Aw$ to arrive at a compatible expression for $F$.

      \end{proof}






          \end{document}